\documentclass[10pt,a4paper]{amsart}

\usepackage[latin1]{inputenc}
\usepackage[english]{babel}
\usepackage[T1]{fontenc}
\usepackage{float} 	% H insere figura exatamente no local
\usepackage{hyperref}
\usepackage{subfigure}
\usepackage{amsmath}
\usepackage{setspace}
\usepackage{indentfirst}
\usepackage{multirow}
\usepackage{color} 

\usepackage[Sonny]{fncychap}

\usepackage{array}
\usepackage{enumerate}
\usepackage{multicol}
\usepackage{hyperref}
\usepackage{pdfpages}
\usepackage{amsfonts}
\usepackage{amsthm}
\usepackage{amssymb}
\usepackage{hhline}

\usepackage{graphicx}
\usepackage{xcolor}

\usepackage{hyphenat}
\hypersetup{colorlinks=false}
 \usepackage{hyperref}
 \hypersetup{
     colorlinks,%
     citecolor=black,%
     filecolor=black,%
     linkcolor=black,%
     urlcolor=black
 }
 
\newtheorem{teo}{Theorem}[section]
\newtheorem{lema}{Lemma}[section]

\newtheorem{obs}[teo]{Remark}
\newtheorem{defi}{Definition}[section]
\newtheorem{ex}[teo]{Example}
\newtheorem*{nota}{Notation}
\newtheorem{remark}[teo]{Remark}

\def\corank{\operatorname{corank}}
\def\rank{\operatorname{rank}}
\def\dim{\operatorname{dim}}

\def\cof{\operatorname{cof}}
\def\Jac{\operatorname{Jac}}

\def\z{\Omega}
\def\o{\omega}
\def\f{\varphi}
\def\x{\xi}

\def\a{\alpha}
\def\b{\beta}
\def\l{\ell}
\def\g{\gamma}
\def\L{\mathcal{L}}
\def\R{\mathbb{R}}
\def\m{\textsl{m}}
\def\M{\textbf{M}}
\def\S{\Sigma}
\def\d{\Delta}
\def\Rnm{\R^n\setminus\{\vec{0}\}}

\def\zr1{z_{|_{A_1(V)}}}

\newcommand{\ffrac}{\displaystyle\frac}

\begin{document}

\title{Morin singularities of coframes and frames}

\author{Camila M. \textsc{Ruiz}}% Author Name (\sc should NOT be used here)
%\address{Address of the organization to which the second author beongs}
%\email{mail\_address\_of\_the\_second\_author}
%
%\author{Nicolas \textsc{Dutertre}}% Author Name (\sc should NOT be used here)
%\address{Address of the organization to which the first author beongs}
%\email{mail\_address\_of\_the\_first\_author}
%
%\author{Nivaldo \textsc{Grulha}}% Author Name (\sc should NOT be used here)
%\address{Address of the organization to which the second author beongs}
%\email{mail\_address\_of\_the\_second\_author}

%\subjclass[2010]{Primary xxxxx; Secondary xxxx}% Subject code(s)

%\keywords{Theory of singularities, Morin singularities, n-vector fields, Euler's Characteristic}
\address{Universidade de São Paulo\\
Instituto de Ciências Matemática e de Computação \\
Avenida Trabalhador São-carlense, 400 - Centro\\
CEP: 13566-590 - São Carlos - SP, Brasil}
\email{cmruiz@icmc.usp.br}
\maketitle

\begin{abstract}
Inspired by the properties of an $n$-frame of gradients $(\nabla f_1, \ldots, \nabla f_n)$ of a Morin map $f:M\rightarrow\mathbb{R}^n$, with $\dim M\geq n$, we introduce the notion of Morin singularities in the context of singular $n$-coframes and singular $n$-frames. We also study the singularities of generic 1-forms associated to a Morin $n$-coframe, in order to generalize a result of T. Fukuda \cite[Theorem 1]{Fukuda}, which establishes a modulo 2 congruence between the Euler characteristic of a compact manifold $M$ and the Euler characteristics of the singular sets of a Morin map defined on $M$, to the case of Morin $n$-coframes and Morin $n$-frames.
\end{abstract}

%%%%%%%%%%%%%%%%%%%%%%%%%%%%%%%%%%%%%%%%%%%%%% Introdução %%%%%%%%%%%%%%%%%%%%%%%%%%%%%%%%%%%%%%%%%%%%%%%%%%%%%%%%%%%%%%%%%%%%%%%%%%%%
\section{Introduction}\label{Introduction}

%*******************************
%Morin mappings are maps that only admit Morin singularities. It is well known that singularities of Morin mappings are stable, and conversely, that stable map-germs with corank 1 are Morin singularities. In this way, as observed by K. Saji \cite{Saji1}, Morin singularities are fundamental and frequently arise as singularities of maps from one manifold to another. This subject has been studied by many authors in different contexts as \cite{Morin2}, \cite{Ando}, \cite{Fukuda}, \cite{Saeki}, \cite{SaekiSakuma} and more recently \cite{KalmarTerpai}, \cite{SzaboSzucsTerpai}, \cite{Szucs}, \cite{InaIshiKawaThang}, \cite{Dutertrefukui}, \cite{Saji1}, \cite{Saji2}, \cite{Saji3}, \cite{Ruiz1}.\\

{Morin maps are maps that only admit Morin singularities. It is well known that these singularities are stable, and conversely, that stable map-germs which have corank 1 are Morin singularities. Therefore, Morin singularities are fundamental and frequently arise as singularities of maps from one manifold to another, as observed by K. Saji in \cite{Saji1}. Morin singularities have been studied by many authors in different contexts as \cite{Morin2,Ando,Fukuda,Saeki,SaekiSakuma}, and more recently \cite{KalmarTerpai,SzaboSzucsTerpai,Szucs,InaIshiKawaThang,Dutertrefukui,Saji1,Saji2,Saji3,Ruiz1}. In particular, papers of J.M. Èlia\v{s}berg \cite{Eliasberg}, J.R. Quine \cite{Quine}, T. Fukuda \cite{Fukuda}, O. Saeki \cite{Saeki} and N. Dutertre and T. Fukui \cite{Dutertrefukui} investigate relations between the topology of a manifold and the topology of the critical locus of maps with Morin singularities.}

Let $f:M^m\rightarrow\R^n$ be a smooth Morin map defined on an $m$-dimensional {Riemannian manifold} $M$, with $m\geq n$. The singular points of $f=(f_1, \ldots, f_n)$ are the points $x\in M$, such that the rank of the derivative $df(x)$ is equal to $n-1$. Then, taking the gradient of each coordinate function $f_1, \ldots, f_n$, we obtain a singular $n$-frame $(\nabla f_1(x), \ldots, \nabla f_n(x))$ defined on $M$ whose singular locus $\S$ is given by $$\S=\{x\in M \, | \, \rank(\nabla f_1(x), \ldots, \nabla f_n(x))=n-1\}.$$ 

It is well known that the singular sets of $f$, $A_k(f)$ and $\overline{A_k(f)}$ ($k=1, \ldots, n-1$), are submanifolds of $M$ of dimension $n-k$, such that $\overline{A_k(f)}=\cup_{i\geq k}A_i(f)$ and
\begin{equation*}\rank df_{|_{\overline{A_k(f)}}}(x)=\left\{\begin{array}{ll}
n-k, &\text{ if } x\in A_k(f);\\
n-k-1, &\text{ if } x\in \overline{A_{k+1}(f)};
\end{array}\right.\end{equation*} (see \cite{Fukuda}, \cite{Morin2}, \cite{Saeki} for Morin singularities). This means that the intersection of the vector space spanned by $\nabla f_1(x), \ldots, \nabla f_n(x)$ with the normal space to $\overline{A_k(f)}$ at $x$ is a subspace of dimension: $$\dim(\langle\nabla f_1(x), \ldots, \nabla f_n(x)\rangle\cap N_x\overline{A_k(f)})=\left\{\begin{array}{ll}
k-1, &\text{ if } x\in A_k(f);\\
k, &\text{ if } x\in \overline{A_{k+1}(f)}.
\end{array}\right.$$ In particular, if $x\in A_k(f)$ then $\langle\nabla f_1(x), \ldots, \nabla f_n(x)\rangle\pitchfork N_x\overline{A_k(f)}.$
%Podemos notar ainda que se $x\in \overline{A_{k+1}(f)}$, então $$\dim(\langle\nabla f_1(x), \ldots, \nabla f_n(x)\rangle\cap N_x\overline{A_{k+1}(f)})=\left\{\begin{array}{ll}
%k, &\text{ se } x\in A_{k+1}(f);\\
%k+1, &\text{ se } x\in \overline{A_{k+2}(f)}.
%\end{array}\right.$$ 

{Furthermore, if $x\in \overline{A_{k+1}(f)}$} and $\{z_1(x), \ldots, z_{n-k-1}(x)\}$ is a basis of a vector space supplementary to $\langle\nabla f_1(x), \ldots, \nabla f_n(x)\rangle\cap {N_x\overline{A_{k}(f)}}$ in $\langle\nabla f_1(x), \ldots, \nabla f_n(x)\rangle$ then $$\dim(\langle z_1(x), \ldots, z_{n-k-1}(x)\rangle\cap N_x\overline{A_{k+1}(f)})=\left\{\begin{array}{ll}
0, &\text{ if } x\in A_{k+1}(f);\\
1, &\text{ if } x\in \overline{A_{k+2}(f)}.
\end{array}\right.$$

Based on properties of an $n$-frame of gradients $(\nabla f_1, \ldots, \nabla f_n)$ of a Morin map $f$, in this paper we introduce the notion of Morin singular points of type $A_k$ in the context of singular $n$-frames that are not necessarily gradients {(Definition \ref{def:nframe})} and $n$-coframes that are not necessarily differentials {(Definition \ref{def:ncoframe})}. To do this, in Section \ref{s1} we consider an $n$-coframe $\o=(\o_1, \ldots, \o_n)$ with corank 1 {(Definition \ref{corank1})} defined on a smooth $m$-dimensional manifold $M$, with $m\geq n$, and we proceed by induction on $k$, for $k=1,\ldots,n$, in order to define Morin singular sets $\S^k(\o)$ and $A_k(\o)$ {(Definitions \ref{defs1}, \ref{defiSk} and \ref{Morinsing})}. We will say that $\o$ is a Morin $n$-coframe {(Definition \ref{def:ncoframe})} if it admits only Morin singular points, that is, if each singular point $x\in M$ of $\o$ belongs to $A_k(\o)$, for some $k=1, \ldots, n$ {(see Remark \ref{singak})}. {In particular, we show that the Morin singular sets $A_k(\o)$ and $\S^k(\o)=\overline{A_k(\o)}$ ($k=1, \ldots, n$) are smooth submanifolds of $M$ of dimension $n-k$ (Lemmas \ref{dimsigmaum} and \ref{skdimension}), such that $\overline{A_k(\o)}=\cup_{i\geq k}A_i(f)$ (Remark \ref{akclosure}) and in Lemmas \ref{eqlocalSk} and \ref{eqloc3} we exhibit equations that define locally the singular sets $\S^{k}(\o)$.}

The definition of Morin singularities for $n$-coframes can be analogously adapted to $n$-frames as follows. When considering a smooth manifold $M$, differential 1-forms are naturally dual to vector fields, more specifically, if we fix a Riemannian metric on $M$ then there exists an isomorphism between the tangent and cotangent bundles of M, so that vector fields and 1-forms can be identified. To illustrate this notion, we give some examples {of Morin $n$-frames} in the end of Section \ref{s1}.
%, where $\pi_L:\R^n\rightarrow L$ is a orthogonal projection over a generic straight line $L\in\R P^{n-1}$

Let $L\in\R P^{n-1}$ be a straight line in $\R^n$ and let $\pi_L:\R^n\rightarrow L$ be the orthogonal projection to $L$. In \cite{Fukuda}, T. Fukuda applied Morse theory and well known properties of singular sets $A_k(f)$ of a Morin map $f:M\rightarrow\R^n$ to study the critical points of mappings $\pi_L\circ f:M\rightarrow L$ and their restrictions to the singular sets $\pi_L\circ f|_{A_k(f)}$ and $\pi_L\circ f|_{\overline{A_k(f)}}$. Similarly, in Sections \ref{s2} and \ref{s3} of this paper, we investigate the zeros of a generic 1-form $$\x(x)=\displaystyle\sum_{i=1}^{n}{a_i\o_i(x)}$$ associated to a Morin $n$-coframe $\o=(\o_1, \ldots, \o_n)$ and we verify that $\x$, $\x_{|_{A_k(\o)}}$ and $\x_{|_{\overline{A_k(\o)}}}$ have properties that are analogous to the properties of the generic orthogonal projections $\pi_L\circ f(x)$ associated to a Morin map $f=(f_1, \ldots, f_n)$ and of their restrictions. {More precisely, let $a=(a_1, \ldots, a_n)\in\R^n\setminus\{\vec{0}\}$ and let $\o=(\o_1, \ldots, \o_n)$ be a Morin $n$-coframe defined on a manifold $M$, in Section \ref{s2} we prove that if $p\in M$ is a zero of $\x(x)=\sum_{i=1}^{n}{a_i\o_i(x)}$ then $p\in\S^1(\o)$ and $p$ is a zero of $\x_{|_{\S^1(\o)}}$ (Lemma \ref{zeroszsobresigma1}). In Lemma \ref{lemazerosrestricoes}, we show that, for $k=0,\ldots,n-2$, if $p\in A_{k+1}(\o)$ then $p$ is a zero of $\x_{|_{\S^{k+1}(\o)}}$ if and only if $p$ is a zero of $\x_{|_{\S^{k}(\o)}}$. And, in Lemma \ref{ptosansaozeros} we verify that if $p\in A_n(\o)$ then $p$ is a zero of the restriction $\x_{|_{\S^{n-1}(\o)}}$. Let $Z(\x_{|_{\S^k(\o)}})$ be the zero set of the restriction of the 1-form $\x$ to $\S^k(\o)$, we also prove in Lemmas \ref{lemainterzeroscomsigma2} and \ref{ptscrticrestasigmak} that for almost every $a\in\R^n\setminus\{\vec{0}\}$, $Z(\x_{|_{\S^k(\o)}})\cap\S^{k+2}(\o)=\emptyset$, for $k=0, \ldots, n-2$.}

{In Section \ref{s3}, in Lemmas \ref{nondegeneratexi}, \ref{zerosnaodegdeak}, \ref{zerosnaodegdeaum} and \ref{nondegeneratexisk}, we prove that generically the 1-form $\x(x)$ and its restrictions $\x_{|_{\S^k(\o)}}$ and $\x_{|_{A_k(\o)}}$ admit only non-degenerate zeros. We also show that, for $k=0,\ldots, n-2$, if $p\in A_{k+1}(\o)$ is a zero of $\x_{|_{\S^{k+1}(\o)}}$then, for almost every $a\in\R^n\setminus\{\vec{0}\}$, $p$ is a non-degenerate zero of $\x_{|_{\S^{k+1}(\o)}}$ if and only if $p$ is a non-degenerate zero of $\x_{|_{\S^{k}(\o)}}$ (Lemmas \ref{nondegenerateequivalenceA1} and \ref{naodegkekmaisum}). Finally, in Lemma \ref{naodegan} we verify that, for almost every $a\in\R^n\setminus\{\vec{0}\}$, if $p\in A_n(\o)$ then $p$ is a non-degenerate zero of $\x_{|_{\S^{n-1}(\o)}}$.}

As a consequence of these results, we obtain a generalization of Fukuda's Theorem \cite[Theorem 1]{Fukuda} for the case of Morin $n$-coframes (Theorem \ref{fukudaparacampos}). { More precisely, we prove in Theorem \ref{fukudaparacampos} that if $\o=(\o_1, \ldots, \o_n)$ is a Morin $n$-coframe defined on an $m$-dimensional compact manifold $M$ then $$\chi(M)\equiv\displaystyle\sum_{k=1}^{n}{\chi(\overline{A_k(\o)})} \mod 2,$$ where $\chi(M)$ denotes the Euler characteristic of $M$.} We end the paper with this generalized theorem, whose proof uses the classical Poincaré-Hopf Theorem for 1-forms.  

The author would like to express her sincere gratitude to Nicolas Dutertre and Nivaldo de Góes Grulha Júnior for fruitful discussions and valuable comments that resulted in this work.
The author was supported by CNPq, "Conselho Nacional de Desenvolvimento Científico e Tecnológico", Brazil (grants 143479/2011-3 and 209531/2014-2).

%Finally, we prove this generalized theorem using the classical Poincaré-Hopf Theorem for 1-forms.

%%%%%%%%%%%%%%%%%%%%%%%%%%%%%%%%%%%%%%%%%%%%%%%%%%%%%%%%%%%%%%%%%%%%%%%%%%%%%%%%%%%%%%%%%%%%%%%%%%%%%%%%%%%%%%%%%%%%%%%%%%%%%%%%%%%%%%
\section{The Morin $n$-coframes}\label{s1}

Let $M$ be a smooth manifold of dimension $m$ and $\o=(\o_1,\ldots,\o_n)$ be a (singular) $n$-coframe, that is, a set of $n$ smooth 1-forms defined on $M$, with $m\geq n$: \begin{equation*}\begin{array}{llll}
\o:& M&\rightarrow &{T^{\ast}M}^n\\
&x&\mapsto&(x,\o_1(x),\cdots,\o_n(x))
\end{array}
\end{equation*} where ${T^{\ast}M}^n=\{(x,\f_1,\ldots,\f_n) \ | \ x\in M; \ \f_i\in T^{\ast}_xM, i=1,\ldots,n\}$ is the ``$n$-cotangent bundle'' of $M$. Note that $T^{\ast}M^n$ is a smooth manifold of dimension $m(n+1)$, because it is locally diffeomorphic to $U\times M_{m,n}(\R)$, %$T^{\ast}M^n\cong \mathcal{U}\times M_{m,n}(\R)$, 
where $U\subset\mathbb{\R}^m$ is an open set and $M_{m,n}(\R)$ denotes the space of matrices of dimension $m\times n$ with real coefficients.

\begin{lema}\label{lemaTMnum} Let $T^{\ast}M^{n,n-1}\subset T^{\ast}M^n$ be the subset defined by $$T^{\ast}M^{n,n-1}=\left\{(x,\f_1,\ldots,\f_n)\in T^{\ast}M^n \ | \ \rank(\f_1,\ldots,\f_n)=n-1 \right\}.$$ Then $T^{\ast}M^{n,n-1}$ is a submanifold of $T^{\ast}M^n$ of dimension $n(m+1)-1$.
\end{lema}

\begin{proof} Let $M_{m,n}^{n-1}(\R)$ be the submanifold of $M_{m,n}(\R)$ of codimension $m-n+1$ given by the matrices of rank $n-1$ in $M_{m,n}(\R)$, then $T^{\ast}M^{n,n-1}$ is locally diffeomorphic to $U\times M_{m,n}^{n-1}(\R)$, for some open subset $U\subset\mathbb{\R}^m$. Hence, $T^{\ast}M^{n,n-1}$ is a submanifold of $T^{\ast}M^n$ and $\dim(T^{\ast}M^{n,n-1})=n(m+1)-1$.
%Locally, $T^{\ast}M^{n,n-1}\cong\mathcal{U}\times M_{m,n}^{n-1}(\R)$, for some open $\mathcal{U}\subset\mathbb{\R}^m$ and $M_{m,n}^{n-1}(\R)$ denotes the submanifold of $M_{m,n}(\R)$ of codimension $m-n+1$ formada pelas matrizes de posto $n-1$. Logo, $\dim(T^{\ast}M^{n,n-1})=n(m+1)-1$.
\end{proof}

\begin{defi}\label{corank1}
We say that $\o=(\o_1,\ldots,\o_n)$ has $\corank 1$ if the following properties hold:
\begin{enumerate}[(a)]
\item $\o\pitchfork T^{\ast}M^{n,n-1}$ in $T^{\ast}M^n$;
\item $\o^{-1}(T^{\ast}M^{n,\leq n-2})=\emptyset$;
\end{enumerate}
where $T^{\ast}M^{n,\leq n-2}=\left\{(x,\f_1,\ldots,\f_n)\in T^{\ast}M^n \ | \ \rank(\f_1,\ldots,\f_n)\leq n-2 \right\}$.
\end{defi}

Note that by Definition \ref{corank1}, if an $n$-coframe $\o=(\o_1,\ldots,\o_n)$ has $\corank 1$ then, for each $x\in M$, $\rank(\o_1(x),\ldots,\o_n(x))$ is either equal to $n$ or equal to $n-1$.

\begin{defi}\label{defs1}
Let $\o=(\o_1,\ldots,\o_n)$ be an $n$-coframe with $\corank 1$. The singular set of $\o$, $\Sigma^1(\o)$, is the set of points $x\in M$ at which the rank is not maximal, that is $$\Sigma^1(\o)=\{x\in M \ | \ \rank(\o_1(x),\ldots,\o_n(x))=n-1\}.$$
\end{defi}

\begin{lema}\label{dimsigmaum}
If $\o$ is an $n$-coframe with $\corank 1$ then $\Sigma^1(\o)$ is either the empty set or a submanifold of $M$ of dimension $n-1$.
\end{lema} 
\begin{proof} Note that $\Sigma^1(\o)=\o^{-1}(T^{\ast}M^{n,n-1})$ and that $\o\pitchfork T^{\ast}M^{n,n-1}$. Thus, if $\Sigma^1(\o)\neq\emptyset$ then $\Sigma^1(\o)$ is a submanifold of $M$ of codimension $m-n+1$, that is, $\dim(\Sigma^1(\o))=n-1$.
\end{proof}

Let $\o=(\o_1,\ldots,\o_n):M\rightarrow T^{\ast}M^n$ be an $n$-coframe with $\corank 1$ defined on an $m$-dimensional smooth manifold $M$. Next, we will define the subsets $A_k(\o)$ and $\S^{k+1}(\o)$ of $M$, for $k=1,\ldots,n$. To do this we will proceed by induction on $k$ starting from the definition of the singular set $\S^1(\o)$ .\\

\begin{nota} Let us denote by $\S^0(\o)$ the manifold $M$ and by $N_x^{\ast}\S^0(\o)=\{0\}$ the set that contains only the null 1-form of $T_x^{\ast}M$. { Moreover, if $S\subset M$ is a smooth submanifold of $M$, let us denote by $N^{\ast}_xS$ the set $N^{\ast}_xS=\{\psi\in T_{x}^{\ast}M \, | \, \psi(T_xS)=0\}.$}
\end{nota}
We know that $\S^1(\o)=\{x\in\S^0(\o) \, | \, \rank(\o_1(x),\ldots,\o_n(x))=n-1\}$ and that $\dim(\S^1(\o))=n-1$. In particular, $$p\in \S^1(\o)\Rightarrow\dim(\langle\o_1(p),\ldots,\o_n(p)\rangle\cap N_p^{\ast}\S^0(\o))=0,$$ where $\langle\o_1(p),\ldots,\o_n(p)\rangle$ is the vector subspace of $T_p^{\ast}M$ spanned by the 1-forms $\o_1(p),\ldots,\o_n(p)$.

Let us suppose that $\S^i(\o)$ is defined for $i=1,\ldots, k-1$ so that $\S^i(\o)$ is a smooth submanifold of $M$ of dimension $n-i$, $\S^{i}(\o)\subset\S^{i-1}(\o)$ and, { for $i=2,\ldots, k-1$}, $$p\in \S^i(\o)\Leftrightarrow\dim(\langle\o_1(p),\ldots,\o_n(p)\rangle\cap N_p^{\ast}\S^{i-1}(\o))=i-1,$$ where  $\S^i(\o)$ is locally given by $$\mathcal{U}\cap\S^i(\o)=\{x\in\mathcal{U} \, | \, F_1(x)=\ldots= F_{m-n+i}(x)=0\}$$ {and $$\mathcal{U}\cap\S^{i-1}(\o)=\{x\in\mathcal{U} \, | \, F_1(x)=\ldots= F_{m-n+i-1}(x)=0\},$$} for some open neighborhood $\mathcal{U}\subset M$ and smooth functions $F_i:\mathcal{U}\rightarrow\R$ whose derivatives $dF_i(x)\in T_{x}^{\ast}M$ are linearly independent for each $x\in\S^i(\o)\cap\mathcal{U}$. Also, $N_{x}^{\ast}\S^{i-1}(\o)$ is the vector subspace of $T_{x}^{\ast}M$ spanned by these derivatives, that is, $N_{x}^{\ast}\S^{i-1}(\o)=\langle dF_1(x), \ldots, dF_{m-n+i-1}(x)\rangle$.

%We define $\S^k(\o)$ as follows:\\ 

We set $r=n-k+1$ and $(x,\f)=(x,\f_1,\ldots,\f_{r})$. In order to define $\S^k(\o)$ we first consider: $$T_{\S^{k-1}}^{\ast}M^{r}=\{(x,\f) \ | \ x\in\S^{k-1}(\o); \f_1,\ldots,\f_{r}\in T_x^{\ast}M\}$$ and 
\begin{equation*}{\renewcommand{\arraystretch}{1}\begin{array}{r}
N_{\S^{k-1}}^{\ast}M^{r}=\{(x,\f)\in T_{\S^{k-1}}^{\ast}M^{r} \ | \rank(\f_1,\ldots,\f_{r})=r,\\
\dim(\langle\f_1,\ldots,\f_{r}\rangle\cap N_x^{\ast}\S^{k-1}(\o))=1\}.\end{array}}
\end{equation*}

%The set

\begin{lema} $T_{\S^{k-1}}^{\ast}M^{r}$ is a smooth manifold of dimension $mr+r$.
\end{lema}

\begin{proof}
By the induction hypothesis, $\S^{k-1}(\o)$ is a smooth submanifold of $M$ of dimension $r$. Then, there exists an open subset $V\subset\R^{r}$ so that $T_{\S^{k-1}}^{\ast}M^{r}$ is locally diffeomorphic to $V\times M_{m,r}(\R)$. Thus, $T_{\S^{k-1}}^{\ast}M^{r}$ is a smooth manifold and $\dim(T_{\S^{k-1}}^{\ast}M^{r})=mr+r.$ \end{proof}

\begin{lema} \label{dimNSM} $N_{\S^{k-1}}^{\ast}M^{r}$ is a hypersurface of $T_{\S^{k-1}}^{\ast}M^{r}$, that is, a submanifold of dimension $mr+r-1$.
\end{lema}

\begin{proof} By the induction hypothesis, {for each $p\in\S^{k-1}(\o)$, there exist an open neighborhood $\mathcal{U}\subset M$ of $p$ and functions $F_1,\ldots,F_{m-r}:\mathcal{U}\rightarrow\R$ such that $$\mathcal{U}\cap\S^{k-1}(\o)=\{x\in\mathcal{U} \ | \ F_1(x)=\ldots=F_{m-r}(x)=0\}$$ with $\rank(dF_1(x),\ldots,dF_{m-r}(x))=m-r$, for each $x\in\S^{k-1}(\o)\cap\mathcal{U}.$}

%$\S^{k-1}(\o)$ is a smooth submanifold of $M$ of dimension $r$ and there exist functions $F_1,\ldots,F_{m-r}:\mathcal{U}\rightarrow\R$ defined on some open neighborhood $\mathcal{U}\subset M$ such that, locally, $$\S^{k-1}(\o)=\{x\in\mathcal{U} \ | \ F_1(x)=\ldots=F_{m-r}(x)=0\}$$ with $\rank(dF_1(x),\ldots,dF_{m-r}(x))=m-r$, for each $x\in\S^{k-1}(\o)\cap\mathcal{U}.$ Let $p\in\S^{k-1}(\o)\cap\mathcal{U}$. 

If $(p,\tilde{\f})\in N_{\S^{k-1}}^{\ast}M^{r}$ then $\rank(\tilde{\f}_1,\ldots,\tilde{\f}_{r})=r$ and $$\rank (\tilde{\f}_1,\ldots,\tilde{\f}_{r},dF_1(p), \ldots, dF_{m-r}(p))=m-1$$ since $N_{p}^{\ast}\S^{k-1}(\o)=\langle dF_1(p), \ldots, dF_{m-r}(p)\rangle$. Thus, $$\det(dF_1(p), \ldots, dF_{m-r}(p), \tilde{\f}_1,\ldots,\tilde{\f}_{r})=0$$ and fixing the notation $\tilde{\f}_i=(\tilde{\f}_i^1,\ldots,\tilde{\f}_i^m)$ for $i=1,\ldots,r,$ we can suppose without loss of generality that
\begin{equation*}{\renewcommand{\arraystretch}{2}\left|\begin{array}{ccccccc}
\ffrac{\partial F_1}{\partial x_1}(p)& \cdots &\ffrac{\partial F_{m-r}}{\partial x_1}(p) & \tilde{\f}_1^1&\cdots &\tilde{\f}_{r-1}^1\\
\vdots&\ddots&\vdots & \vdots  &\ddots &\vdots\\
\ffrac{\partial F_1}{\partial x_{m-1}}(p)& \cdots &\ffrac{\partial F_{m-r}}{\partial x_{m-1}}(p) &\tilde{\f}_1^{m-1}&\cdots &\tilde{\f}_{r-1}^{m-1}
\end{array}\right|}\neq0
\end{equation*} and consequently, that
\begin{equation}\label{matrizFU}{\renewcommand{\arraystretch}{2}\left|\begin{array}{ccccccc}
\ffrac{\partial F_1}{\partial x_1}(x)& \cdots &\ffrac{\partial F_{m-r}}{\partial x_1}(x) & \f_1^1&\cdots &\f_{r-1}^1\\
\vdots&\ddots&\vdots & \vdots  &\ddots &\vdots\\
\ffrac{\partial F_1}{\partial x_{m-1}}(x)& \cdots &\ffrac{\partial F_{m-r}}{\partial x_{m-1}}(x) &\f_1^{m-1}&\cdots &\f_{r-1}^{m-1}
\end{array}\right|}\neq0
\end{equation} for all $(x,\f)\in(\S^{k-1}(\o)\cap\mathcal{U})\times \mathcal{V}$, where $\mathcal{V}\subset\R^{mr}$ is an open subset. Thus,  $N_{\S^{k-1}}^{\ast}M^{r}$ can be locally defined by
\begin{equation}\label{eqlocalNSM}
N_{\S^{k-1}}^{\ast}M^{r}=\{(x,\f)\in\mathcal{U}\times \mathcal{V} \ | \ F_1=\ldots=F_{m-r}=\d=0\},
\end{equation} where $\d(x,\f)=\det(dF_1(x),\ldots,dF_{m-r}(x),\f_1,\ldots,\f_{r})$.

Let $B(x,\f)$ be the square matrix of order $m$ whose columns are given by the coefficients of the 1-forms $dF_1(x)$, $\ldots$, $dF_{m-r}(x)$, $\f_1$, $\ldots$, $\f_{r}$:
\begin{equation*}B(x,\f)=\left(\begin{array}{cccccc}
dF_1(x) & \cdots & dF_{m-r}(x) & \f_1 & \cdots & \f_{r}
\end{array}\right).
\end{equation*} We have, $$\d(x,\f)=\displaystyle\sum_{i=1}^{m}{\f_{r}^{i}\cof(\f_{r}^i,B)},$$ where $\cof(\f_{r}^i,B)$ denotes the cofactor of $\f_{r}^i$ in the matrix $B(x,\f)$ so that
%d\d(x,\f)=\sum_{i=1}^{m}{\cof(\f_{r}^i,B)d\f_{r}^{i} + \f_{r}^{i}d\cof(\f_{r}^{i},B)}.
%\end{equation*} Em particular, 
\begin{equation*}
\ffrac{\partial\d}{\partial \f_{r}^{m}}(x,\f)=\displaystyle\sum_{i=1}^{m}{\cof(\f_{r}^i,B)\ffrac{\partial \f_{r}^{i}}{\partial \f_{r}^{m}} + \f_{r}^{i}\ffrac{\partial \cof(\f_{r}^{i},B)}{\partial \f_{r}^{m}}}
\end{equation*} and since $\cof(\f_{r}^{i},B)$ does not depend on the variable $\f_{r}^{m}$, $$\frac{\partial \cof(\f_{r}^{i},B)}{\partial \f_{r}^{m}}=0, \text{ for } i=1,\ldots,m.$$ Then, $$\ffrac{\partial\d}{\partial \f_{r}^{m}}(x,\f)=\cof(\f_{r}^{m},B)\overset{(\ref{matrizFU})}{\neq}0,$$ and the derivative of $\d(x,\f)$ with respect to $\f$ does not vanish, that is, $d_\f\d(x,\f)\neq0$ and the matrix
\begin{equation*}\left[\renewcommand{\arraystretch}{1.7}{\begin{array}{c}
d F_1(x)\\
\vdots\\
d F_{m-r}(x)\\
d \d(x,\f)\\
\end{array}}\right]=\left[\setlength{\arraycolsep}{0.1cm}{
\begin{array}{ccc}
      d_x F_1(x)         &  \vdots &                                     \\
      \vdots                  &  \vdots &   O_{(m-r)\times(r)}        \\
      d_x F_{m-r}(x) &  \vdots &                                     \\
\cdots \ \cdots \ \cdots \ \cdots \ \cdots &  \vdots & \cdots \ \cdots \ \cdots \ \cdots \ \cdots \ \cdots \\
       d_x\d(x,\f)        &  \vdots &   d_\f\d(x,\f)                    \\
\end{array}}\right]
\end{equation*} has rank $ m-r+1$, where $O_{(m-r)\times(r)}$ denotes a null matrix. Hence, $$\rank(d F_1(x),\ldots,d F_{m-r}(x),d\d(x,\f))=m-r+1,$$ for each $(x,\f)\in N_{\S^{k-1}}^{\ast}M^{r}\cap(\mathcal{U}\times \mathcal{V})$ and, therefore, $N_{\S^{k-1}}^{\ast}M^{r}$ is a smooth submanifold of $T_{\S^{k-1}}^{\ast}M^{r}$ of dimension $m+mr-(m-r+1)=mr+r-1.$
\end{proof}

{By the induction hypothesis, we have that for each $p\in\S^{k-1}(\o)$, $$\dim(\langle\o_1(p),\ldots,\o_{n}(p)\rangle\cap N_p^{\ast}\S^{k-2}(\o))=k-2$$ and there exist an open neighborhood $\mathcal{U}\subset M$ of $p$ and functions $F_1,\ldots,F_{m-r}:\mathcal{U}\rightarrow\R$ such that $\mathcal{U}\cap\S^{k-1}(\o)=\{x\in\mathcal{U} \ | \ F_1(x)=\ldots=F_{m-r}(x)=0\}$ with $\rank(dF_1(x),\ldots,dF_{m-r}(x))=m-r$, for each $x\in\S^{k-1}(\o)\cap\mathcal{U}.$ Then, we can choose $\{\z_1(x),\ldots,\z_{r}(x)\}$ a smooth $r$-coframe defined on $\mathcal{U}$ which restriction to $\mathcal{U}\cap \S^{k-1}(\o)$ is a smooth basis of a vector subspace supplementary to \begin{equation}\label{base} \langle\o_1(x),\ldots,\o_{n}(x)\rangle\cap N_x^{\ast}\S^{k-2}(\o)\end{equation} in $\langle\o_1(x),\ldots,\o_{n}(x)\rangle$.} Let $\z^{k-1}: \S^{k-1}(\o)\cap\mathcal{U}\rightarrow T_{\S^{k-1}}^{\ast}M^{r}$ be the map given by $\z^{k-1}(x)=(x,\z_1(x),\ldots,\z_{r}(x))$, we define:

%By the induction hypothesis, we have that for each $x\in\S^{k-1}(\o)$, $$\dim(\langle\o_1(x),\ldots,\o_{n}(x)\rangle\cap N_x^{\ast}\S^{k-2}(\o))=k-2.$$ Then, we can choose $\{\z_1(x),\ldots,\z_{r}(x)\}$ a smooth $r$-coframe defined on $\mathcal{U}$ which restriction to $\S^{k-1}(\o)$ is a smooth basis of a vector subspace supplementary to \begin{equation}\label{base} \langle\o_1(x),\ldots,\o_{n}(x)\rangle\cap N_x^{\ast}\S^{k-2}(\o)\end{equation} in $\langle\o_1(x),\ldots,\o_{n}(x)\rangle$. Let $\z^{k-1}: \S^{k-1}(\o)\cap\mathcal{U}\rightarrow T_{\S^{k-1}}^{\ast}M^{r}$ be the map given by $\z^{k-1}(x)=(x,\z_1(x),\ldots,\z_{r}(x))$, we define:

\begin{defi}\label{def2point3} We say that the $n$-coframe $\o=(\o_1,\ldots,\o_{n})$ satisfies the ``intersection properties $I_k$'', if {for each $p\in\S^{k-1}(\o)$ there exist an open neighborhood $\mathcal{U}\subset M$ of $p$ and a map $\z^{k-1}: \S^{k-1}(\o)\cap\mathcal{U}\rightarrow T_{\S^{k-1}}^{\ast}M^{r}$ as defined above, such that on $\mathcal{U}$ the following properties hold:} 
\begin{enumerate}[(a)] 
\item $\z^{k-1}\pitchfork N_{\S^{k-1}}^{\ast}M^{r}$ in $T_{\S^{k-1}}^{\ast}M^{r}$;
\item $(\z^{k-1})^{-1}(N_{\S^{k-1}}^{\ast}M^{r,\geq2})=\emptyset$;
\end{enumerate} where \small{$N_{\S^{k-1}}^{\ast}M^{r,\geq2}=\{(x,\f)\in T_{\S^{k-1}}^{\ast}M^{r} \ | \ \rank(\f_1,\ldots,\f_{r})=r,
\dim(\langle\f_1,\ldots,\f_{r}\rangle\cap N_x^{\ast}\S^{k-1}(\o))\geq2\}$}.
\end{defi}

\normalsize 

Note that, if the $n$-coframe $\o$ satisfies the properties $I_{k}$ $(a)$ and $(b)$ then, for each $x\in\S^{k-1}(\o)\cap\mathcal{U}$, $\dim(\langle\z_1(x),\ldots,\z_{r}(x)\rangle\cap N_x^{\ast}\S^{k-1}(\o))$ is either equal to $0$ or equal to $1$. 

\begin{defi}\label{defiSk} Let $\o$ be an $n$-coframe with $\corank 1$ that satisfies the intersection properties $I_{k}$ $(a)$ and $(b)$. We say that a point $x\in\S^{k-1}(\o)$ belongs to $A_{k-1}(\o)$ if $$\dim(\langle\z_1(x),\ldots,\z_{r}(x)\rangle\cap N_x^{\ast}\S^{k-1}(\o))=0;$$ and we say that $x$ belongs to $\S^{k}(\o)$ if $x\in\S^{k-1}(\o)\setminus A_{k-1}(\o)$, that is, if $$\dim(\langle\z_1(x),\ldots,\z_{r}(x)\rangle\cap N_x^{\ast}\S^{k-1}(\o))=1.$$ Therefore, \begin{equation*}\setlength{\arraycolsep}{0.05cm}{\begin{array}{cl}
A_{k-1}(\o)&=\{x\in\S^{k-1}(\o)|\dim(\langle\z_1(x),\ldots,\z_{r}(x)\rangle\cap N_x^{\ast}\S^{k-1}(\o))=0\};\\
 \S^{k}(\o)&=\{x\in\S^{k-1}(\o)|\dim(\langle\z_1(x),\ldots,\z_{r}(x)\rangle\cap N_x^{\ast}\S^{k-1}(\o))=1\}.
\end{array}}\end{equation*}
\end{defi}

\begin{defi}\label{Morinsing} {Let $\o$ be an $n$-coframe with $\corank 1$ that satisfies the intersection properties $I_{k}$ $(a)$ and $(b)$.} We say that a point $x\in M$ is a Morin singular point of type $A_{k}$ of the $n$-coframe $\o$ if $x\in A_{k}(\o)$.
\end{defi}

\begin{lema}\label{skdimension} By Definition \ref{def2point3}, $\S^{k}(\o)$ is either the empty set or a smooth submanifold of $M$ of dimension $n-k$ .
\end{lema}

\begin{proof} Note that, locally, $\S^{k}(\o)=(\z^{k-1})^{-1}(N^{\ast}_{\S^{k-1}}M^{r})$ and $\z^{k-1}\pitchfork N_{\S^{k-1}}^{\ast}M^{r}$, thus, if $\Sigma^k(\o)\neq\emptyset$ then $\Sigma^k(\o)$ is a smooth submanifold of $\S^{k-1}(\o)$ of codimension 1, that is, $\dim(\Sigma^k(\o))=n-k$.
\end{proof}

\begin{lema}\label{dimensaointersecao} If $p\in\S^{k}(\o)$ then $\dim(\langle\o_1(p),\ldots,\o_{n}(p)\rangle\cap N_p^{\ast}\S^{k-1}(\o))=k-1.$ 
\end{lema}
\begin{proof} For clearer notations, let us write $\langle\bar{\o}(x)\rangle=\langle\o_1(x),\ldots,\o_n(x)\rangle$ and $\langle\bar{\z}^{k-1}(x)\rangle=\langle\z_1(x),\ldots,\z_{r}(x)\rangle$.
%\begin{itemize}
%\item $\langle\o(x)\rangle=\langle\o_1(x),\ldots,\o_n(x)\rangle$,
%\item $\langle\z(x)\rangle=\langle\z_1(x),\ldots,\z_{r}(x)\rangle$.
%\end{itemize} 
Let $p\in\S^{k}(\o)$, since $\S^{k}(\o)\subset\S^{k-1}(\o)\subset\S^{k-2}(\o)$, {there exist an open neighborhood $\mathcal{U}\subset M$ of $p$ and functions $F_{1}, \ldots, F_{m-n+k}:\mathcal{U}\rightarrow\R$ such that} the submanifolds $\S^{i}(\o)$, $i=k-2,k-1,k,$ can be locally defined by $$\mathcal{U}\cap\S^{i}(\o)=\{x\in\mathcal{U} \, | \, F_1(x)= \ldots= F_{m-n+i}(x)=0\},$$ where the derivatives $\{dF_1(x), \ldots, dF_{m-n+i}(x)\}$ are 1-forms linearly independent for each $x\in \S^{i}(\o)\cap\mathcal{U}$ and $N_{x}^{\ast}\S^{i}(\o)=\langle dF_1(x), \ldots, dF_{m-n+i}(x)\rangle$.

By the way the $r$-coframe $\{\z_1(x),\ldots,\z_{r}(x)\}$ has been chosen, for each $x\in\S^{k-1}(\o)\cap\mathcal{U}$ we have $$\langle\bar{\o}(x)\rangle=(\langle\bar{\o}(x)\rangle\cap N_x^{\ast}\S^{k-2}(\o))\oplus\langle\bar{\z}^{k-1}(x)\rangle,$$ and since $N_x^{\ast}\S^{k-2}(\o)\subset N_x^{\ast}\S^{k-1}(\o)$, $\langle\bar{\o}(x)\rangle\cap N_x^{\ast}\S^{k-1}(\o)$ is equal to $$(\langle\bar{\o}(x)\rangle\cap N_x^{\ast}\S^{k-2}(\o))\oplus(\langle\bar{\z}^{k-1}(x)\rangle\cap N_x^{\ast}\S^{k-1}(\o)).$$ Since $p\in\S^{k}(\o)\subset\S^{k-1}(\o)$, we know that $\dim(\langle\bar{\o}(p)\rangle\cap N_p^{\ast}\S^{k-2}(\o))=k-2$ and by the definition of $\S^{k}(\o)$, $\dim(\langle\bar{\z}^{k-1}(p)\rangle\cap N_p^{\ast}\S^{k-1}(\o))=1$. Therefore, $\dim(\langle\bar{\o}(p)\rangle\cap N_p^{\ast}\S^{k-1}(\o))=(k-2)+1=k-1$. \end{proof}

Next, we will show that Definitions \ref{def2point3} and \ref{defiSk} do not depend on the choice of the basis $\{\z_1(x), \ldots, \z_{r}(x)\}$. To do this, first we must find equations that define the manifold $\S^{k}(\o)$ locally.

%\textcolor{red}{Pela forma como definimos os conjuntos singulares} $\S^{k+1}(\o)$ and $A_k(\o)$, we shall consider a basis $\{\z_1(x), \ldots, \z_{n-k}(x)\}$ of a vector subspace supplementary to $\langle\o_1(x), \ldots, \o_n(x)\rangle\cap N_x^{\ast}\S^{k-1}(V)$ in $\langle\o_1(x), \ldots, \o_n(x)\rangle$ (see Equation \ref{base}). We will show next that the definitions of $\S^{k+1}(\o)$ and $A_k(\o)$ do not depend on the choice of these basis. To do this, we start finding equations that describe locally a manifold $\S^{k}(\o)$.

{
\begin{lema}\label{eqlocalSk} Let $p\in\S^{k-1}(\o)$. There are an open neighborhood $\mathcal{U}\subset M$ of $p$ and functions $F_i:\mathcal{U}\rightarrow\R$, $i=1, \ldots, m-r$, such that $$\mathcal{U}\cap\S^{k-1}(\o)=\{x\in\mathcal{U} \, | \, F_1(x)=\ldots=F_{m-r}(x)=0\},$$ and a smooth $r$-coframe defined on $\mathcal{U}$ $\{\z_1(x), \ldots, \z_{r}(x)\}$ which is a basis of a vector subspace supplementary to $\langle\bar{\o}(x)\rangle\cap N_x^{\ast}\S^{k-2}(\o)$ in $\langle\bar{\o}(x)\rangle$ for each $x\in\mathcal{U}\cap\S^{k-1}(\o)$. Let $$\d_{k}(x)=\det(dF_1,\ldots,dF_{m-r},\z_1, \ldots, \z_{r})(x).$$ Then $\o$ satisfies the intersection properties $I_k$ on $\mathcal{U}$ if and only if the following properties hold:
\begin{enumerate}[(i)] 
\item $\dim\langle\z_1(x), \ldots, \z_{r}(x)\rangle\cap N_x^{\ast}\S^{k-1}(\o)=0$ or $1$ for $x\in\mathcal{U}\cap\S^{k-1}(\o)$;\\
\item if $\dim\langle\z_1(x), \ldots, \z_{r}(x)\rangle\cap N_x^{\ast}\S^{k-1}(\o)=1$ (or equivalently $\d_k(x)=0$), then $\rank(dF_1(x),\ldots, dF_{m-r}(x), d\d_{k}(x))=m-r+1$.
\end{enumerate}
In this case, $\mathcal{U}\cap\S^{k}(\o)=\{x\in\mathcal{U} \, | \, F_1(x)=\ldots=F_{m-r}(x)=\d_k(x)=0\}.$
\end{lema}}

\begin{proof} %Assume that $\o$ satisfies the intersection properties $I_k$ and that $p\in\S^{k}(\o)$. 

%As we have seen in the proof of Lemma \ref{dimNSM}, we can consider $\mathcal{U}\subset M$, with $p\in \mathcal{U}$, and functions $F_1,\ldots,F_{m-r}:\mathcal{U}\rightarrow\R$, such that \begin{equation}\label{sigmak1}\mathcal{U}\cap\S^{k-1}(\o)=\{x\in\mathcal{U} \ | \ F_1(x)=\ldots=F_{m-r}(x)=0\}\end{equation} and $\rank(dF_1(x),\ldots,dF_{m-r}(x))=m-r,$ for each $ x\in\S^{k-1}(\o)\cap\mathcal{U}$.

%By Definition \ref{defiSk}, we know that $x\in\S^{k}(\o)$ if and only if $x\in\S^{k-1}(\o)$ and $\dim(\langle\z_1(x),\ldots,\z_{r}(x)\rangle\cap N_x^{\ast}\S^{k-1}(\o))=1$. Then, we can define $\S^k(\o)$ locally as $$\mathcal{U}\cap\S^k(\o)=\{x\in\mathcal{U} \ | \ F_1(x)=\ldots=F_{m-r}(x)=\d_k(x)=0\},$$ since $\dim(\langle\z_1(x),\ldots,\z_{r}(x)\rangle\cap N_x^{\ast}\S^{k-1}(\o))=1$ if and only if $\d_k(x)=0$, for each $x\in\S^{k-1}(\o)\cap\mathcal{U}$. 

{First, let us show that for each $\bar{x}\in\mathcal{U}\cap\S^k(\o)$, $$\rank\left(dF_1(\bar{x}),\ldots,dF_{m-r}(\bar{x}),d\d_k(\bar{x})\right)$$ is equal to $m-r+1$  if and only if $\z^{k-1}\pitchfork N_{\S^{k-1}}^{\ast}M^{r}$ in $T_{\S^{k-1}}^{\ast}M^{r}$ at $\bar{x}$.}

{By Lemma \ref{dimNSM}, $N_{\S^{k-1}}^{\ast}M^{r}$ can be locally defined by $$N_{\S^{k-1}}^{\ast}M^{r}=\{(x,\f)\in\mathcal{U}\times \mathcal{V}|F_1=\ldots=F_{m-r}=\d=0\},$$ where $\d(x,\f)=\det(dF_1(x),\ldots,dF_{m-r}(x),\f_1,\ldots,\f_{r})$ and $\mathcal{V}\subset\R^{mr}$. Let $$G(\z^{k-1})=\{(x,\z_1(x),\ldots,\z_{r}(x))\ | \ x\in\mathcal{U}\cap\S^{k-1}(\o)\}$$ be the restriction of the graph of $(\z_1(x),\ldots,\z_{r}(x))$ to $\mathcal{U}\cap\S^{k-1}(\o)$, $G(\z^{k-1})$ can be locally defined by 
\begin{equation*}\label{eqlocalGz}\begin{array}{c}
G(\z^{k-1})=\{(x,\f)\in T^{\ast}M^{r} \ | \ F_1(x)=\ldots=F_{m-r}(x)=0;\\
\ \ \ \ \ \ \ \ \ \ \ \ \ \ \ \ \ \ \ \ \z_i^j(x)-\f_i^j=0, i=1,\ldots,r \text{ and } j=1,\ldots,m\},
\end{array}
\end{equation*} where $T^{\ast}M^{r}$ denotes the $r$-cotangent bundle of $M$, $\z_i(x)=(\z_i^1(x),\ldots,\z_i^m(x))$ and $\f_i=(\f_i^1,\ldots,\f_i^m)$ for $i=1,\ldots, r$. In particular, the local equations of $G(\z^{k-1})$ are clearly independent and $\dim G(\z^{k-1})=r$.
Let $(x,\f)$ be local coordinates in $T^{\ast}M^{r}$, with $x=(x_1,\ldots,x_m)$ and $$\f=(\f_1^1,\ldots,\f_1^m,\f_2^1,\ldots,\f_2^m,\ldots,\f_{r}^1,\ldots,\f_{r}^m),$$ let us consider the derivatives of the local equations of $N_{\S^{k-1}}^{\ast}M^{r}$ and $G(\z^{k-1})$ with respect to $(x,\f)$.
%Primeiramente, vamos considerar os gradientes em $TM^{n-k+1}$ das equações locais de $N_{\S^{k-1}}M^{n-k+1}$ e $G(z)$ descritas, respectivamente, na equação (\ref{eqlocalNSM}) do Lema \ref{dimNSM} e na Observação \ref{eqlocalGz}. 
We will denote the derivative with respect to $x$ by $d_x$ and the derivative with respect to $\f$ by $d_{\f}$, then we have 
\begin{equation}\label{gradeqgraf}d\left( \z_i^j(x)-\f_i^j\right)=\left(  d_x\z_i^j(x) \ , -d_{\f}\f_i^j \right),\end{equation} for $i=1,\ldots,r$ and $j=1,\ldots,m$, where $d_{\f} {\f}_i^j=(0,\ldots,0,1,0,\ldots,0)$ is the vector whose $m(i-1)+j^{th}$ entry is equal to $1$ and the others are zero. By Lagrange's rules the determinant $\d(x,{\f})=\det(d F_1(x),\ldots,d F_{m-r}(x),{\f}_1,\ldots,{\f}_{r})$ can be written as $$\d(x,{\f})=\sum_I{F_I(x)N_I({\f})}$$ for $I=\{i_1,\ldots,i_{r}\}\subset\{1,\ldots,m\}$, where
\begin{equation}\label{NIU}N_I({\f})=\left|\begin{array}{ccc}
{\f}_1^{i_1} & \ldots & {\f}_{r}^{i_1}\\ 
\vdots & \ddots & \vdots\\
{\f}_1^{i_{r}} & \ldots & {\f}_{r}^{i_{r}}\\ 
\end{array}\right|
\end{equation} is the minor obtained from the matrix 
\begin{equation*}\left[\begin{array}{ccc}
{\f}_1^{1} & \ldots & {\f}_{r}^{1}\\ 
\vdots & \ddots & \vdots\\
{\f}_1^{m} & \ldots & {\f}_{r}^{m}\\ 
\end{array}\right]
\end{equation*} taking the lines $i_1,\ldots,i_{r}$, and 
\begin{equation}\label{FI}F_I(x)=\pm\left|\begin{array}{ccc}
\ffrac{\partial F_1}{\partial x_{k_1}}(x) & \ldots & \ffrac{\partial F_{m-r}}{\partial x_{k_1}}(x)\\ 
\vdots & \ddots & \vdots\\
\ffrac{\partial F_{1}}{\partial x_{k_{m-r}}}(x) & \ldots & \ffrac{\partial F_{m-r}}{\partial x_{k_{m-r}}}(x)\\ 
\end{array}\right|
\end{equation} is, up to sign, the minor obtained from the matrix $(d F_1(x) \ldots d F_{m-r}(x))$ removing the lines $i_1,\ldots,i_{r}$, that is, $\{k_1,\ldots,k_{m-r}\}=\{1,\ldots,m\}\setminus I$. Therefore,
$$d\d(x,\f)=( \ \displaystyle\sum_I{N_I(\f)d_xF_I(x)} \ , \ \displaystyle\sum_I{F_I(x)d_{\f}N_I(\f)} \ ).$$}

{Note that $\z^{k-1}\pitchfork N_{\S^{k-1}}^{\ast}M^{r}$ in $T_{\S^{k-1}}^{\ast}M^{r}$ at the point $x\in\mathcal{U}\cap\S^{k-1}(\o)$ if and only if $G(\z^{k-1})\pitchfork N_{\S^{k-1}}^{\ast}M^{r}$ in $T_{\S^{k-1}}^{\ast}M^{r}$ at $(x,\z^{k-1}(x))$. Let $\pi_1$ be the projection of the contangent space of $T^{\ast}M^{r}$ over the contangent space of $T_{\S^{k-1}}^{\ast}M^{r}$:
\begin{equation*}\begin{array}{cccc}
\pi_1: & T_{(x,\f)}^{\ast}(T^{\ast}M^{r}) & \longrightarrow & T_{(x,\f)}^{\ast}(T_{\S^{k-1}}^{\ast}M^{r})\\
  & (\psi(x),\f_1,\ldots,\f_{r}) & \longmapsto & (\pi(\psi(x)),\f_1,\ldots,\f_{r})
\end{array}
\end{equation*} where $\pi$ denotes the restriction to $T_x\S^{k-1}(\o)$, that is, $\pi(\psi(x))=\psi(x)_{|_{T_x\S^{k-1}(\o)}}$. By Equation (\ref{gradeqgraf}), $$\pi_1\left(d( \z_i^j(x)-\f_i^j)\right)=\left(  \pi(d_x\z_i^j(x)) \ , -d_{\f}\f_i^j \right),$$ for $i=1,\ldots,r$ and $j=1,\ldots,m$. We also have that $$\pi_1\left(d\d(x,\f)\right)=\left( \ \pi\left(\displaystyle\sum_I{N_I(\f)d_xF_I(x)}\right) \ , \ \displaystyle\sum_I{F_I(x)d_{\f}N_I(\f)} \ \right).$$ Then, $G(\z^{k-1})\pitchfork N_{\S^{k-1}}^{\ast}M^{r}$ in $T_{\S^{k-1}}^{\ast}M^{r}$ at $(x,\z^{k-1}(x))$ if and only if the matrix
\begin{equation}\label{LijLdelta}\left[\setlength{\arraycolsep}{0.12cm}{\begin{array}{ccc}
  \pi(d_x\z_1^1(x))       &  \vdots &                                     \\
  \vdots                      &  \vdots &           \\
  \pi(d_x\z_1^m(x))       &  \vdots &      -Id_{mr}                               \\
  \vdots                      &  \vdots &   \\
  %\pi(d_x\z_{r}^1(x)) &  \vdots &                      \\
  %\vdots                      &  \vdots &   \\
  \pi(d_x\z_{r}^m(x)) &  \vdots &                  \\
\cdots \  \cdots \ \cdots \ \cdots \ \cdots \ \cdots \ \cdots & \vdots & \cdots \ \cdots \ \cdots \ \cdots \ \cdots \ \cdots \\
\pi\left(\displaystyle\sum_I{N_I(\f)d_xF_I(x)}\right)  & \vdots &   \displaystyle\sum_I{F_I(x)d_{\f}N_I(\f)}
\end{array}}\right]
\end{equation} has maximal rank at $x$. By the expression of $N_I(\f)$ in (\ref{NIU}), we have 
\begin{equation}\label{dN}
d_{\f} N_I(\f)=\sum_{i,j}{\cof(\f_i^j)d_{\f} \f_i^j},
\end{equation} for $i=1,\ldots, r$, $j\in I$ and $\cof(\f_i^j)$ denoting the cofactor of $\f_i^j$ in the matrix 
\begin{equation*}\left[\begin{array}{ccc}
{\f}_1^{i_1} & \ldots & {\f}_{r}^{i_1}\\ 
\vdots & \ddots & \vdots\\
{\f}_1^{i_{r}} & \ldots & {\f}_{r}^{i_{r}}\\ 
\end{array}\right].
\end{equation*} Let $d=C_{m,r}=\ffrac{m!}{r!(m-r)!}$, we will denote by $I_1, \ldots, I_d$ the subsets of $\{1, \ldots, m\}$ containing exactly $r$ elements. By equation (\ref{dN}), $$\displaystyle\sum_I{F_I(x)d_{\f}N_I({\f})}=\displaystyle\sum_{\l=1}^{d}F_{I_{\l}}(x)\left(\displaystyle\sum_{i=1}^{r}\displaystyle\sum_{j\in {I_{\l}}}\cof({\f}_i^j)d_{\f} {\f}_i^j\right)$$ and, 
$$\setlength{\arraycolsep}{0.1cm}{\renewcommand{\arraystretch}{2.7}{\begin{array}{l}
%\displaystyle\sum_I{F_I(x)d_{\f}N_I({\f})}\\
%\displaystyle\sum_IF_I(x)\left(\displaystyle\sum_{i=1}^{r}\displaystyle\sum_{j\in I}\cof({\f}_i^j)d_{\f} {\f}_i^j\right)=
\displaystyle\sum_{\l=1}^{d}F_{I_{\l}}(x)\left(\displaystyle\sum_{i=1}^{r}\displaystyle\sum_{j\in {I_{\l}}}\cof({\f}_i^j)d_{\f} {\f}_i^j\right)\\
%=\displaystyle\sum_{s=1}^{d}F_{I_s}(x)\left(\displaystyle\sum_{i=1}^{r}\displaystyle\sum_{j\in {I_s}}\cof({\f}_i^j)d_{\f} {\f}_i^j\right)\\
=\displaystyle\sum_{i=1}^{r}\left[F_{I_1}(x)\left(\displaystyle\sum_{j\in I_1}\cof({\f}_i^j)d_{\f} {\f}_i^j\right)+ \ldots +F_{I_d}(x)\left(\displaystyle\sum_{j\in I_d}\cof({\f}_i^j)d_{\f} {\f}_i^j\right)\right]\\
=\displaystyle\sum_{i=1}^{r}\left[\left(\displaystyle\sum_{I: \, 1\in I}F_I(x)\right)\cof({\f}_i^1)d_{\f}{{\f}_i^1}+\ldots+\left(\displaystyle\sum_{I: \, m\in{I}}F_I(x)\right)\cof({\f}_i^m)d_{\f} {\f}_i^m \right]\\
=\displaystyle\sum_{i=1}^{r}\left[\displaystyle\sum_{j=1}^{m}\left(\displaystyle\sum_{I: \, j\in I}F_I(x)\right)\cof({\f}_i^j)d_{\f} {\f}_i^j\right].
\end{array}}}$$}

{Thus, for $i=1, \ldots, r$ and $j=1, \ldots, m$, we can write
\begin{equation}\label{Ldelta2}
\displaystyle\sum_I{F_I(x)d_{\f}N_I({\f})}= \displaystyle\sum_{i,j}{\beta_{i}^j(x,{\f})d_{\f} {\f}_i^j},
\end{equation} where $$\beta_{i}^j(x,{\f})=\left(\displaystyle\sum_{I: \, j\in I}F_I(x)\right)\cof({\f}_i^j).$$}

%%E denotando $$\beta_{i}^j(x,{\f})=\left(\displaystyle\sum_{j\in I}F_I(x)\right)\cof({\f}_i^j),$$ obtemos 
%%\begin{equation}\label{Ldelta2}
%%\displaystyle\sum_I{F_I(x)d_{\f}N_I({\f})}= \displaystyle\sum_{i,j}{\beta_{i}^j(x,{\f})d {\f}_i^j},
%%\end{equation} com $i=1, \ldots, r$ e $j=1, \ldots, m$.

{We will denote the rows of the Matrix (\ref{LijLdelta}) by $R_i^j=\left(  \pi(d_x\z_i^j(x)) \ , -d_{\f}{\f}_i^j \right)$, for $i=1,\ldots,r$ and $j=1,\ldots,m$, and we denote the last row of the Matrix (\ref{LijLdelta}) by $R_{\d}$.
Replacing the row $R_{\d}$ by $$R_{\d}+\sum_{i,j}\beta_i^j(x,{\f})R_i^j$$ for $i=1,\ldots,r$ and $j=1, \ldots, m$, we obtain a new matrix 
\begin{equation}\label{novaLijLdelta}\left[\setlength{\arraycolsep}{0.1cm}{\begin{array}{ccc}
  \pi(d_x\z_1^1(x))        & \vdots &     \\
  \vdots                       & \vdots & -Id_{mr}     \\
  %\pi(d_x\z_1^m(x))        & \vdots &      \\
  %\vdots                       & \vdots &  -Id_{mr} \\
  %\pi(d_x\z_{r}^1(x))  & \vdots &                  \\
  %\vdots  & \vdots &   \\
  \pi(d_x\z_{r}^m(x))  & \vdots &    \\
 \cdots \ \cdots \ \cdots \ \cdots  \ \cdots \ \cdots & \vdots & \cdots \ \cdots  \ \cdots \ \cdots \ \cdots \\
  R_{\d}'  & \vdots &   R_{\d}''
\end{array}}\right]
\end{equation} which has rank equal to the rank of the Matrix (\ref{LijLdelta}), where 
\begin{equation*}
R_{\d}''=\displaystyle\sum_I{F_I(x)d_{\f}N_I({\f})}+\sum_{i,j}\beta_i^j(x,{\f})(-d_{\f} {\f}_i^j)\overset{(\ref{Ldelta2})}{=}\vec{0}
\end{equation*} and
\begin{equation*}{\renewcommand{\arraystretch}{2.5}\begin{array}{ll}
R_{\d}'&=\pi\left(\displaystyle\sum_I{N_I({\f})d_xF_I(x)}\right)+ \displaystyle\sum_{i,j}\beta_i^j(x,{\f})\pi\left(d_x \z_i^j(x)\right)\\
 &=\pi\left(\displaystyle\sum_I{N_I({\f})d_xF_I(x)}+ \sum_{i,j}\beta_i^j(x,{\f})d_x \z_i^j(x)\right).
\end{array}}
\end{equation*} Note that for each $\bar{x}\in\mathcal{U}\cap\S^{k}(\o)$, we have $\z_i^j(\bar{x})={\f}_i^j$. In this case, Equation (\ref{Ldelta2}) implies that $$\displaystyle\sum_{i,j}\beta_i^j(\bar{x},{\f})d_x \z_i^j(\bar{x})=\sum_{i,j}\beta_i^j(\bar{x},\z^{k-1}(\bar{x}))d_x \z_i^j(\bar{x})=\displaystyle\sum_I{F_I(\bar{x})d_xN_I(\z^{k-1}(\bar{x}))}.$$ Thus, at $\bar{x}$ $$R_{\d}'=\pi\left(\displaystyle\sum_I{N_I(\z^{k-1}(\bar{x}))d_xF_I(\bar{x})}+ \displaystyle\sum_I{F_I(\bar{x})d_xN_I(\z^{k-1}(\bar{x}))}\right)=\pi(d\d_k(\bar{x}))$$ and the Matrix (\ref{novaLijLdelta}) is equal to
\begin{equation*}\left[\setlength{\arraycolsep}{0.1cm}{\begin{array}{ccc}
  \pi(d_x\z_1^1(\bar{x}))       &  \vdots &                                     \\
  \vdots                      &  \vdots &        -Id_{mr}    \\
  %\pi(d_x\z_1^m(x))       &  \vdots &                                     \\
  %\vdots                      &  \vdots &  -Id_{mr} \\
  %\pi(d_x\z_{r}^1(x)) &  \vdots &                      \\
  %\vdots                      &  \vdots &   \\
  \pi(d_x\z_{r}^m(\bar{x})) &  \vdots &                  \\
 \cdots \ \cdots \ \cdots \ \cdots  \ \cdots \ \cdots & \vdots & \cdots \ \cdots  \ \cdots \ \cdots \ \cdots \\
 \pi(d\d_k(\bar{x})) & \vdots &   \vec{0}
\end{array}}\right].
\end{equation*} Thus, for each $\bar{x}\in\mathcal{U}\cap\S^{k}(\o)$, $\z^{k-1}\pitchfork N_{\S^{k-1}}^{\ast}M^{r}$ in $T_{\S^{k-1}}^{\ast}M^{r}$ at $\bar{x}$ if and only if $\pi(d\d_k(\bar{x}))\neq0$, that is, the restriction of $d\d_k(\bar{x})$ to $T_{\bar{x}}\S^{k-1}(\o)$ is not zero, which means that $d\d_k(\bar{x})\notin\langle d F_1(\bar{x}),\ldots,d F_{m-r}(\bar{x})\rangle$, or equivalently $\rank\left(dF_1(\bar{x}),\ldots,dF_{m-r}(\bar{x}),d\d_k(\bar{x})\right)=m-r+1$.}

{Now suppose that $\o$ satisfies the intersection properties $I_k$ on $\mathcal{U}$. By property $(b)$ of Definition \ref{def2point3}, we have that $\dim\langle\z_1(x), \ldots, \z_{r}(x)\rangle\cap N_x^{\ast}\S^{k-1}(\o)$ is either equal to $0$ or equal to $1$ for $x\in\mathcal{U}\cap\S^{k-1}(\o)$. If $\dim\langle\z_1(x), \ldots, \z_{r}(x)\rangle\cap N_x^{\ast}\S^{k-1}(\o)=1$, then $\d_k(x)=0$ and $x\in\mathcal{U}\cap\S^{k}(\o)$. In this case, the tranversality given by property $(a)$ of Definition \ref{def2point3} implies that $\rank\left(dF_1(x),\ldots,dF_{m-r}(x),d\d_k(x)\right)=m-r+1$.}

{On the other hand, we assume that properties $(i)$ and $(ii)$ hold for each $x\in\mathcal{U}\cap\S^{k-1}(\o)$. By property $(i)$, the intersection property $(b)$ of Definition \ref{def2point3} holds on $\mathcal{U}$. If $\dim\langle\z_1(x), \ldots, \z_{r}(x)\rangle\cap N_x^{\ast}\S^{k-1}(\o)=0$ then $\z^{k-1}(x)$ does not intersect $N_{\S^{k-1}}^{\ast}M^r$, thus $\z^{k-1}\pitchfork N_{\S^{k-1}}^{\ast}M^{r}$ in $T_{\S^{k-1}}^{\ast}M^{r}$ at $x$. If $\dim\langle\z_1(x), \ldots, \z_{r}(x)\rangle\cap N_x^{\ast}\S^{k-1}(\o)=1$ then $x\in\mathcal{U}\cap\S^{k}(\o)$ by Definition \ref{defiSk} and $\rank\left(dF_1(x),\ldots,dF_{m-r}(x),d\d_k(x)\right)=m-r+1$ by property $(ii)$. Thus $\z^{k-1}\pitchfork N_{\S^{k-1}}^{\ast}M^{r}$ in $T_{\S^{k-1}}^{\ast}M^{r}$ at $x$ and $\o$ satisfies the intersection properties $I_k$ on $\mathcal{U}$.}

{Finally, if $\o$ satisfies the intersection properties $I_k$ on $\mathcal{U}$, it follows by Definition \ref{defiSk} that $\mathcal{U}\cap\S^{k}(\o)=\{x\in\mathcal{U} \, | \, F_1(x)=\ldots=F_{m-r}(x)=\d_k(x)=0\}$.}
\end{proof}

The following technical lemma will be used in the proof of Lemma \ref{indepdabase}.

%%%%%%%%%%%%%%%%%%%%%%%%%%%%%%%%%%%%%%%%%%%%%%%%%%%%%%% NÃO APAGAR %%%%%%%%%%%%%%%%%%%%%%%%%%%%%%%%%%%%%%%%%%%%%%%%%%%%%%%%%%%%%%%%%%

%\begin{lema}\label{lematecnicoum} Let $f_i:\mathcal{V}\subset\R^{\l}\rightarrow\R, i=1, \ldots,s$ be smooth functions defined on an open neighborhood of $\R^{\l}$. Let $M\subset\R^{\l}$ be a manifold locally defined by $M=\{x\in\mathcal{V}| f_1(x)=\ldots=f_s(x)=0\}$, with $\rank(df_1(x), \ldots, df_s(x))=s$ for all $x\in M\cap\mathcal{V}$. If $g,h:\mathcal{V}\subset\R^{\l}\rightarrow\R$ are smooth functions such that $g|_{M\cap\mathcal{V}}=h|_{M\cap\mathcal{V}}$ then, for all $x\in M\cap\mathcal{V}$, we have $$\langledf_1(x), \ldots, df_s(x),dg(x)\rangle=\langledf_1(x), \ldots, df_s(x), dh(x)\rangle.$$
%\end{lema}
%
%\begin{proof} If $g|_{M\cap\mathcal{V}}=h|_{M\cap\mathcal{V}}$ then $g-h\equiv0$ on $M\cap\mathcal{V}$. That is, $T_xM$ is contained in $\Ker d(g-h)(x)$ for each $x\in M\cap\mathcal{V}$. Thus, $d(g-h)(x)\in \langledf_1(x), \ldots, df_s(x)\rangle $, and 
%$$dg(x)=dh(x)+\displaystyle\sum^{s}_{i=1}{\lambda_i df_i(x)}.$$ Therefore $\langledf_1(x),\ldots,df_s(x),dg(x)\rangle=\langle df_1(x),\ldots,df_s(x),dh(x)\rangle.$ 
%\end{proof}

%%%%%%%%%%%%%%%%%%%%%%%%%%%%%%%%%%%%%%%%%%%%%%%%%%%%%%%%%%%%%%%%%%%%%%%%%%%%%%%%%%%%%%%%%%%%%%%%%%%%%%%%%%%%%%%%%%%%%%%%%%%%%%%%%%%%%%%%%%%%%%%

\begin{lema}\label{lematecnicodois} Let $f_i:\mathcal{V}\subset\R^{\l}\rightarrow\R, i=1, \ldots,s$ be smooth functions defined on an open neighborhood of $\R^{\l}$. Let $M\subset\R^{\l}$ be a manifold given locally by $M=\{x\in\mathcal{V}| f_1(x)=\ldots=f_s(x)=0\}$, with $\rank(df_1(x), \ldots,df_s(x))=s$, for all $x\in M\cap\mathcal{V}$. If $g,h:\mathcal{V}\subset\R^{\l}\rightarrow\R$ are smooth functions such that $g(x)=\lambda(x)h(x)$, for all $ x\in M\cap\mathcal{V}$ and some smooth function $\lambda:\mathcal{V}\rightarrow\R$, then:
\begin{enumerate}[(i)]
\item If $\lambda(x)\neq0$ and $x\in M$ then $g(x)=0\Leftrightarrow h(x)=0.$
\item If $\lambda(x)\neq0$, $x\in M$ and $h(x)=0$ then $$\langle df_1(x), \ldots, df_s(x),dg(x)\rangle=\langle df_1(x), \ldots, df_s(x), dh(x)\rangle.$$
\end{enumerate}
\end{lema}

%%%%%%%%%%%%%%%%%%%%%%%%%%%%%%%%%%%%%%%%%%%%%%%%%%%%%%% NÃO APAGAR %%%%%%%%%%%%%%%%%%%%%%%%%%%%%%%%%%%%%%%%%%%%%%%%%%%%%%%%%%%%%%%%%%%%%%%%%%%%

%\begin{proof} The statement $(i)$ clearly holds. We will proof $(ii)$.
%
%Let $\tilde{\lambda}:\mathcal{V}\rightarrow\R$ be a smooth function such that $\tilde{\lambda}(x)=\lambda(x)$ for any $x\in M\cap\mathcal{V}$ and let $\tilde{g}:\mathcal{V}\rightarrow\R$ be defined by $\tilde{g}(x)=\tilde{\lambda}(x)h(x)$. Since $\tilde{g}\equiv g$ on $M\cap\mathcal{V}$, by Lemma \ref{lematecnicoum}, we have $$\langle df_1(x),\ldots,df_s(x),d\tilde{g}(x)\rangle=\langle df_1(x),\ldots,df_s(x),dg(x)\rangle,$$ for all $x\in M\cap\mathcal{V}.$ Since $d\tilde{g}(x)=\tilde{\lambda}(x)dh(x)+h(x)d\tilde{\lambda}(x)$, if $x\in M\cap\mathcal{V}$, $\tilde{\lambda}(x)\neq0$ and $h(x)=0$ then $d\tilde{g}(x)=\tilde{\lambda}(x)dh(x)$. Thus, $$\langle df_1(x), \ldots, df_s(x),d\tilde{g}(x)\rangle=\langle df_1(x), \ldots, df_s(x), dh(x)\rangle.$$ Therefore, $\langle df_1(x), \ldots, df_s(x),d{g}(x)\rangle=\langle df_1(x), \ldots, df_s(x),dh(x)\rangle.$
%\end{proof}

%%%%%%%%%%%%%%%%%%%%%%%%%%%%%%%%%%%%%%%%%%%%%%%%%%%%%%%%%%%%%%%%%%%%%%%%%%%%%%%%%%%%%%%%%%%%%%%%%%%%%%%%%%%%%%%%%%%%%%%%%%%%%%%%%%%%%%%%%%%%%%%

\begin{lema}\label{indepdabase} The definitions of $\S^{k+1}(\o)$ and $A_k(\o)$ do not depend on the choice of the basis $\{\z_1, \ldots, \z_{n-k}\}$, for every $k\geq 1$.
\end{lema}

\begin{proof} As for the definition of $\S^{k+1}(\o)$ and $A_k(\o)$, for $k\geq 1$, we will proceed by induction on $k$. First, note that the definition of $\S^1(\o)$ does not depend on the choice of any basis. Then, assume as induction hypothesis that the definition of $\S^i(\o)$ does not depend on the choice of the basis for every $i\leq k$.
{We know that, for each $p\in\S^k(\o)$, there is an open neighborhood $\mathcal{U}\subset M$ of $p$ so that} 
\begin{equation*}\setlength{\arraycolsep}{0.06cm}{\begin{array}{cll}
\mathcal{U}\cap\S^{k}(\o)&=&\{x\in\mathcal{U}: F_1(x)=\ldots=F_{m-n+1}(x)=\d_2(x)=\ldots=\d_{k}(x)=0\},\\
\mathcal{U}\cap\S^{k+1}(\o)&=&\{x\in\mathcal{U}: F_1(x)=\ldots=F_{m-n+1}(x)=\d_2(x)=\ldots=\d_{k+1}(x)=0\},\\
\end{array}}
\end{equation*} with $\rank(dF_1(x),\ldots,dF_{m-n+1}(x),d\d_2(x),\ldots,d\d_{k}(x))=m-n+k$, for $x\in\mathcal{U}\cap\S^{k}(\o)$ and $\rank(dF_1(x),\ldots,dF_{m-n+1}(x),d\d_2(x),\ldots,d\d_{k+1}(x))=m-n+k+1$, for $x\in\mathcal{U}\cap\S^{k+1}(\o)$.
 %$$\setlength{\arraycolsep}{0.06cm}{\begin{array}{cll}
%\rank(dF_1(x),\ldots,dF_{m-n+1}(x),d\d_2(x),\ldots,d\d_{k}(x))&=&m-n+k, \\
%\rank(dF_1(x),\ldots,dF_{m-n+1}(x),d\d_2(x),\ldots,d\d_{k+1}(x))&=&m-n+k+1.
%\end{array}}$$
Let us recall that $$\d_{k+1}(x)=\det(dF_1,\ldots,dF_{m-n+1},d\d_2,\ldots,d\d_{k}, \z_1, \ldots, \z_{n-k})(x),$$ {where $\{\z_1(x), \ldots, \z_{n-k}(x)\}$ is a smooth $(n-k)$-coframe defined on $\mathcal{U}$ which is a basis of a vector subspace supplementary to $\langle\bar{\o}(x)\rangle\cap N_x^{\ast}\S^{k-1}(\o)$ in $\langle\bar{\o}(x)\rangle$ for each $x\in\mathcal{U}\cap\S^k(\o)$.}

{Let us consider $\{\tilde{\z}_1(x), \ldots, \tilde{\z}_{n-k}(x)\}$ a smooth $(n-k)$-coframe defined on $\mathcal{U}$ such that, for each $x\in\mathcal{U}\cap\S^k(\o)$, $\{\tilde{\z}_1(x), \ldots, \tilde{\z}_{n-k}(x)\}$ is another basis of a vector subspace supplementary to $\langle\bar{\o}(x)\rangle\cap N_x^{\ast}\S^{k-1}(\o)$ in $\langle\bar{\o}(x)\rangle$.} Then, $$\langle\bar{\o}(x)\rangle=\left(\langle\bar{\o}(x)\rangle\cap N_x^{\ast}\S^{k-1}(\o)\right)\oplus\langle\tilde{\z}_1(x), \ldots, \tilde{\z}_{n-k}(x)\rangle$$ and {$\dim(\langle\tilde{\z}_1(x), \ldots, \tilde{\z}_{n-k}(x)\rangle\cap N_x^{\ast}\S^{k}(\o))$ is either equal to 0 or equal to $1$, for $x\in\mathcal{U}\cap\S^k(\o)$. Moreover,}

\begin{equation*}\left\{\renewcommand{\arraystretch}{2.5}{\begin{array}{l}
\tilde{\z}_1(x)=\displaystyle\sum_{\l=1}^{n-k}{a_{\l1}(x)\z_{\l}(x)} + \f_1(x)\\
\tilde{\z}_2(x)=\displaystyle\sum_{\l=1}^{n-k}{a_{\l2}(x)\z_{\l}(x)} + \f_2(x)\\
\vdots\\
\tilde{\z}_{n-k}(x)=\displaystyle\sum_{\l=1}^{n-k}{a_{\l(n-k)}(x)\z_{\l}(x)} + \f_{n-k}(x)\\
\end{array}}\right.
\end{equation*} where $a_{ij}(x)\in\R$ and $\f_j(x)\in\langle\bar{\o}(x)\rangle\cap N_x^{\ast}\S^{k-1}(\o)$, for $j=1, \ldots, n-k$. {We will show that for each $x\in\mathcal{U}\cap\S^k(\o)$}, $$\det(A(x))=\left|\begin{array}{cccc}
a_{11}(x) & a_{12}(x) & \cdots & a_{1(n-k)}(x)\\
\vdots & \vdots & \ddots & \vdots\\
a_{(n-k)1}(x) & a_{(n-k)2}(x) & \cdots & a_{(n-k)(n-k)}(x)\\
\end{array}\right|\neq0.$$

Suppose that the statement is false, that is, $\det(A(x))=0$.
This means that the columns of matrix $A(x)$ are linearly dependent.
So we can suppose without loss of generality that the first column of $A(x)$ can be written as a linear combination of the others columns: $$(a_{11}(x), \ldots, a_{(n-k)1}(x))=\displaystyle\sum_{s=2}^{n-k}{\lambda_s(a_{1s}(x), \ldots, a_{(n-k)s}(x))},$$ where $\lambda_s\in\R$, for $s=2, \ldots, n-k$. Thus, deleting $x$ in the notation, we have

$$\renewcommand{\arraystretch}{3.5}{\begin{array}{ll}
\tilde{\z}_1=\displaystyle\sum_{\l=1}^{n-k}{a_{\l1}\z_{\l}} + \f_1 &\Rightarrow\tilde{\z}_1=\displaystyle\sum_{\l=1}^{n-k}{\left(\displaystyle\sum_{s=2}^{n-k}{\lambda_sa_{\l s}}\right)\z_{\l}} + \f_1\\
&\Rightarrow\tilde{\z}_1=\displaystyle\sum_{s=2}^{n-k}{\lambda_s\left(\displaystyle\sum_{\l=1}^{n-k}{a_{\l s}\z_{\l}}\right)} + \f_1\\
\end{array}}$$ then, $$\renewcommand{\arraystretch}{3.5}{\setlength{\arraycolsep}{0.1cm}{\begin{array}{lcl}
\tilde{\z}_1-\displaystyle\sum_{s=2}^{n-k}{\lambda_s\tilde{\z}_s}&=&\left[\displaystyle\sum_{s=2}^{n-k}{\lambda_s\left(\displaystyle\sum_{\l=1}^{n-k}{a_{\l s}\z_{\l}}\right)} + \f_1\right] -\displaystyle\sum_{s=2}^{n-k}{\lambda_s\left(\displaystyle\sum_{\l=1}^{n-k}{a_{\l s}\z_{\l}} + \f_s\right)}\\
&=&\f_1-\displaystyle\sum_{s=2}^{n-k}{\lambda_s\f_s}.
\end{array}}}$$ This means that $$\tilde{\z}_1-\displaystyle\sum_{s=2}^{n-k}{\lambda_s\tilde{\z}_s}\in\left(\langle\bar{\o}\rangle\cap N_x^{\ast}\S^{k-1}(\o)\right)\cap\langle\tilde{\z}_1, \ldots, \tilde{\z}_{n-k}\rangle=\{0\},$$ that is, $\tilde{\z}_1(x), \ldots, \tilde{\z}_{n-k}(x)$ are linearly dependent. However, this contradicts the initial assumption that $\{\tilde{\z}_1(x), \ldots, \tilde{\z}_{n-k}(x)\}$ is a basis of a vector subspace for each $x$ in $\mathcal{U}\cap\S^k(\o)$. Therefore, $\det(A(x))\neq0$.

{Let ${}^t\!A(x)$ be the transpose of matrix $A(x)$}. For each $x\in\mathcal{U}\cap\S^{k}(\o)$, we have $\det({}^t\!A(x))=\det(A(x))\neq0$ and, deleting $x$ in the notation, 
$$\setlength{\arraycolsep}{0.1cm}{\begin{array}{lll}
\tilde{\d}_{k+1}&=&\det(dF_1,\ldots,dF_{m-n+1},d\d_2,\ldots,d\d_{k},\tilde{\z}_1, \ldots, \tilde{\z}_{n-k})\\
&=&\det(dF_1,\ldots,dF_{m-n+1},d\d_2,\ldots,d\d_{k},\displaystyle\sum_{\l=1}^{n-k}{a_{\l1}\z_{\l}}, \ldots, \displaystyle\sum_{\l=1}^{n-k}{a_{\l(n-k)}\z_{\l}})\\
&=&\det({}^t\!A)\det(dF_1,\ldots,dF_{m-n+1},d\d_2,\ldots,d\d_{k},\z_1, \ldots, \z_{n-k})\\
&=&\det({}^t\!A)\d_{k+1}.
\end{array}}$$ {So, by statement $(i)$ of Lemma \ref{lematecnicodois}, $\tilde{\d}_{k+1}(x)=0\Leftrightarrow \d_{k+1}(x)=0$ for $x\in\mathcal{U}\cap\S^{k}(\o)$. Since $\d_{k+1}(x)=0$ if and only if $\dim(\langle\z_1(x),\ldots,\z_{n-k}(x)\rangle\cap N_x^{\ast}\S^{k}(\o))=1$ and $\tilde{\d}_{k+1}(x)=0$ if and only if $\dim(\langle\tilde{\z}_1(x), \ldots, \tilde{\z}_{n-k}(x)\rangle\cap N_x^{\ast}\S^{k}(\o))=1$, by Definition \ref{defiSk} we have that
\begin{equation*}\begin{array}{ccl}
x\in\mathcal{U}\cap\S^{k+1}(\o) & \Leftrightarrow & \dim(\langle\z_1(x),\ldots,\z_{n-k}(x)\rangle\cap N_x^{\ast}\S^{k}(\o))=1\\
& \Leftrightarrow &\d_{k+1}(x)=0\\
& \Leftrightarrow &\tilde{\d}_{k+1}(x)=0\\
& \Leftrightarrow & \dim(\langle\tilde{\z}_1(x), \ldots, \tilde{\z}_{n-k}(x)\rangle\cap N_x^{\ast}\S^{k}(\o))=1
\end{array}
\end{equation*}}

{In particular, if $x\in\mathcal{U}\cap\S^{k+1}(\o)$ we have $\d_{k+1}(x)=0$ and $\tilde{\d}_{k+1}(x)=0$ so that by statement $(ii)$ of Lemma \ref{lematecnicodois}, 
$$\renewcommand{\arraystretch}{1.5}{\begin{array}{l}
\langle dF_1(x),\ldots,dF_{m-n+1}(x),d\d_2(x),\ldots,d\d_{k}(x), d\d_{k+1}(x)\rangle\\
=\langle dF_1(x),\ldots,dF_{m-n+1}(x),d\d_2(x),\ldots,d\d_{k}(x), d\tilde{\d}_{k+1}(x)\rangle,\end{array}}$$ which implies that $\rank(dF_1(x),\ldots,dF_{m-n+1}(x),d\d_2(x),\ldots,d\d_{k}(x), d\tilde{\d}_{k+1}(x))$ is equal to $m-n+k+1$.
%{\small$$\rank(dF_1(x),\ldots,dF_{m-n+1}(x),d\d_2(x),\ldots,d\d_{k}(x), d\tilde{\d}_{k+1}(x))=m-n+k+1.$$} 
Therefore, the intersection properties $I_{k+1}$ and the definition of $\S^{k+1}(\o)$ do not depend on the choice of the basis $\{\z_1(x),\ldots,\z_{n-k}(x)\}$. Since $A_k(\o)=\S^k(\o)\setminus\S^{k+1}(\o),$ we conclude that $A_k(\o)$ also does not depend on the choice of the basis.} \end{proof}

\begin{obs}\label{akclosure} It is not difficult to see that $\S^k(\o)$ is a closed submanifold of $M$, for $k\geq1$. Moreover, we can write $$\S^k(\o)=A_k(\o)\cup\S^{k+1}(\o)=\displaystyle\cup_{i\geq k}A_i(\o)$$
%$$\begin{array}{ccl}
%\S^k(\o)&=&A_k(\o)\cup\S^{k+1}(\o);\\
%\S^k(\o)&=&\displaystyle\cup_{i\geq k}A_i(\o);
%\end{array}$$ 
so that $A_k(\o)=\S^k(\o)\setminus\S^{k+1}(\o).$ That is, the singular sets $A_k(\o)$ are $(n-k)$-dimensional submanifolds of $M$ such that $\overline{A_k(\o)}=\S^k(\o)$.
\end{obs}

Finally, based on the previous considerations, we define:

\begin{defi}\label{def:ncoframe} An $n$-coframe $\o$ is a Morin $n$-coframe if $\o$ has $\corank 1$ and it satisfies the intersection properties $I_{k}$ $(a)$ and $(b)$ for $k=2,\ldots,n$.  
\end{defi}

\begin{obs}\label{singak} By Definition \ref{def:ncoframe}, if $\o$ is a Morin $n$-coframe then $\o$ admits only singular points of type $A_{k}$ for $k=1,\ldots,n$.
\end{obs}

As we mentioned in Section \ref{Introduction}, fixed a Riemannian metric on $M$, we can consider vector fields instead of 1-forms and define the notion of Morin $n$-frames analogously to the definition of Morin $n$-coframes:

\begin{defi}\label{def:nframe} {An $n$-frame $V=(V_1,\ldots, V_n): M\rightarrow TM^n$ is a Morin $n$-frame if $V$ has $\corank 1$ and it satisfies the intersection properties $I_{k}$ $(a)$ and $(b)$ for $k=2,\ldots,n$.}
\end{defi}

Next, we present some examples of Morin $n$-frames.

\begin{ex}\label{ex1} Let $f:M^m\rightarrow\R^n$ be a smooth Morin map defined on an $m$-dimensional Riemannian manifold $M$, with $m\geq n$. The $n$-frame $V(x)=(\nabla f_1(x), \ldots, \nabla f_n(x))$ given by the gradient of the coordinate functions of $f$ is, clearly, a Morin $n$-coframe whose singular points are the same that the singular points of $f$. That is, $A_k(V)=A_k(f)$, $\forall k=1, \ldots, n$.
\end{ex}

\begin{ex}\label{ex2} Let $a\in\R$ be a regular value of a $C^2$ mapping $f:\R^3\rightarrow\R$. Suppose that $M=f^{-1}(a)$ and consider $V=(V_1,V_2)$ be a $2$-frame on $M$, given by
\begin{equation*}\begin{array}{ccc}
 V_1(x)&=&(-f_{x_2}(x),f_{x_1}(x),0);\\
 V_2(x)&=&(-f_{x_3}(x),0,f_{x_1}(x)).
\end{array}
\end{equation*} Since $a$ is a regular value of $f$, we have that $\nabla f(x)=(f_{x_1}(x), f_{x_2}(x), f_{x_3}(x))\neq\vec{0}$, $\forall x\in M$. Thus, $\rank(V_1(x), V_2(x))$ is either equal to $2$ or equal to $1$ . The singular points of $V$ are the points $x\in M$ where $\rank(V_1(x), V_2(x))=1$, that is, $$\S^1(V)=\{x\in M | f_{x_1}(x)=0\}$$ {and $V=(V_1, V_2)$ has $\corank 1$ if and only if  $\rank(\nabla f(x), \nabla f_{x_1}(x))=2$ for each $x\in\S^1(V)$. In this case, $\S^1(V)$ is a submanifold of $M$ of dimension 1}. Let $x\in\S^1(V)$ be a singular point of $V$, then the space $\langle V_1(x),V_2(x)\rangle$ is spanned by the vector $e_1=(1,0,0)$ and $x\in A_2(V)$ if and only if $$\rank(\nabla f(x), \nabla f_{x_1}(x), e_1)<3,$$ that is, if and only if 
%\begin{equation*}\left|\begin{array}{ccc}
%f_{x_1}(x) & f_{x_2}(x) & f_{x_3}(x)\\
%f_{x_1x_1}(x) & f_{x_2x_1}(x) & f_{x_3x_1}(x)\\
%1 & 0 & 0
%\end{array}\right|=0
%\end{equation*} which is equivalent to 
$\d_2:=f_{x_2}f_{x_1x_3}-f_{x_3}f_{x_1x_2}$ vanishes at $x$. {Moreover, $V$ satisfies the intersection properties $I_2$ if and only if $\rank(\nabla f(x), \nabla f_{x_1}(x), \nabla\d_2(x))=3$ for $x\in A_2(V)$.} In this case, $A_2(V)$ is a submanifold of $M$ of dimension 0. Therefore, $V=(V_1, V_2)$ is a Morin $2$-frame if and only if $\rank(\nabla f(x), \nabla f_{x_1}(x))=2$ on the singular set $\S^1(V)=\{x\in M | f_{x_1}(x)=0\}$ and $\det(\nabla f(x), \nabla f_{x_1}(x), \nabla\d_2(x))\neq0$ on $A_2(V)=\{x\in M | f_{x_1}(x)=0, \d_2(x)=0\}.$
\end{ex}

\begin{ex}\label{ex3} Let us apply Example \ref{ex2} to the $2$-frame $V=(V_1, V_2)$ defined on the torus $\emph{T}:=f^{-1}(R^2)$, where $R^2$ is a regular value of $$f(x_1,x_2,x_3)=(\sqrt{x_2^2+x_3^2}-a)^2+(x_1+x_2)^2,$$ with $a>R$. Then, one can verify that $\S^1(V)=\{x\in \emph{T} \, | \, x_1+x_2=0\}$, that is, 
$$\S^1(V)=\{(x_1,x_2,x_3)\in\R^3 | \sqrt{x_2^2+x_3^2}-a)^2=R^2\}$$ and $\rank(\nabla f(x), \nabla f_{x_1}(x))$ is equal to 
$$\rank\left[\begin{array}{ccc}
0& \ffrac{2x_2(\sqrt{x_2^2+x_3^2}-a)}{\sqrt{x_2^2+x_3^2}} & \ffrac{2x_3(\sqrt{x_2^2+x_3^2}-a)}{\sqrt{x_2^2+x_3^2}}\\
1&1&0
\end{array}\right]
$$ which is $2$, for all $x\in \emph{T}\cap\S^1(V)$. Moreover, $$\d_2(x)=\ffrac{-4x_3(\sqrt{x_2^2+x_3^2}-a)}{\sqrt{x_2^2+x_3^2}},$$ so that $A_2(V)=\{x\in \emph{T} \, | \, x_1+x_2=0; x_3=0\}$ which is the set given by the points $(-a-R,a+R,0)$, $(a+R,-a-R,0)$, $(-a+R,a-R,0)$ and $(a-R,-a+R,0)$. It is not difficult to see that $\rank(\nabla f(x), \nabla f_{x_1}(x), \nabla\d_2(x))=3, \forall x\in\emph{T}\cap A_2(V)$. Therefore, the frame $V=(V_1,V_2)$ given by
\begin{equation*}\begin{array}{lll}
 V_1(x)&=&\left(\frac{-2x_2(\sqrt{x_2^2+x_3^2}-a)}{\sqrt{x_2^2+x_3^2}}-2(x_1+x_2),2(x_1+x_2),0\right);\\
 V_2(x)&=&\left(\frac{-2x_3(\sqrt{x_2^2+x_3^2}-a)}{\sqrt{x_2^2+x_3^2}},0,2(x_1+x_2)\right).
\end{array}
\end{equation*} is a Morin $2$-frame defined on the torus $\emph{T}$ which admits singular points of type $A_1$ and $A_2$.
\end{ex}

\begin{ex}\label{ex4} Let $a\in\R$ be a regular value of a $C^2$ mapping $f:\R^3\rightarrow\R$. Suppose that $M=f^{-1}(a)$ and consider $\overline{W_1}$ and $\overline{W_2}$ {be the orthogonal projections of $e_2=(0,1,0)$ and $e_3=(0,0,1)$ over $T_xM$} given by
\begin{equation*}\renewcommand{\arraystretch}{1.7}{\begin{array}{ccc}
 \overline{W_1}&=& e_2-\left\langle e_2,\frac{\nabla f}{|\nabla f|}\right\rangle\frac{\nabla f}{|\nabla f|};\\
 \overline{W_2}&=& e_3-\left\langle e_3,\frac{\nabla f}{|\nabla f|}\right\rangle\frac{\nabla f}{|\nabla f|}.\\
\end{array}}
\end{equation*} Let $W=(W_1,W_2)$ be the $2$-frame defined by $W_1=\left\|\nabla f\right\|^2\overline{W_1}$ and $W_2=\left\|\nabla f\right\|^2\overline{W_2}$, that is,
\begin{equation*}\begin{array}{ccc}
 W_1&=&(-f_{x_1}f_{x_2},f_{x_1}^2+f_{x_3}^2,-f_{x_2}f_{x_3});\\
 W_2&=&(-f_{x_1}f_{x_3},-f_{x_2}f_{x_3},f_{x_1}^2+f_{x_2}^2).
\end{array}
\end{equation*} {Note that in this case, $W_1$ and $W_2$ are gradients vector fields, that is, $W$ is a $2$-frame gradient}. It is not difficult to see that $\rank(W_1(x), W_2(x))$ is either equal to $2$ or equal to $1$ and the singular set of $W$ is $\S^1(W)=\{x\in M | f_{x_1}(x)=0\}$. Let $x\in\S^1(W)$ be a singular point of $W$, then the space $\langle W_1(x),W_2(x)\rangle$ is spanned by the vector $(0,f_{x_3}, -f_{x_2})$, so that $A_2(W)=\{x\in M | f_{x_1}(x)=0, f_{x_1x_1}(x)=0\}$. Therefore, $W=(W_1, W_2)$ is a Morin $2$-frame if and only if $\rank(\nabla f(x), \nabla f_{x_1}(x))=2$ on the singular set $\S^1(W)$ and $\det(\nabla f(x), \nabla f_{x_1}(x), \nabla f_{x_1x_1}(x))\neq0$ on $A_2(W)$.
\end{ex}

\begin{ex} \label{ex5} Let us apply Example \ref{ex4} to the $2$-frame $W=(W_1, W_2)$ defined on the torus $\emph{T}:=f^{-1}(R^2)$ of Example \ref{ex3}. In this situation, one can verify that $\S^1(W)$ is the same singular set as $\S^1(V)$ in the Example \ref{ex3}. Moreover, $\rank(\nabla f(x), \nabla f_{x_1}(x))=2, \forall x\in\S^1(W)$. However, since $f_{x_1x_1}(x)=2, \forall x\in\S^1(W)$, we have that $W$ does not admits singular points of type $A_2$. That is, $W$ is Morin $2$-frame on $\emph{T}$ which admits only Morin singularities of type $A_1$.
\end{ex}

\begin{ex} \label{ex6} Let us consider the $2$-frames $V=(V_1,V_2)$ and $W=(W_1, W_2)$ from Examples \ref{ex2} and \ref{ex4} defined on {the unit sphere $M:=f^{-1}(1)$}, where $f(x_1,x_2,x_3)=x_1^2+x_2^2+x_3^2$. We know that the singular sets of $V$ and $W$ are the same, that is, $\S^1(V)=\S^1(W)=\{x\in M \, | \, x_1=0\}$ and $\rank(\nabla f(x), \nabla f_{x_1}(x))=2$ for all singular point $x$. However, $\d_2(x)=0, \forall x\in\S^1(V)$, so that $\nabla\d_2\equiv\vec{0}$. On the other hand, $f_{x_1x_1}(x)\neq0, \forall x\in\S^1(W)$, so that $A_2(W)=\emptyset$. Therefore, $V$ is not a Morin $2$-frame and $W$ is a Morin $2$-frame that admits only Morin singularities of type $A_1$.
\end{ex}

\begin{ex}\label{ex7} In the Example \ref{ex6}, if we consider $f(x_1,x_2,x_3)=x_1^2-x_1x_2+x_3^2$ then one can verify that $V$ and $W$ are both Morin $2$-frames that admits only Morin singularities of type $A_1$. {Let us consider the case where $V$ of Example \ref{ex2} is defined on $M:=f^{-1}(-1)$ and $f(x_1,x_2,x_3)=x_1^2-x_1x_2+x_3^2$. It is easy to see that $-1$ is a regular value of $f$ and $\S^1(V)=\{x\in \emph{M} \, | \, 2x_1-x_2=0\}$. That is, 
$$\S^1(V)=\{(x_1,x_2,x_3)\in\R^3 | \, x_1^2-x_1x_2+x_3^2+1=0; 2x_1-x_2=0\}$$ and $\rank(\nabla f(x), \nabla f_{x_1}(x))$ is equal to
$$\rank\left[\begin{array}{ccc}
(2x_1-x_2)& -x_1 & 2x_3\\
2&-1&0
\end{array}\right]
$$ which is $2$, for all $x\in M\cap\S^1(V)$. Moreover, $\d_2(x)=2x_3$ and $$A_2(V)=\{(x_1,x_2,x_3)\in\R^3 \, | \, x_1^2-x_1x_2+x_3^2+1=0; 2x_1-x_2=0; x_3=0\}$$ which is the set given by the points $(1,2,0)$ and $(-1,-2,0)$. We also have that $\det(\nabla f(x), \nabla f_{x_1}(x), \nabla\d_2(x))$ is equal to
$$\det\left[\begin{array}{ccc}
(2x_1-x_2)& -x_1 & 2x_3\\
2&-1&0\\
0&0&2
\end{array}\right]=4x_1
$$ which is equal to $\pm4$ for each $x\in A_2(V)$. That is, $\rank(\nabla f(x), \nabla f_{x_1}(x), \nabla\d_2(x))=3$, $\forall x\in M\cap A_2(V)$. Therefore, the $2$-frame $V=(V_1,V_2)$ given by
\begin{equation*}\begin{array}{lll}
 V_1(x)&=&\left(x_1,  2x_1-x_2,0\right);\\
 V_2(x)&=&\left(-2x_3, 0, 2x_1-x_2\right).
\end{array}
\end{equation*} is a Morin $2$-frame defined on $M$ which admits singular points of type $A_1$ and $A_2$.}

\end{ex}

%%%%%%%%%%%%%%%%%%%%%%%%%%%%%%%%%%%%%%%%%%%%%%%  SEÇÃO 2 %%%%%%%%%%%%%%%%%%%%%%%%%%%%%%%%%%%%%%%%%%%%%%%%%%%%%%%%%%%%%%%%%%%%%%%%%%%

\section{Zeros of a generic 1-form $\x(x)$ associated to a Morin $n$-coframe}\label{s2}

Let $a=(a_1,\ldots,a_n)\in\R^n\setminus\{\vec{0}\}$ and let $\o=(\o_1,\ldots,\o_n)$ be a Morin $n$-coframe defined on an $m$-dimensional manifold $M$. In this section, we will consider the 1-form $\x(x)=\displaystyle\sum_{i=1}^{n}{a_i\o_i(x)}$ defined on $M$ and we will proof some properties of the zeros of $\x$ and its restrictions to the singular sets of $\o$.

\begin{lema}\label{zeroszsobresigma1} If $p$ is a zero of the 1-form $\x$ %isto é, $\x(p)=\displaystyle\sum_{i=1}^{n}{a_i\o_i(p)}=0$, 
then $p\in\S^1(\o)$ and $p$ is a zero of $\x_{|_{\S^1(\o)}}$.
\end{lema}
\begin{proof} Suppose that $\x(p)=0$. So $\rank(\o_1(p),\ldots,\o_n(p))\leq n-1$, since $a\neq\vec{0}$. However, the $n$-coframe $\o$ has $\corank 1$, thus $\rank(\o_1(p),\ldots,\o_n(p))=n-1$. That is, $p\in\S^1(\o)$. Moreover, $\x(p)=0$ implies that $T_pM\subset\ker(\x(p))$ and since $T_p\S_1(\o)\subset T_pM$, we conclude that $p$ is a zero of $\x_{|_{\S^1(\o)}}=0$.
\end{proof}

\begin{lema}\label{lemazerosrestricoes} If $p\in A_{k+1}(\o)$ then, for each $k=0,\ldots,n-2$, $p$ is a zero of $\x_{|_{\S^{k+1}(\o)}}$ if and only if $p$ is a zero of $\x_{|_{\S^{k}(\o)}}$.
\end{lema}

\begin{proof} Suppose that $p\in A_{k+1}(\o)$ and that, locally, we have:
%Suponha que \begin{equation*}\M(x)=\left|
%{\renewcommand{\arraystretch}{2}
%\begin{array}{cccc}
%V_1^1(x)&V_2^1(x)&\cdots &V_{n-1}^1(x)\\
%\vdots & \vdots &\ddots &\vdots\\
%V_1^{n-1}(x)&V_2^{n-1}(x)&\cdots &V_{n-1}^{n-1}(x)
%\end{array}}\right|\neq0
%\end{equation*} em uma vizinhança $\mathcal{U}\subset M$, com $p\in\mathcal{U}$ de maneira que, localmente, escrevemos:
\begin{equation*}\setlength{\arraycolsep}{0.06cm}{\begin{array}{cl}
\mathcal{U}\cap\S^{k}(\o)&=\{x\in\mathcal{U} \, | \, F_1(x)=\ldots=F_{m-n+1}(x)=\d_2(x)=\ldots=\d_k(x)=0\};\\
\mathcal{U}\cap\S^{k+1}(\o)&=\{x\in\mathcal{U} \, | \, F_1(x)=\ldots=F_{m-n+1}(x)=\d_2(x)=\ldots=\d_{k+1}(x)=0\};
\end{array}}\end{equation*} for an open neighborhood $\mathcal{U}\subset M$, with $p\in\mathcal{U}$. If $p$ is a zero of the restriction $\x_{|_{\S^{k}(\o)}}$ then $\x(p)\in N_p^{\ast}\S^{k}(\o)=\langle dF_1(p),\ldots,dF_{m-n+1}(p),d\d_2(p),\ldots,d\d_k(p)\rangle.$ In particular, %$$\x(p)\in\langle dF_1(p),\ldots,dF_{m-n+1}(p),d\d_2(p),\ldots,d\d_k(p),d\d_{k+1}(p)\rangle.$$ Isto é,
$\x(p)\in N_p^{\ast}\S^{k+1}(\o)$, therefore $p$ is a zero of $\x_{|_{\S^{k+1}(\o)}}$.

On the other hand, if $p$ is a zero of $\x_{|_{\S^{k+1}(\o)}}$ then $\x(p)\in N_p^{\ast}\S^{k+1}(\o)\cap\langle\bar{\o}(p)\rangle$. 

Since $p\in A_{k+1}(\o)$, we have that $p\in\S_{k+1}(\o)\setminus\S_{k+2}(\o)$, thus
\begin{equation*}\left\{\begin{array}{l}
\dim(\langle\bar{\o}(p)\rangle\cap N_p^{\ast}\S^{k}(\o))=k;\\
\dim(\langle\bar{\z}^{k+1}(p)\rangle\cap N_p^{\ast}\S^{k+1}(\o))=0; 
\end{array}\right.
\end{equation*} where $\bar{\z}^{k+1}(p)$ represents a smooth basis for a vector subspace supplementary to $\langle\bar{\o}(p)\rangle\cap N_p^{\ast}\S^{k}(\o)$ in $\langle\bar{\o}(p)\rangle$. Since $\dim(N_p^{\ast}\S^{k}(\o))=m-n+k$, $\dim(N_p^{\ast}\S^{k+1}(\o))=m-n+k+1$ and $N_p^{\ast}\S^{k}(\o)\subset N_p^{\ast}\S^{k+1}(\o)$, we have $$\dim(\langle\bar{\o}(p)\rangle\cap N_p^{\ast}\S^{k+1}(\o))=\dim(\langle\bar{\o}(p)\rangle\cap N_p^{\ast}\S^{k}(\o))=k.$$ Thus, $\langle\bar{\o}(p)\rangle\cap N_p^{\ast}\S^{k}(\o)=\langle\bar{\o}(p)\rangle\cap N_p^{\ast}\S^{k+1}(\o).$ Therefore, $\x(p)\in N_p^{\ast}\S^k(\o)$, that is, $p$ is a zero of $\x_{|_{\S^{k}(\o)}}$. \end{proof}

\begin{lema}\label{ptosansaozeros} If $p\in A_n(\o)$ then $p$ is a zero of the restriction $\x_{|_{\S^{n-1}(\o)}}$.
\end{lema}

\begin{proof} Analogously to Lemma \ref{lemazerosrestricoes}, we consider local equations of $\S^{n}(\o)$:
\begin{equation*} %\setlength{\arraycolsep}{0.1cm}{\begin{array}{rll}
\mathcal{U}\cap\S^{n}(\o)=\{x\in\mathcal{U} \, | \, F_1(x)=\ldots=F_{m-n+1}(x)=\d_2(x)=\ldots=\d_n(x)=0\},%\\ N_x^{\ast}\S^{n}(\o)&=&\langle dF_1(x),\ldots,dF_{m-n+1}(x),d\d_2(x),\ldots,d\d_n(x)\rangle.
%\end{array}}
\end{equation*} with $N_x^{\ast}\S^{n}(\o)=\langle dF_1(x),\ldots,dF_{m-n+1}(x),d\d_2(x),\ldots,d\d_n(x)\rangle.$ Since $A_n(\o)=\S^{n}(\o)$, if $p\in A_n(\o)$ then $$\dim(\langle\bar{\o}(p)\rangle\cap N_p^{\ast}\S^{n-1}(\o))=n-1.$$
%\begin{equation*}\begin{array}{ll}
%p\in A_n(\o)&\Rightarrow\dim(\langle\bar{\o}(p)\rangle\cap N_p^{\ast}\S^{n-1}(\o))=n-1\\
%&\Rightarrow\langle\bar{\o}(p)\rangle\subset N_p^{\ast}\S^{n-1}(\o)\\
%&\Rightarrow \x(p)\in N_p^{\ast}\S^{n-1}(\o).
%\end{array}\end{equation*} 
Thus, $\langle\bar{\o}(p)\rangle\subset N_p^{\ast}\S^{n-1}(\o)$ and consequently, $\x(p)\in N_p^{\ast}\S^{n-1}(\o)$. Therefore, %se $p\in A_n(\o)$ então 
$p$ is a zero of $\x_{|_{\S^{n-1}(\o)}}$. \end{proof}

\begin{remark}\label{obsmatrixM} If $p\in\S^1(\o)$ then $\rank(\o_1(p), \ldots, \o_n(p))=n-1$ and, writing $\o_i=(\o_i^1, \ldots, \o_i^m)$, we can suppose without loss of generality that \begin{equation}\label{matrizli}\M(x)=\left|
{\renewcommand{\arraystretch}{2}
\begin{array}{cccc}
\o_1^1(x)&\o_2^1(x)&\cdots &\o_{n-1}^1(x)\\
\vdots & \vdots &\ddots &\vdots\\
\o_1^{n-1}(x)&\o_2^{n-1}(x)&\cdots &\o_{n-1}^{n-1}(x)
\end{array}}\right|\neq0, 
\end{equation} for all $x$ in an open neighborhood $\mathcal{U}\subset M$ with $p\in\mathcal{U}$. In particular, if $p\in\mathcal{U}$ is a singular point of $\x$ then $a_n\neq0$, otherwise, we would have $a_1=\ldots=a_{n-1}=a_n=0$. We will use this fact in next results.
\end{remark}

\begin{lema}\label{lemaauxiliar} Let $p\in\S^1(\o)$ such that $\M(p)\neq0$. Then $\x(p)=0$ if and only if $\displaystyle\sum_{i=1}^{n}{a_i\o_i^j(p)}=0$, for every $j=1,\ldots,n-1$.
\end{lema}

\begin{proof}
It follows easily from the definition of $\S^1(\o)$ and $\x$.
\end{proof}

\begin{lema}\label{lemainterzeroscomsigma2} Let $Z(\x)$ be the zero set of the 1-form $\x$. Then for almost every $a\in\R^n\setminus\{\vec{0}\}$, $Z(\x)\cap\S^2(\o)=\emptyset$.
\end{lema}

\begin{proof} Let $\mathcal{U}\subset M$ be an open neighborhood on which $\M(x)\neq0$ and
%\begin{equation*}\M(x)=\left|
%{\renewcommand{\arraystretch}{2}
%\begin{array}{cccc}
%V_1^1(x)&V_2^1(x)&\cdots &V_{n-1}^1(x)\\
%\vdots & \vdots &\ddots &\vdots\\
%V_1^{n-1}(x)&V_2^{n-1}(x)&\cdots &V_{n-1}^{n-1}(x)
%\end{array}}\right|\neq0,
%\end{equation*} de forma que, localmente, podemos escrever 
$$\mathcal{U}\cap\S^2(V)=\{x\in\mathcal{U} \, | \, F_1(x)=\ldots=F_{m-n+1}(x)=\d_2(x)=0\},$$ with $\rank(dF_1(x),\ldots, dF_{m-n+1}(x),d\d_2(x))=m-n+2,$ for each $x\in\S^2(V)\cap\mathcal{U}$. Let us consider $F:\mathcal{U}\times\R^n\setminus\{\vec{0}\}\rightarrow\R^{m+1}$ the mapping defined by  $$F(x,a)=(F_1(x),\ldots,F_{m-n+1}(x),\d_2(x),\displaystyle\sum_{i=1}^{n}{a_i\o_i^1(x)},\ldots,\displaystyle\sum_{i=1}^{n}{a_i\o_i^{n-1}(x)}).$$ By Lemma \ref{lemaauxiliar}, if $x\in\S^1(\o)$ then $$\displaystyle\sum_{i=1}^{n}{a_i\o_i(x)}=0\Leftrightarrow\displaystyle\sum_{i=1}^{n}{a_i\o_i^j(x)}=0, \forall j=1,\ldots,n-1.$$ Thus, if $(x,a)\in F^{-1}(\vec{0})$ we have that $x\in Z(\x)\cap\S^2(V)$. Furthermore, the Jacobian matrix of $F$ at a point $(x,a)\in F^{-1}(\vec{0})$:
\begin{equation*}\left[\setlength{\arraycolsep}{0.12cm}{\begin{array}{cccccc}
			 dF_1(x) & \vdots & &&& \\
			 \vdots          & \vdots  &&\multicolumn{2}{c}{\multirow{2}{*}{$O_{(m-n+2)\times n}$}}& \\
 			  dF_{m-n+1}(x) & \vdots  &&&& \\
        d\d_2(x) & \vdots  &&&& \\
       \cdots \ \cdots \ \cdots \ \cdots    & \vdots & \cdots \ \, \cdots \ \, \cdots & \cdots  &\cdots \ \, \cdots \ \, \cdots& \cdots \ \, \cdots \\
                        & \vdots & \o_1^1(x) & \cdots & \o_{n-1}^1(x) & \o_n^1(x) \\
       \multirow{2}{*}{$(*)$}      & \vdots & \o_1^2(x) & \cdots & \o_{n-1}^2(x) & \o_n^2(x) \\
                        & \vdots & \vdots   & \ddots & \vdots       & \vdots \\
                        & \vdots & \o_1^{n-1}(x) & \cdots & \o_{n-1}^{n-1}(x) & \o_n^{n-1}(x) \\   
\end{array}}\right]
\end{equation*} has $\rank$ $m+1$. That is, $\vec{0}$ is regular value of $F$ and $F^{-1}(\vec{0})$ is a submanifold of dimension $n-1$. Let $\pi:F^{-1}(\vec{0})\rightarrow\R^n\setminus\{\vec{0}\}$ be the projection over $\R^n\setminus\{\vec{0}\}$ given by $\pi(x,a)=a$, by Sard's Theorem, $a$ is regular value of $\pi$ for almost every $a\in\R^n\setminus\{\vec{0}\}$. Therefore, $\pi^{-1}(a)\cap F^{-1}(\vec{0})=\emptyset$ for almost every $a\in\R^n\setminus\{\vec{0}\}$. However, $
\pi^{-1}(a)\cap F^{-1}(\vec{0})=\{(x,a)\in\mathcal{U}\times\{a\}: x\in Z(\x)\cap\S^2(\o)\}.$ Thus, $Z(\x)\cap\S^2(\o)=\emptyset$ for almost every $a\in\R^n\setminus\{\vec{0}\}$. \end{proof}

\begin{lema}\label{ptscrticrestasigmak} Let $Z(\x_{|_{\S^k(\o)}})$ be the zero set of the restriction of the 1-form $\x$ to $\S^k(\o)$, with $k\geq1$. Then for almost every $a\in\R^n\setminus\{\vec{0}\}$, $Z(\x_{|_{\S^k(\o)}})\cap\S^{k+2}(\o)=\emptyset$.
\end{lema}

\begin{proof} {For each $k=1, \ldots, n-2$}, let $\mathcal{U}\subset M$ be an open neighborhood on which, 
$$\mathcal{U}\cap\S^{k}(\o)=\{x\in\mathcal{U} \, | \, F_1(x)=\ldots=F_{m-n+k}(x)=0\},$$ with $\rank(dF_1(x),\ldots, dF_{m-n+k}(x))=m-n+k$, for all $x\in\mathcal{U}\cap\S^{k}(\o)$ {and $$\mathcal{U}\cap\S^{k+2}(V)=\{x\in\mathcal{U} \, | \, F_1(x)=\ldots=F_{m-n+k+2}(x)=0\},$$ with $\rank(dF_1(x),\ldots, dF_{m-n+k+2}(x))=m-n+k+2$, for all $x\in\mathcal{U}\cap\S^{k+2}(V)$.}

By Szafraniec's characterization (see \cite[p. 196]{sza2}) adapted to 1-forms, $x$ is a zero of the restriction $\x_{|_{\S^k(\o)}}$ if and only if there exists $(\lambda_1,\ldots,\lambda_{m-n+k})\in\R^{m-n+k}$ such that $$\x(x)=\displaystyle\sum_{j=1}^{m-n+k}{\lambda_jdF_j(x)}.$$ %That is, if and only if $\x(x)-\displaystyle\sum_{j=1}^{m-n+k}{\lambda_jdF_j(x)}=\vec{0}.$ 
Let us write $\x(x)=(\x_1(x),\ldots,\x_m(x))$, where $\x_s(x)=\displaystyle\sum_{i=1}^{n}{a_i\o_i^s(x)}$, $s= 1,\ldots, m$, we define
 $$N_s(x,a,\lambda):=\x_s(x)-\displaystyle\sum_{j=1}^{m-n+k}{\lambda_j\ffrac{\partial F_j}{\partial x_s}(x)},$$ so that $\x_{|_{\S^k(\o)}}(x)=0$ if and only if $N_s(x,a,\lambda)=0$, for all $s=1,\ldots,m$. 
 
Let $F:\mathcal{U}\times\R^n\setminus\{\vec{0}\}\times\R^{m-n+k}\rightarrow\R^{2m-n+k+2}$ be the mapping defined by $$F(x,a,\lambda)=(F_1,\ldots, F_{m-n+k+2}, N_1,\ldots,N_m),$$ if $(x,a,\lambda)\in F^{-1}(\vec{0})$ then $x\in Z(\x_{|_{\S^k(\o)}})\cap\S^{k+2}(\o)$ and the Jacobian matrix of $F$ at $(x,a,\lambda)$: \begin{equation*}\left[\setlength{\arraycolsep}{0.12cm}{\renewcommand{\arraystretch}{1}{\begin{array}{ccccc}
			  dF_1(x) & \vdots & && \\
			  \vdots          & \vdots &\multicolumn{3}{c}{O_{(m-n+k+2)\times(m+k)}} \\
 			  dF_{m-n+k+2}(x) & \vdots & && \\
        %\nabla_x\d_2(x) & \vdots  &&& \\
        %\vdots          & \vdots  &&& \\
 			  %\nabla_x\d_{k+2}(x) & \vdots &&& \\
\cdots \ \, \cdots \ \, \cdots \ \, \cdots \ \, \cdots   & \vdots &\cdots \  \, \cdots \  \, \cdots \  \, \cdots & \cdots & \cdots \  \, \cdots \  \, \cdots \  \, \cdots\\
d_xN_1(x,a,\lambda)  & \vdots  & & \vdots &  \\
\vdots       & \vdots  & B_{m\times n} &  \vdots & C_{m\times(m-n+k)}\\
d_xN_m(x,a,\lambda)  & \vdots   & &   \vdots    &    \\
\end{array}}}\right]
\end{equation*} has $\rank$ $2m-n+k+1$, where $O_{(m-n+k+2)\times(m+k)}$ is a null matrix, $B_{m\times n}$ is a matrix whose columns vectors are given by the coefficients of the 1-forms $\o_1(x), \ldots, \o_n(x)$ of the $n$-coframe $\o$:
\begin{equation*}B_{m\times n}=\left[\begin{array}{ccc}
 \o_1^1(x) & \cdots &  \o_n^1(x) \\
 %V_1^2(x) & \cdots & V_{n-1}^2(x) & V_n^2(x) \\
 \vdots   & \ddots    & \vdots \\
 %\o_1^{n-1}(x) & \cdots & \o_{n-1}^{n-1}(x) & \o_n^{n-1}(x) \\   
 %\o_1^{n}(x) & \cdots & \o_{n-1}^{n}(x) & \o_n^{n}(x) \\ 
  %\vdots   & \ddots & \vdots       & \vdots \\
 \o_1^{m}(x) & \cdots &  \o_n^{m}(x) \\    
\end{array}\right]
\end{equation*} and $C_{m\times(m-n+k)}$ is the matrix whose columns vectors are, up to sign, the coefficients of the derivatives $dF_1,\ldots,dF_{m-n+k}$ with respect to $x$:
\begin{equation*}C_{m\times(m-n+k)}=\left[\setlength{\arraycolsep}{0.12cm}{\begin{array}{ccc}
-\ffrac{\partial F_1}{\partial x_1}(x) & \cdots & -\ffrac{\partial F_{m-n+k}}{\partial x_1}(x) \\%& -\ffrac{\partial\d_2}{\partial x_1}(x) & \cdots & -\ffrac{\partial\d_k}{\partial x_1}(x)\\
 \vdots   & \ddots & \vdots \\       %& \vdots \\
-\ffrac{\partial F_{1}}{\partial x_m}(x) & \cdots & -\ffrac{\partial F_{m-n+k}}{\partial x_m}(x) \\ %& -\ffrac{\partial\d_2}{\partial x_m}(x) & \cdots & -\ffrac{\partial\d_k}{\partial x_m}(x)\\   
\end{array}}\right].
\end{equation*}

Note that, if $(x,a,\lambda)\in F^{-1}(\vec{0})$ then, in particular, $x\in\S^{k+1}(\o)$ and by Lemma \ref{dimensaointersecao}, $\dim(\langle\bar{\o}(x)\rangle\cap N_x^{\ast}\S^k(\o))=k$. Thus, $\dim(\langle\bar{\o}(x)\rangle + N_x^{\ast}\S^k(\o))=m-1$.
%is equal to
%$$\begin{array}{l}
%\dim(\langle\bar{\o}(x)\rangle)+\dim(N_x^{\ast}\S^k(\o))-\dim(\langle\bar{\o}(x)\rangle\cap N_x^{\ast}\S^k(\o))=m-1.
%%=(n-1)+m-(n-k)-k\\
%%=m-1.
%\end{array}$$ 
Therefore, 
\begin{equation*}\rank\left[\begin{array}{ccc}
                % & & \vdots &  \\
               B_{m\times n} &  \vdots & C_{m\times(m-n+k)}  \\
                 %&  & \vdots    &   \\
\end{array}\right]=m-1
\end{equation*} and the Jacobian matrix of $F$ at $(x,a,\lambda)$ has rank $2m-n+k+1$. That is, $F^{-1}(\vec{0})$ has dimension less or equal to $n-1$. Let $\pi:F^{-1}(\vec{0})\rightarrow\R^n\setminus\{\vec{0}\}$ be the projection over $\R^n\setminus\{\vec{0}\}$, that is, $\pi(x,a,\lambda)=a.$ By Sard's Theorem, $a$ is regular value of $\pi$ for almost every $a\in\R^n\setminus\{\vec{0}\}$. Therefore, $\pi^{-1}(a)\cap F^{-1}(\vec{0})=\emptyset$ for almost every $a\in\R^n\setminus\{\vec{0}\}$. 
However, \begin{equation*}
\pi^{-1}(a)\cap F^{-1}(\vec{0})=\{(x,a,\lambda)\in\mathcal{U}\times\{a\}\times\R^{m-n+k}\, | \, x\in Z(\x_{|_{\S^k(\o)}})\cap\S^{k+2}(\o)\}.
\end{equation*} Thus, $Z(\x_{|_{\S^k(\o)}})\cap\S^{k+2}(\o)=\emptyset$ for almost every $a\in\R^n\setminus\{\vec{0}\}$.
\end{proof}

\section{Non-degenerate zeros of a generic 1-form $\x(x)$ associated to a Morin $n$-coframe}\label{s3}

%\sectionmark{Zeros non-degenerate de $\x(x)$}

In this section we will verify that, generically, the 1-form $\x(x)$ and its restrictions $\x_{|_{\S^k(\o)}}$, $\x_{|_{A_k(\o)}}$ admit only non-degenerate zeros. Furthermore, we will see how these non-degenerate zeros 
%of $\x(x)$ and $\x_{|_{\S^k(\o)}}$ 
can be related.
%according to the singular set $\S^k(\o)$ ($k=1, \ldots,n$) they belong to. 
%its occurrence over the singular sets $\S^k(\o)$, $k=1, \ldots,n$. We present first the following technical lemma.
We start with some technical lemmas.

\begin{lema}\label{lematecnico} Let $A$ be a square matrix of order m given by: $$A=\left[\begin{array}{ccc}
a_{11} & \cdots & a_{1m}\\
a_{21} & \cdots & a_{2m}\\
\vdots & \cdots & \vdots\\
a_{m1} & \cdots & a_{mm}
\end{array}\right]. $$ If there exist $(\lambda_1, \ldots, \lambda_m)\in\R^m\setminus\{\vec{0}\}$ such that $\displaystyle\sum_{j=1}^{m}{\lambda_ja_{ij}}=0$, $i=1, \ldots, m,$ then $$\lambda_j\cof(a_{ik})-\lambda_k\cof(a_{ij})=0, \,\forall j,k=1,\ldots, m.$$
\end{lema}

\begin{lema}\label{lemadoscofatores} Let us consider the matrix 
\begin{equation*}M_i(x)=\left[\renewcommand{\arraystretch}{2}{\begin{array}{ccccc}
 \o_1^1(x) & \cdots & \o_{n-1}^1(x) & \o_n^1(x) \\
 \vdots   & \ddots & \vdots       & \vdots \\
 \o_1^{n-1}(x) & \cdots & \o_{n-1}^{n-1}(x) & \o_n^{n-1}(x) \\   
 \o_1^{i}(x) & \cdots & \o_{n-1}^{i}(x) & \o_n^{i}(x) \\  
\end{array}}\right].
\end{equation*} If $x$ is a zero of $\x$ then for $\l\in\{1, \ldots,n-1\}$, $j\in\{1, \ldots, n-1,i\}$ and $i\in\{n, \ldots, m\}$, we have $$a_n\cof(\o_{\l}^j, M_i)=a_{\l}\cof(\o_n^j, M_i).$$
\end{lema}

\begin{proof} This result is a consequence of Lemma \ref{lematecnico} applied to the matrix $A=M_i(x)$, where $a_{\l j}=\o_j^{\l}(x)$, for $j=1,\ldots,n$ and $\l=1,\ldots, n-1, i$. It is enough to take $(\lambda_1, \ldots, \lambda_n)=(a_1, \ldots, a_n)$. %, since $$\x(x)=\displaystyle\sum_{j=1}^{n}{a_{j}\o_{j}(x)}=\vec{0}\Rightarrow \displaystyle\sum_{j=1}^{n}{a_{j}\o_{j}^{\l}(x)}=0, \forall \l=1,\ldots, n-1, i.$$ 
\end{proof}

\begin{lema}\label{lemadorank} Let $\mathcal{U}\subset\R^m$ be an open set and let $H:\mathcal{U}\times\R^n\setminus\{\vec{0}\}\rightarrow\R^m$ be a smooth mapping given by $H(x,a)=(h_1(x,a),\ldots, h_m(x,a))$. If $$\rank(dh_1(x,a),\ldots, dh_m(x,a))=m, \forall (x,a)\in H^{-1}(\vec{0})$$ then $\rank(d_x h_1(x,a),\ldots, d_x h_m(x,a))=m$ for almost every $a\in\R^n\setminus\{\vec{0}\}$. 
\end{lema}

In the previous section we proved that every zero of $\x$ belongs to $\S^1(\o)$. Next, we will show that, generically, such zeros belong to $A_1(\o)$ and they are non-degenerate. To do this, we must find explicit equations that define locally the manifolds $T^{\ast}M^{n,n-1}$ and $\S^1(\o)$. %Let us keep the previous notation.

\begin{lema}\label{eqloc2} Let $(p,\tilde{\f})\in T^{\ast}M^{n,n-1}$, it is possible to exhibit, explicitly, functions $\m_i(x,\f):\tilde{\mathcal{U}}\rightarrow\R$, $i=n, \ldots, m$, defined on an open neighborhood $\tilde{\mathcal{U}}\subset T^{\ast}M^n$, with $(p,\tilde{\f})\in\tilde{\mathcal{U}}$, such that, locally $$T^{\ast}M^{n,n-1}=\left\{(x,\f)\in \tilde{\mathcal{U}} \ | \ \m_n=\ldots=\m_m=0\right\}$$ with $\rank\left(d\m_n,\ldots,d\m_m\right)=m-n+1$, for all $(x,\f)\in T^{\ast}M^{n,n-1}\cap\tilde{\mathcal{U}}$.
\end{lema}

\begin{proof} Let $(p,\tilde{\f})\in T^{\ast}M^{n,n-1}$, we may suppose without loss of generality that
\begin{equation*}\m(\f)=\left|\begin{array}{cccc}
\f_1^1&\f_2^1&\cdots &\f_{n-1}^1\\
\vdots & \vdots &\ddots &\vdots\\
\f_1^{n-1}&\f_2^{n-1}&\cdots &\f_{n-1}^{n-1}
\end{array}\right|\neq0
\end{equation*} for $(x,\varphi)$ in an open neighborhood $\tilde{\mathcal{U}}$ of $T^{\ast}M^n$, with $(p,\tilde{\f})\in\tilde{\mathcal{U}}$. In this situation, $T^{\ast}M^{n,n-1}$ can be locally defined as $$T^{\ast}M^{n,n-1}=\left\{(x,\f)\in \tilde{\mathcal{U}} \ | \ \m_n=\ldots=\m_m=0\right\}$$ where $\m_i:=\m_i(\f)$ is the determinant \begin{equation*}\m_i(\f)=\left|\begin{array}{ccccc}
\f_1^1&\f_2^1&\cdots &\f_{n-1}^1&\f_{n}^1\\
\vdots & \vdots &\ddots &\vdots &\vdots\\
\f_1^{n-1}&\f_2^{n-1}&\cdots &\f_{n-1}^{n-1}&\f_{n}^{n-1}\\
\f_1^{i}&\f_2^{i}&\cdots &\f_{n-1}^{i}&\f_{n}^{i}
\end{array}\right|, \ i=n,\ldots,m.
\end{equation*} %para $i=n,\ldots,m$. 
Let us verify that $\rank\left(d\m_n,\ldots,d\m_m\right)=m-n+1$ in $(T^{\ast}M^{n,n-1})\cap\tilde{\mathcal{U}}$. 

For clearer notations, consider $I=\{1,\ldots,n\}$ and $I_i=\{1,\ldots,n-1,i\}$ for each $i\in\{n,\ldots, m\}$. Then
\begin{equation}\label{gradmii}d\m_i(\f)=\displaystyle\sum_{j\in I, \l\in I_i}{\cof(\f_j^\l,\m_i)d \f_j^\l},\end{equation} where $\cof(\f_j^\l,\m_i)$ is the cofactor of $\f_j^\l$ in the matrix $$\left[\begin{array}{ccccc}
\f_1^1&\f_2^1&\cdots &\f_{n-1}^1&\f_{n}^1\\
\vdots & \vdots &\ddots &\vdots &\vdots\\
\f_1^{n-1}&\f_2^{n-1}&\cdots &\f_{n-1}^{n-1}&\f_{n}^{n-1}\\
\f_1^{i}&\f_2^{i}&\cdots &\f_{n-1}^{i}&\f_{n}^{i}
\end{array}\right]$$ and $$d \f_j^\l=\left(\frac{\partial \f_j^\l}{\partial \f_1^1},\ldots,\frac{\partial \f_j^\l}{\partial \f_1^m},\frac{\partial \f_j^\l}{\partial \f_2^1},\ldots,\frac{\partial \f_j^\l}{\partial \f_2^m},\ldots,\frac{\partial \f_j^\l}{\partial \f_n^1},\ldots,\frac{\partial \f_j^\l}{\partial \f_n^m}\right)$$ is the vector whose coordinate at the position $(j-1)m+\l$ is equal to $1$ and all the others are zero. In particular, since $i\in\{n, \ldots, m\}$, $$d \f_n^i=(0,\ldots,0,\underbrace{0,\ldots,\overset{i}{1},\ldots,0}_{m-n+1})\in\underbrace{(\R^m)^{\ast}\times\ldots\times(\R^m)^{\ast}}_{n \text{ times}}$$ and the $m-n+1$ last coordinates of $d \f_j^\l$ are zero for all $j\neq n$ or $\l\neq i$. Moreover, $\cof(\f_n^i,\m_i)=\m(\f)\neq0$, for $i=n,\ldots,m$. Thus, $$\frac{\partial(\m_n,\ldots,\m_m)}{\partial(\f_n^n,\ldots,\f_n^m)}=\left|\begin{array}{ccc}
\cof(\f_n^n,\m_n) &  \cdots & 0\\
\vdots &  \ddots & \vdots \\
0 & \cdots & \cof(\f_n^m,\m_m)\\
\end{array}\right|.$$ That is, for all $(x,\f)\in(T^{\ast}M^{n,n-1})\cap\tilde{\mathcal{U}}$, we have
\begin{equation}\label{menornaonulo}\frac{\partial(\m_n,\ldots,\m_m)}{\partial(\f_n^n,\ldots,\f_n^m)}=\m(\f)^{(m-n+1)}\left|
{\renewcommand{\arraystretch}{1}
\begin{array}{ccc}
1&  \cdots & 0\\
\vdots & \ddots & \vdots \\
0 & \cdots & 1\\
\end{array}}\right|\neq0.\end{equation} Therefore, $\rank(\m_n,\ldots,\m_m)=m-n+1$
%$$\rank\left[\begin{array}{c}
%d \m_n\\
%\vdots\\
%d \m_m
%\end{array}\right]=m-n+1$$ 
for all $(x,\f)\in(T^{\ast}M^{n,n-1})\cap\tilde{\mathcal{U}}$.
\end{proof}

%Seja $G(\o)=\{(x,\o_1(x),\ldots,\o_n(x))\ | \ x\in M\}$ o gráfico do $n$-coframe $\o$.

\begin{lema}\label{eqloc3} Let $p\in\S^1(\o)$ be a singular point of $\o$, it is possible to exhibit, explicitly, functions ${\normalfont\M_i(x):\mathcal{U}\rightarrow\R}$, $i=n, \ldots, m$, defined on an open neighborhood $\mathcal{U}\subset M$, with $p\in\mathcal{U}$, such that, locally $$\mathcal{U}\cap\S^1(\o)=\left\{x\in \mathcal{U} \ | \ {\normalfont \M_n(x)}=\ldots={\normalfont \M_m(x)=0}\right\}$$ with $\rank\left({\normalfont d\M_n(x),\ldots,d\M_m(x)}\right)=m-n+1$, for all $x\in \S^1(\o)\cap\mathcal{U}$.
\end{lema}

{\normalfont }

\begin{proof}
Let $\o=(\o_1,\ldots,\o_n)$ be a Morin $n$-coframe and let $p\in\Sigma^1(\o)$. By Remark \ref{obsmatrixM}, we can consider $\mathcal{U}\subset M$ an open neighborhood with $p\in\mathcal{U}$, where $\M(x)\neq0$. %then we can suppose without loss of generality that \begin{equation*}\M(x)=\left|
%{\renewcommand{\arraystretch}{2}
%\begin{array}{cccc}
%\o_1^1(x)&\o_2^1(x)&\cdots &\o_{n-1}^1(x)\\
%\vdots & \vdots &\ddots &\vdots\\
%\o_1^{n-1}(x)&\o_2^{n-1}(x)&\cdots &\o_{n-1}^{n-1}(x)
%\end{array}}\right|\neq0
%\end{equation*} in a neighborhood $\mathcal{U}\subset M$, with $p\in\mathcal{U}$. 
Thus, in this neighborhood the set $\S^1(\o)$ can be defined  as $$\mathcal{U}\cap\S^1(\o)=\{x\in\mathcal{U} \ | \ \M_n=\ldots=\M_m=0\},$$ where $\M_i:=\M_i(x)$ is the determinant \begin{equation}\label{expMi}\M_i(x)=\left|
{\renewcommand{\arraystretch}{2}
\begin{array}{ccccc}
\o_1^1(x)&\o_2^1(x)&\cdots &\o_{n-1}^1(x)&\o_{n}^1(x)\\
\vdots & \vdots &\ddots &\vdots &\vdots\\
\o_1^{n-1}(x)&\o_2^{n-1}(x)&\cdots &\o_{n-1}^{n-1}(x)&\o_{n}^{n-1}(x)\\
\o_1^{i}(x)&\o_2^{i}(x)&\cdots &\o_{n-1}^{i}(x)&\o_{n}^{i}(x)
\end{array}}\right|
\end{equation} for $i=n,\ldots,m$. Let us verify that $\rank\left(d\M_n(x),\ldots,d\M_m(x)\right)=m-n+1$, for all $x\in\S^1(\o)\cap\mathcal{U}$.

Let $G(\o)=\{(x,\o_1(x),\ldots,\o_n(x))\ | \ x\in M\}$ be the graph of the $n$-coframe $\o$. By Lemma \ref{eqloc2}, we can consider an open neighborhood $\tilde{\mathcal{U}}$ in $T^{\ast}M^n$ with $\left(p,\o(p)\right)\in \tilde{\mathcal{U}}$ and $\pi_x(\tilde{\mathcal{U}})=\mathcal{U}$, where $\pi_x:(\R^m)^n\rightarrow\R^m$ is the projection on the $m$ first coordinates, so that the manifolds $T^{\ast}M^{n,n-1}$ and $G(\o)$ can be locally defined as: $$T^{\ast}M^{n,n-1}=\{(x,\f)\in \tilde{\mathcal{U}} \ | \ \m_n=\ldots=\m_m=0\},$$ with $\rank\left(d \m_n,\ldots,d \m_m\right)=m-n+1$ on $ T^{\ast}M^{n,n-1}\cap\tilde{\mathcal{U}}$; and \begin{equation*}\setlength{\arraycolsep}{0.06cm}{\begin{array}{cl}
G(\o)&=\{(x,\f)\in \tilde{\mathcal{U}} \ | \ \f_j^{\l}=\o_j^{\l}(x); \ j=1,\ldots,n; \ \l=1,\ldots,m \}\\
&=\{(x,\f)\in \tilde{\mathcal{U}}\ | \ g_{\l j}(x,\f)=0; \ j=1,\ldots,n; \ \l=1,\ldots,m\}, 
\end{array}}
\end{equation*} with $\rank \left(d g_{11},\ldots,d g_{m1},\ldots,d g_{1n},\ldots,d g_{mn}\right)=nm$ on $G(\o)\cap \tilde{\mathcal{U}} $, where the functions $g_{\l j}:T^{\ast}M^n\rightarrow \R$ are given by $g_{\l j}(x,\f)=\f_j^{\l}-\o_j^{\l}(x)$.

Let $x\in\S^1(\o)\cap\mathcal{U}$, then $G(\o)\pitchfork T^{\ast}M^{n,n-1}$ at $(x,\o(x))\in \tilde{\mathcal{U}}$ and, at this point, $\rank \left(d \m_n,\ldots,d \m_m,d g_{11},\ldots,d g_{mn}\right)=m-n+1+nm$. That is, the matrix \begin{equation}\label{matrizJAc}\left[\begin{array}{c}
d\m_n\\
\vdots\\
d\m_m\\
d g_{11}\\
\vdots\\
d g_{mn}
\end{array}\right]=\left[\setlength{\arraycolsep}{0.1cm}{
\begin{array}{ccccccc}
      &                    & & \vdots &  & d_{\f}\m_n &       \\
      & O_{(m-n+1)\times m}& & \vdots &  & \vdots       &       \\
      &                    & & \vdots &  & d_{\f}\m_m &       \\
\cdots&\cdots \  \cdots \  \cdots  \ \cdots & \cdots & \vdots & \cdots & \cdots \ \cdots  &\cdots \\

      & -d_x \o_{1}^{1} & & \vdots &  &              &        \\
      & \vdots             & & \vdots &  &   Id_{(nm)}    &        \\
      & -d_x \o_{n}^{m} & & \vdots &  &              &        \\
\end{array}}\right]
\end{equation} has maximal rank at $(x,\o(x))$, where $O_{(m-n+1)\times m}$ is the null matrix of size $(m-n+1)\times m$, $Id_{(nm)}$ represents the identity matrix of order $nm$ and $d_x$ and $d_{\f}$ denote the derivatives with respect to $x=(x_1,\ldots,x_m)$ and $\f=(\f_1^1,\ldots,\f_1^m,\ldots,\f_n^1,\ldots,\f_n^m)$ respectively. 

We have that Equation (\ref{gradmii}) of Lemma \ref{eqloc2} is the derivative of $\m_i$ with respect to $\f$. Thus, we can write \begin{equation}\label{gradmi}
d\m_i(\f)=\displaystyle\sum_{j\in I, \l\in I_i}{\cof(\f_j^\l,\m_i)f_{\l j}}
\end{equation} for $I=\{1,\ldots,n\}$, $I_i=\{1,\ldots,n-1,i\}$ and $i=n,\ldots, m$. Where $f_{\l j}$ denotes the vector $f_{\l j}=(0,\ldots,0,0\ldots,0,1,0,\ldots,0)\in \underbrace{(\R^m)^{\ast}\times\ldots\times(\R^m)^{\ast}}_{n+1 \text{ times }}$ whose coordinates are all zero, except at the position $jm+\l$, for each $j\in I$ and $\l\in I_i$. 

We also have, \begin{equation*}
d\M_i(x)=\displaystyle\sum_{j\in I, \l\in I_i}{\cof(\o_j^\l(x),\M_i(x))d_x \o_j^{\l}(x)}
\end{equation*} and since $(x,\o(x))\in G(\o)\pitchfork T^{\ast}M^{n,n-1}$, we have $\o_j^{\l}(x)=\f_j^{\l}$ so that \begin{equation}\label{gradMi}
d\M_i(x)=\displaystyle\sum_{j\in I, \l\in I_i}{\cof(\f_j^{\l},\m_i)d_x \o_j^{\l}(x)}.
\end{equation}

Let us suppose that $\rank\left(d\M_n(x),\ldots,d\M_m(x)\right)<m-n+1$. Then, there exists $(\a_n,\ldots,\a_m)\neq(0,\ldots,0)$ such that $$\displaystyle\sum_{i=n}^{m}{\a_id\M_i(x)}=0.$$ Thus,
\begin{equation}\label{eqnula}0=\displaystyle\sum_{i=n}^{m}{\a_id\M_i(x)}\overset{\left(\ref{gradMi}\right)}{=}
\displaystyle\sum_{i=n}^{m}{\a_i\left[ \ \displaystyle\sum_{j\in I, \l\in I_i}{\cof(\f_j^{\l},\m_i)d_x \o_j^{\l}(x)} \ \right]}.
\end{equation} 
Let $d \tilde{\o}_j^{\l}(x)=(d_x \o_j^{\l}(x),0,\ldots,0)\in\underbrace{(\R^m)^{\ast}\times\ldots\times(\R^m)^{\ast}}_{n+1 \text{ times }}$, we have \begin{equation}\label{equaalpha}{\renewcommand{\arraystretch}{2.5}\begin{array}{rl}
\displaystyle\sum_{i=n}^{m}\a_i\left[ \displaystyle\sum_{j\in I, \l\in I_i}{\cof(\f_j^{\l},\m_i)d g_{\l j}}\right]\overset{\left(\ref{matrizJAc}\right)}{=}&\displaystyle\sum_{i=n}^{m}\a_i\left[\displaystyle\sum_{j\in I, \l\in I_i}\cof(\f_j^{\l},\m_i)\left(f_{\l j}-d\tilde{\o}_j^{\l}\right) \right]\\
\overset{\left(\ref{eqnula}\right)}{=}&\displaystyle\sum_{i=n}^{m}{\a_i\left[\displaystyle\sum_{j\in I, \l\in I_i}{\cof(\f_j^{\l},\m_i)f_{\l j}} \right]}\\
\overset{\left(\ref{gradmi}\right)}{=}&\displaystyle\sum_{i=n}^{m}{\a_id\m_i}.
\end{array}}
\end{equation} 
On the other hand, 
\begin{equation}\label{equabeta}
\displaystyle\sum_{i=n}^{m}\a_i\left[ \displaystyle\sum_{j\in I, \l\in I_i}{\cof(\f_j^{\l},\m_i)d g_{\l j}}\right]=\displaystyle\sum_{j\in I, \l\in \{1,\ldots,m\}}\beta_{\l j}d g_{\l j}
\end{equation} where 
\begin{equation*}\beta_{\l j}=\left\{
\begin{array}{ll}
\displaystyle\sum_{i=n}^{m}{\a_i\cof(\f_j^{\l},\m_i)}, & j\in I, \l=1,\ldots,n-1;\\
\a_{\l}\cof(\f_j^{\l},\m_{\l}), & j\in I, \l=n,\ldots,m.
\end{array}\right.
\end{equation*} Since $(\a_n,\ldots,\a_m)\neq(0,\ldots,0)$, by Equations (\ref{equaalpha}) and (\ref{equabeta}), we obtain
\begin{equation*}
\displaystyle\sum_{j, \l}\beta_{\l j}d g_{\l j}-\displaystyle\sum_{i=n}^{m}{\a_id\m_i}=0
\end{equation*} which is a linear combination (with non-zero coefficients) of the row vectors of the matrix (\ref{matrizJAc}). This is a contradiction, since $\rank \left(d \m_n,\ldots,d \m_m,d g_{11},\ldots,d g_{mn}\right)$ is maximal. Therefore $\rank\left(d\M_n(x),\ldots,d\M_m(x)\right)=m-n+1.$
\end{proof}

\begin{lema}\label{nondegeneratexi} For almost every $a\in\R^n\setminus\{\vec{0}\}$, the 1-form $\x(x)=\displaystyle\sum_{i=1}^{n}{a_i\o_i(x)}$ admits only non-degenerate zeros. Moreover, such zeros belong to $A_1(\o)$.
\end{lema}

\begin{proof} Suppose that $p\in M$ is a zero of $\x$. Then, by Lemmas \ref{zeroszsobresigma1} and \ref{lemainterzeroscomsigma2}, for almost every $a\in\R^n\setminus\{\vec{0}\}$ we have that $p\in\S^1(\o)\setminus\S^2(\o)$, that is, $p\in A_1(\o)$. Assume that $\M(x)\neq0$ in an open neighborhood $\mathcal{U}\subset M$ of $p$ (see Remark \ref{obsmatrixM}) so that $\mathcal{U}\cap\S^1(\o)=\{x\in\mathcal{U}: \M_n(x)=\ldots=\M_m(x)=0\}$. % com $\rank(d\M_n(x), \ldots,d\M_m(x))=m-n+1$.
Let us write $$\x_s(x)=\displaystyle\sum_{i=1}^{n}{a_i\o_i^s(x)}, \, s= 1,\ldots, m$$ and let us consider the mapping $F:\mathcal{U}\times\Rnm\rightarrow\R^m$ defined by $$F(x,a)=(\M_n(x),\ldots, \M_m(x), \x_1(x), \ldots, \x_{n-1}(x)).$$ Its Jacobian matrix at a point $(x,a)$ is given by:
\begin{equation*}\Jac F(x,a)=\left[\setlength{\arraycolsep}{0.1cm}{\begin{array}{cccccccc}
			  d_x\M_n(x) & \vdots & &&& \\
			  \vdots          & \vdots & &\multicolumn{2}{c}{O_{(m-n)\times n}}& \\
 			  d_x\M_m(x) & \vdots & &&& \\
 \cdots \ \ \cdots \ \ \cdots& \vdots & \cdots \ \ \cdots & \cdots & \cdots \ \ \cdots \ \ \cdots& \cdots \ \ \cdots \\
        d_x\x_1(x)  & \vdots & \o_1^1(x) & \cdots & \o_{n-1}^1(x) & \o_n^1(x) \\
                 \vdots & \vdots & \vdots   & \ddots & \vdots       & \vdots \\
       d_x\x_{n-1}(x) & \vdots & \o_1^{n-1}(x) & \cdots & \o_{n-1}^{n-1}(x) & \o_n^{n-1}(x) \\   
\end{array}}\right].
\end{equation*} Note that, by Lemma {\ref{lemaauxiliar}}, $F^{-1}(\vec{0})$ corresponds to the zeros of $\x$ on $\S^1(\o)\cap\mathcal{U}$. Since $\M(x)\neq0$ and $\rank(d\M_n(x), \ldots,d\M_m(x))=m-n+1$ for all $x\in\S^1(\o)\cap\mathcal{U}$, then $\rank(\Jac F(x,a))=m$ for all $(x,a)\in F^{-1}(\vec{0})$. Thus, $\dim F^{-1}(\vec{0})=n$. 

Let $\pi:F^{-1}(\vec{0})\rightarrow\Rnm$ be the projection $\pi(x,a)=a$, by Sard's Theorem, almost every $a\in\Rnm$ is a regular value of $\pi$ and $\dim(\pi^{-1}(a)\cap F^{-1}(\vec{0}))=0$. That is, for almost every $a$, the zeros of $\x$ are isolated in $\S^1(\o)$. Let us proof that, moreover, these zeros are non-degenerate.

Since $\rank(\Jac F(x,a))=m$, for all $(x,a)\in F^{-1}(\vec{0})$, then by Lemma \ref{lemadorank} we have that $\rank(d_x\M_n(p), \ldots, d_x\M_m(p), d_x\x_1(p), \ldots, d_x\x_{n-1}(p))=m$, which happens if and only if $\rank(B)=m$, where $B$ is the matrix
\begin{equation*}B=\left[\begin{array}{c}
        d_x\x_1(p)  \\
                 \vdots \\
       d_x\x_{n-1}(p) \\ 
        a_nd_x\M_n(p)  \\
			  \vdots          \\
 			  a_nd_x\M_m(p) \\  
\end{array}\right]
\end{equation*} whose row vectors we will denote by $R_i, i=1, \ldots,m$ (by Remark \ref{obsmatrixM}, $a_n\neq0$). 

Let us denote $I=\{1,\ldots,n\}$ and $I_i=\{1,\ldots,n-1,i\}$ for each $i\in\{n,\ldots, m\}$. By Equation (\ref{expMi}), we can write $$d \M_i(x)=\displaystyle\sum_{\l\in I, j\in I_i}{\cof(\o_{\l}^j(x),M_i)d \o_{\l}^j(x)}$$ and by Lemma \ref{lemadoscofatores}, $$d \M_i(p)=\displaystyle\sum_{\l\in I, j\in I_i}{\ffrac{a_{\l}}{a_n}\cof(\o_{n}^j(p),M_i)d \o_{\l}^j(p)}.$$ Thus, 
\begin{equation*}\renewcommand{\arraystretch}{2}{\begin{array}{ccl}
a_nd \M_i(p)&=&\displaystyle\sum_{\l\in I, j\in I_i}{a_{\l}\cof(\o_{n}^j(p),M_i)d \o_{\l}^j(p)}\\
                 &=&\displaystyle\sum_{j\in I_i}{\cof(\o_{n}^j(p),M_i)\left[\displaystyle\sum_{\l\in I}a_{\l}d \o_{\l}^j(p)\right]}\\
                 &=&\displaystyle\sum_{j\in I_i}{\cof(\o_{n}^j(p),M_i)\left[d_x\x_j(p)\right]}\\
                 &=&\cof(\o_{n}^i(p),M_i)\left[d_x\x_i(p)\right]+\displaystyle\sum_{j\in I_i\setminus\{i\}}{\cof(\o_{n}^j(p),M_i)\left[d_x\x_j(p)\right]}.
\end{array}}
\end{equation*} Note that, $\cof(\o_{n}^i(p),M_i)=\M(p)\neq0$, for all $i=n, \ldots,m$. Then, for each $i=n, \ldots,m$, we replace the $i^{th}$ row $R_i$ of matrix $B$ by $$\ffrac{1}{\cof(\o_n^i(p),M_i)}\left(R_i-\sum_{j=1}^{n-1}{\cof(\o_n^j(p),M_i)R_j}\right)$$ so that we obtain the matrix of maximal rank:
\begin{equation*}\left[\begin{array}{c}
        d_x\x_1(p)  \\
                 \vdots \\
       d_x\x_{n-1}(p) \\ 
        d_x\x_n(p)  \\
			  \vdots          \\
 			 d_x\x_m(p) \\  
\end{array}\right].
\end{equation*} Therefore, the zeros of $\x(x)$ are non-degenerate. \end{proof}

\begin{lema}\label{zerosnaodegdeak} For almost every $a\in\R^n\setminus\{\vec{0}\}$, the 1-form $\x_{|_{A_k(\o)}}$ admits only non-degenerate zeros, $k\geq2$.
\end{lema} 

\begin{proof} Suppose that $\x_{|_{A_k(\o)}}(p)=0$. By Lemmas \ref{eqlocalSk} and \ref{eqloc3}, we can consider $\mathcal{U}\subset M$ an open neighborhood of $p$ where $\M(x)\neq0$ and on which the respective singular sets $(k=2, \ldots, n)$ can be locally defined as  
\begin{equation*}\setlength{\arraycolsep}{0.06cm}{\begin{array}{lll}
\mathcal{U}\cap\S^1(\o)&=&\{x\in\mathcal{U}: \M_n(x)=\ldots=\M_m(x)=0\},\\ 
\mathcal{U}\cap\S^k(\o)&=&\{x\in\mathcal{U}: \M_n(x)=\ldots=\M_m(x)=\d_2(x)=\ldots=\d_k(x)=0\},
%\S^{k+1}(\o)&=&\{x\in\mathcal{U}: \M_n(x)=\ldots=\M_m(x)=\d_2(x)=\ldots=\d_{k+1}(x)=0\},
\end{array}}
\end{equation*} with
\begin{equation*}\setlength{\arraycolsep}{0.06cm}{\begin{array}{l}
\rank(d\M_n, \ldots,d\M_m)=m-n+1, \forall x\in\S^1(\o)\cap\mathcal{U},\\
\rank(d\M_n, \ldots,d\M_m,d\d_2,\ldots,d\d_k)=m-n+k, \forall x\in\S^k(\o)\cap\mathcal{U}.
%\rank(d\M_n, \ldots,d\M_m,d\d_2,\ldots,d\d_{k+1})=m-n+k+1, \ x\in\S^{k+1}(\o)\cap\mathcal{U}.
\end{array}}
\end{equation*} 

Analogously to the proof of Lemma \ref{ptscrticrestasigmak}, by Szafraniec's characterization (see \cite[p. 196]{sza2}), $x$ is a zero of the restriction $\x_{|_{\S^k(\o)}}$ if and only if there exists $(\lambda_n,\ldots,\lambda_m,\beta_2,\ldots,\beta_k)\in\R^{m-n+k}$ such that $$\x(x)=\displaystyle\sum_{j=n}^{m}{\lambda_jd\M_j(x)}+\displaystyle\sum_{\l=2}^{k}{\beta_{\l}d\d_{\l}(x)}.$$ Let us consider the functions $$N_s(x,a,\lambda,\beta):=\x_s(x)-\displaystyle\sum_{j=n}^{m}{\lambda_j\ffrac{\partial\M_j}{\partial x_s}(x)}-\displaystyle\sum_{\l=2}^{k}{\beta_{\l}\ffrac{\partial\d_{\l}}{\partial x_s}(x)}, \ s=1,\ldots, m,$$ so that $\x_{|_{\S^k(\o)}}(x)=0$ if and only if $N_s(x,a,\lambda,\beta)=0$, for all $s=1,\ldots,m$.

Let $G:\mathcal{U}\setminus\{\d_{k+1}=0\}\times\Rnm\times\R^{m-n+k}\rightarrow\R^{2m-n+k}$ be the mapping $$G(x,a,\lambda,\beta)=(\M_n,\ldots,\M_m,\d_2,\ldots,\d_{k},N_1,\ldots,N_m).$$ Its Jacobian matrix at a point $(x,a,\lambda,\beta)\in G^{-1}(\vec{0})$ is given by:
\begin{equation*}\Jac G(x,a,\lambda,\beta)=\left[\setlength{\arraycolsep}{0.12cm}{\begin{array}{ccccc}
			  d_x\M_n(x) & \vdots & \\
			  \vdots          & \vdots & && \\
 			  d_x\M_m(x) & \vdots &\multicolumn{3}{c}{\multirow{2}{*}{$O_{(m-n+k)\times(m+k)}$}} \\
        d_x\d_2(x) & \vdots & && \\
        \vdots          & \vdots & && \\
 			  d_x\d_{k}(x) & \vdots & && \\
\cdots \ \cdots \ \cdots \ \cdots \ \cdots  & \vdots & \cdots \ \cdots & \cdots &\cdots \ \cdots \ \cdots \ \cdots \\                      
d_xN_1(x,a,\lambda,\beta)          & \vdots &  &  \vdots & \\
        \vdots       & \vdots & \ B_{m\times n} &  \vdots & C_{m\times(m-n+k)} \\
 d_xN_m(x,a,\lambda,\beta)       & \vdots  & &   \vdots    &    
\end{array}}\right]
\end{equation*} where $O_{(m-n+k)\times(m+k)}$ is a null matrix, $B_{m\times n}$ is the matrix whose column vectors are given by the coefficients of the 1-forms $\o_1(x), \ldots, \o_n(x)$ 
%\begin{equation*}B_{m\times n}=\left[\begin{array}{cccc}
 %\o_1^1(x) & \cdots & \o_{n-1}^1(x) & \o_n^1(x) \\
 %\vdots   & \ddots & \vdots       & \vdots \\
 %\o_1^{m}(x) & \cdots & \o_{n-1}^{m}(x) & \o_n^{m}(x) \\    
%\end{array}\right]
%\end{equation*} 
and $C_{m\times(m-n+k)}$ is the matrix whose column vectors are, up to sign, the coefficients of the derivatives $d\M_n,\ldots,d\M_m,d\d_2,\ldots,d\d_{k}$ with respect to $x$:
\begin{equation*}C_{m\times(m-n+k)}=\left[\setlength{\arraycolsep}{0.1cm}{\begin{array}{cccccc}
-\ffrac{\partial\M_n}{\partial x_1}(x) & \cdots & -\ffrac{\partial\M_m}{\partial x_1}(x) & -\ffrac{\partial\d_2}{\partial x_1}(x) & \cdots & -\ffrac{\partial\d_k}{\partial x_1}(x)\\
 \vdots   & \ddots & \vdots       & \vdots \\
-\ffrac{\partial\M_n}{\partial x_m}(x) & \cdots & -\ffrac{\partial\M_m}{\partial x_m}(x) & -\ffrac{\partial\d_2}{\partial x_m}(x) & \cdots & -\ffrac{\partial\d_k}{\partial x_m}(x)\\   
\end{array}}\right].
\end{equation*} Thus, if $(x,a,\lambda,\beta)\in G^{-1}(\vec{0})$ then $x\in\S^{k}(\o)\cap\mathcal{U}$, $\d_{k+1}(x)\neq0$ and $\x_{|_{\S^k(\o)}}(x)=0$. And since $A_k(\o)=\S^k(\o)\setminus\S^{k+1}(\o)$, we have $x\in A_k(\o)\cap Z(\x_{|_{\S^k(\o)}})$, for all $(x,a,\lambda,\beta)\in G^{-1}(\vec{0}).$

On the other hand, if $x\in A_k(\o)$ then $\dim(\langle\bar{\o}(x)\rangle \cap N_x^{\ast}\S^{k-1}(\o))=k-1$ and $\dim(\langle\bar{\o}(x)\rangle \cap N_x^{\ast}\S^{k}(\o))=k-1$, so that $\dim(\langle\bar{\o}(x)\rangle + N_x^{\ast}\S^{k}(\o))=m.$ Thus,
%\begin{equation*}\left\{\begin{array}{l}
%\dim(\langle\bar{\o}(x)\rangle \cap N_x^{\ast}\S^{k-1}(\o))=k-1;\\
%\dim(\langle\bar{\o}(x)\rangle \cap N_x^{\ast}\S^{k}(\o))=k-1. 
%\end{array}\right.
%\end{equation*} %Como $\dim(N_x\S^{k-1}(\o))=m-n+k-1$, $\dim(N_x\S^{k}(\o))=m-n+k$ e $N_x\S^{k-1}(\o)\subset N_x\S^{k}(\o)$, então $$\dim(\langle\o(x)\rangle \cap N_x\S^{k}(\o))=\dim(\langle\o(x)\rangle \cap N_x\S^{k-1}(\o))=k-1.$$ 
%Thus, $\dim(\langle\bar{\o}(x)\rangle + N_x^{\ast}\S^{k}(\o))=(n-1)+(m-n+k)-(k-1)=m.$ Therefore, 
\begin{equation*}\rank\left[\begin{array}{ccccccc}
                d_xN_1(x,a,\lambda,\beta)        & \vdots & & & & \vdots &  \\
               \vdots        & \vdots & &B_{m\times n} &  & \vdots & C_{m\times(m-n+k)}  \\
                d_xN_m(x,a,\lambda,\beta)        & \vdots &  & &  & \vdots    &   \\
\end{array}\right]=m
\end{equation*} and the Jacobian matrix of $G$ has maximal rank at every $(x,a,\lambda,\beta)\in G^{-1}(\vec{0})$. Therefore, $\dim G^{-1}(\vec{0})=(2m+k)-(2m-n+k)=n$. Let $\pi:G^{-1}(\vec{0})\rightarrow\R^n\setminus\{\vec{0}\}$ be the projection $\pi(x,a,\lambda,\beta)=a$, then almost every $a\in\R^n\setminus\{\vec{0}\}$ is a regular value of $\pi$. So, for almost every $a\in\R^n\setminus\{\vec{0}\}$, $\dim(\pi^{-1}(a)\cap G^{-1}(\vec{0}))=0$ and $\pi^{-1}(a)\pitchfork G^{-1}(\vec{0})$. Therefore, the zeros of $\x_{|_{A_k(\o)}}$ are non-degenerate.
\end{proof}

\begin{lema}\label{zerosnaodegdeaum} For almost every $a\in\R^n\setminus\{\vec{0}\}$, the 1-form $\x_{|_{A_1(\o)}}$ admits only non-degenerate zeros.
\end{lema}

\begin{proof} This proof follows analogously the proof of Lemma \ref{zerosnaodegdeak}.
\end{proof}

By Lemma \ref{lemazerosrestricoes}, if $p\in A_{k+1}(\o)$, then $p$ is a zero of $\x_{|_{\S^{k+1}(\o)}}$ if and only if $p$ is a zero of $\x_{|_{\S^k(\o)}}$. The next results state that this relation also holds for non-degenerate zeros.

\begin{lema}\label{nondegenerateequivalenceA1} Let $p\in A_1(\o)$ be a zero of $\x_{|_{\S^{1}(\o)}}$, then $p$ is a non-degenerate zero of $\x_{|_{\S^{1}(\o)}}$ if and only if $p$ is a non-degenerate zero of $\x$.
\end{lema}

\begin{proof} Let $p\in A_1(\o)$ be a zero of the restriction $\x_{|_{\S^{1}(\o)}}$ and let $\mathcal{U}\subset M$ be an open neighborhood of $p$ at which $\M(x)\neq0$ 
%\begin{equation*}\M(x)=\left|
%{\renewcommand{\arraystretch}{2}
%\begin{array}{ccc}
%\o_1^1(x)&\cdots &\o_{n-1}^1(x)\\
%\vdots & \ddots &\vdots\\
%\o_1^{n-1}(x)&\cdots &\o_{n-1}^{n-1}(x)
%\end{array}}\right|\neq0.
%\end{equation*} 
and $\mathcal{U}\cap\S^1(\o)=\{x\in\mathcal{U}: \M_n(x)=\ldots=\M_m(x)=0\}$. 
By Szafraniec's characterization (\cite[p. 196]{sza2}), $\exists!(\lambda_n, \ldots,\lambda_m)\in\R^{m-n+1}$, such that $$\x(p)+\displaystyle\sum_{i=n}^{m}{\lambda_id\M_i(p)}=0.$$ Furthermore, $p$ is a non-degenerate zero of $\x_{|_{\S^{1}(\o)}}$ if and only if the matrix
\begin{equation}\label{matrizinicial}\left[\renewcommand{\arraystretch}{1.7}{\setlength{\arraycolsep}{0.1cm}{\begin{array}{ccccc}
                                                    & \vdots & \ffrac{\partial \M_n}{\partial x_1}(p) & \cdots & \ffrac{\partial \M_m}{\partial x_1}(p) \\
		\Jac\left(\x+\displaystyle\sum_{i=n}^{m}{\lambda_id\M_i}\right)(p) & \vdots & \vdots & \ddots & \vdots \\
                                                            & \vdots & \ffrac{\partial \M_n}{\partial x_m}(p) & \cdots & \ffrac{\partial \M_m}{\partial x_m}(p) \\
 \cdots \ \ \cdots \ \ \cdots \ \ \cdots \ \ \cdots \ \ \cdots  & \vdots  &\cdots \ \ \cdots&\cdots \ \ \cdots&\cdots \ \ \cdots\\		
			  d_x\M_n(p) & \vdots & & & \\
			  \vdots          & \vdots & & O_{(m-n+1)} & \\
 			  d_x\M_m(p) & \vdots & & & \\                      
\end{array}}}\right]
\end{equation} is non-singular. Since $\x(p)=0$, then $p\in\S^1(\o)\cap\mathcal{U}$ and $\lambda_nd\M_n(p)+\ldots+\lambda_md\M_m(p)=\vec{0}$, thus, $\lambda_n=\ldots=\lambda_m=0$ and, writing $\x=(\x_1, \ldots, \x_m)$, we have $$\Jac\left(\x+\displaystyle\sum_{i=n}^{m}{\lambda_id\M_i}\right)(p)=\left[\begin{array}{c}d_x \x_1(p)\\
\vdots\\
d_x \x_m(p)
\end{array}\right].$$ This means that the Matrix (\ref{matrizinicial}) is non-singular if and only if the matrix
\begin{equation}\label{matrizinicialxan}\left[\setlength{\arraycolsep}{0.1cm}{\renewcommand{\arraystretch}{1.6}{\begin{array}{ccccc}
d_x\x_1(p)           & \vdots & \ffrac{\partial \M_n}{\partial x_1}(p) & \cdots & \ffrac{\partial \M_m}{\partial x_1}(p) \\
		\vdots & \vdots & \vdots & \ddots & \vdots \\
  d_x\x_m(p)     & \vdots & \ffrac{\partial \M_n}{\partial x_m}(p) & \cdots & \ffrac{\partial \M_m}{\partial x_m}(p) \\
 \cdots \ \ \cdots \ \ \cdots \ \ \cdots  & \vdots  &\cdots \ \ \cdots&\cdots \ \ \cdots&\cdots \ \ \cdots\\		
			  a_nd_x\M_n(p) & \vdots & & & \\
			  \vdots          & \vdots & & O_{(m-n+1)} & \\
 			  a_nd_x\M_m(p) & \vdots & & & \\                      
\end{array}}}\right]
\end{equation} is non-singular (by Remark \ref{obsmatrixM}, $a_n\neq0$). By Equation (\ref{expMi}), $$d_x \M_i(x)=\displaystyle\sum_{\l\in I, j\in I_i}{\cof(\o_{\l}^j(x),M_i)d \o_{\l}^j(x)},$$ and applying %como $\x(p)=\vec{0}$, $a_n\neq0$ e 
Lemma \ref{lemadoscofatores}, we obtain
\begin{equation*}\renewcommand{\arraystretch}{2}{\begin{array}{cl}
a_nd_x \M_i(p)&=\displaystyle\sum_{\l\in I, j\in I_i}{a_{\l}\cof(\o_{n}^j(p),M_i)d \o_{\l}^j(p)}\\
                 &=\displaystyle\sum_{j\in I_i}{\cof(\o_{n}^j(p),M_i)\left[\displaystyle\sum_{\l\in I}a_{\l}d \o_{\l}^j(p)\right]}\\
                 &=\displaystyle\sum_{j\in I_i}{\cof(\o_{n}^j(p),M_i)\left[d_x\x_j(p)\right]}.
\end{array}}
\end{equation*}

Let us denote the $m$ first row vectors of Matrix (\ref{matrizinicialxan}) by $L_j, j=1, \ldots, m,$ and let us denote the $m-n+1$ last row vectors of Matrix (\ref{matrizinicialxan}) by $R_i, i=n, \ldots, m$: $$\renewcommand{\arraystretch}{3}{\begin{array}{ccl}
L_j&=&\left(d_x \x_j(p), \ffrac{\partial\M_n}{\partial x_j}(p), \ldots, \ffrac{\partial\M_m}{\partial x_j}(p)\right);\\
R_i&=&\left(a_n\ffrac{\partial\M_i}{\partial x_1}(p), \ldots, a_n\ffrac{\partial\M_i}{\partial x_m}(p), \vec{0}\right).
\end{array}}$$ We replace each row vector $R_i$, $i=n, \ldots, m$, by $R_i-\sum_{j\in I_i}{\cof(\o_n^j, M_i)L_j}$ so that we obtain $$R_i=\left(\underbrace{0,\ldots\ 0}_{m \text{ times }},-\sum_{j\in I_i}{\cof(\o_n^j,M_i)\ffrac{\partial\M_n}{\partial x_j}}, \ldots, -\sum_{j\in I_i}{\cof(\o_n^j,M_i)\ffrac{\partial\M_m}{\partial x_j}}\right)$$ and the Matrix (\ref{matrizinicialxan}) becomes:
\begin{equation}\label{matrizcomMlinha}\left[\setlength{\arraycolsep}{0.1cm}{\begin{array}{ccccc}
d_x\x_1(p)           & \vdots & \ffrac{\partial \M_n}{\partial x_1}(p) & \cdots & \ffrac{\partial \M_m}{\partial x_1}(p) \\
		\vdots & \vdots & \vdots & \ddots & \vdots \\
  d_x\x_m(p)     & \vdots & \ffrac{\partial \M_n}{\partial x_m}(p) & \cdots & \ffrac{\partial \M_m}{\partial x_m}(p) \\
 \cdots \ \ \cdots \ \ \cdots  & \vdots  &\cdots \ \ \cdots&\cdots \ \ \cdots&\cdots \ \ \cdots\\		
			  & \vdots & & & \\
			  O_{(m-n+1)\times m}          & \vdots & & \M'_{(m-n+1)} & \\
 			  & \vdots & & & \\                      
\end{array}}\right]
\end{equation} where $$\M'_{(m-n+1)}=-\left(\begin{array}{c}m_{ij}\end{array}\right)_{n\leq i,j\leq m}$$ is the matrix given by \begin{equation}\label{eqdosmij} m_{ij}=\sum_{k\in I_i}{\cof(\o_n^k,M_i)\ffrac{\partial \M_j}{\partial x_k}}, \, i,j=n, \ldots, m.\end{equation} 

Next, we will verify that the matrix $\M'$ is non-singular. Since $p\in A_1(\o)$, then $\dim(\langle\bar{\o}(p)\rangle \cap N_p^{\ast}\S^1(\o))=0$ and $\dim(\langle\bar{\o}(p)\rangle \oplus N_p^{\ast}\S^1(\o))=m.$ Since $\M(p)\neq0$, $\{\o_1(p), \ldots, \o_{n-1}(p)\}$ is a basis of the space $\langle\bar{\o}(p)\rangle $. Thus the matrix 
\begin{equation}\label{matrizdeVseMis}\left[\renewcommand{\arraystretch}{2}{\begin{array}{cccccc}
\o_1^1(p)  & \cdots & \o_1^{n-1}(p)& \o_1^{n}(p)& \cdots & \o_1^{m}(p)\\
\vdots & \ddots & \vdots   & \vdots & \ddots & \vdots\\
\o_{n-1}^1(p)  & \cdots & \o_{n-1}^{n-1}(p)& \o_{n-1}^{n}(p)& \cdots & \o_{n-1}^{m}(p)\\
\ffrac{\partial\M_n}{\partial x_1}(p) & \cdots & \ffrac{\partial\M_n}{\partial x_{n-1}}(p) & \ffrac{\partial\M_n}{\partial x_n}(p) & \cdots & \ffrac{\partial\M_n}{\partial x_m}(p)\\
\vdots & \ddots & \vdots & \vdots & \ddots & \vdots\\
\ffrac{\partial\M_m}{\partial x_1}(p) & \cdots & \ffrac{\partial\M_m}{\partial x_{n-1}}(p) & \ffrac{\partial\M_m}{\partial x_n}(p) & \cdots & \ffrac{\partial\M_m}{\partial x_m}(p)
\end{array}}\right]
\end{equation} has rank maximal. Let us denote the row vectors of Matrix (\ref{matrizdeVseMis}) by $L_j', j=1, \ldots, m$. For $j=1, \ldots, n-1$, we replace $L_j'$ by \begin{equation}\label{novaslinhas} \sum_{k=1}^{n-1}{\cof(\o_k^{j},M)L_k'}=\left(\sum_{k=1}^{n-1}{\cof(\o_k^{j},M)\o_k^1}, \ldots, \sum_{k=1}^{n-1}{\cof(\o_k^{j},M)\o_k^m}\right),\end{equation} where  
\begin{equation*} \sum_{k=1}^{n-1}{\cof(\o_k^j,M)\o_k^{\l}}=\left\{\begin{array}{ll}
\M, & \l=j;\\
0  & \l=1, \ldots,n-1 \text{ and } \l\neq j;\\
-\cof(\o_n^{j},\M_{\l}), & \l=n, \ldots, m.
\end{array}\right.
\end{equation*} Indeed,
\begin{itemize}
\item For $\l=1, \ldots,n-1$ with $\l=j$, we have: $$\sum_{k=1}^{n-1}{\cof(\o_k^{j},M)\o_k^j}=\left| \setlength{\arraycolsep}{0.12cm}{\renewcommand{\arraystretch}{1.5}{\begin{array}{ccccc}
\o_1^1 & \cdots & \o_k^1 & \cdots & \o_{n-1}^1\\
\vdots & \ddots & \vdots & \ddots & \vdots\\
\o_1^{j} & \cdots & \o_k^{j} & \cdots & \o_{n-1}^{j}\\
\vdots & \ddots & \vdots & \ddots & \vdots\\
\o_1^{n-1} & \cdots & \o_k^{n-1} & \cdots & \o_{n-1}^{n-1}\\
\end{array}}}\right|=\M;$$ 
\item For $\l=1, \ldots,n-1$ and $\l\neq j$, we have: $$\sum_{k=1}^{n-1}{\cof(\o_k^{j},M)\o_k^{\l}}=\left| \setlength{\arraycolsep}{0.12cm}{\begin{array}{ccccc}
\o_1^1 & \cdots & \o_k^1 & \cdots & \o_{n-1}^1\\
\vdots & \ddots & \vdots & \ddots & \vdots\\
\o_1^{j-1} & \cdots & \o_k^{j-1} & \cdots & \o_{n-1}^{j-1}\\
\o_1^{\l} & \cdots & \o_k^{\l} & \cdots & \o_{n-1}^{\l}\\
\o_1^{j+1} & \cdots & \o_k^{j+1} & \cdots & \o_{n-1}^{j+1}\\
\vdots & \ddots & \vdots & \ddots & \vdots\\
\o_1^{n-1} & \cdots & \o_k^{n-1} & \cdots & \o_{n-1}^{n-1}\\
\end{array}}\right|=0,$$ because this is the determinant of a matrix with two equal rows.
%corresponds to the determinant of Matrix M after replacing the $j^{th}$ line $(\o_1^{\l}, \ldots, \o_{n-1}^{\l})$ by some $\l\in\{1, \ldots, n-1\}\setminus\{j\}$.
%That is, this is the determinant of a matrix that contains two equal lines.
\item For $\l=n, \ldots, m$, we have:  $$\cof(\o_n^{j},\M_{\l})=(-1)^{n+j}\left|\renewcommand{\arraystretch}{1.3}{\setlength{\arraycolsep}{3mm}{\begin{array}{ccc}
\o_1^1 & \cdots & \o_{n-1}^1\\
\vdots & \ddots & \vdots \\
\o_1^{j-1} & \cdots & \o_{n-1}^{j-1}\\
\o_1^{j+1} & \cdots & \o_{n-1}^{j+1}\\
\vdots & \ddots & \vdots\\
\o_1^{n-1} & \cdots & \o_{n-1}^{n-1}\\
\o_1^{\l} & \cdots & \o_{n-1}^{\l}\\
\end{array}}}\right|$$
$$\renewcommand{\arraystretch}{2.6}{\begin{array}{l}
=(-1)^{n+j}\displaystyle\sum_{k=1}^{n-1}{\o_{k}^{\l}(-1)^{n-1+k}}\left|\renewcommand{\arraystretch}{1.5}{\setlength{\arraycolsep}{2mm}{
\begin{array}{cccccc}
\o_1^1 & \cdots & \o_{k-1}^{1} & \o_{k+1}^{1} & \cdots & \o_{n-1}^1\\
\vdots & \ddots & \vdots & \vdots & \ddots & \vdots \\
\o_1^{j-1} & \cdots & \o_{k-1}^{j-1} & \o_{k+1}^{j-1} & \cdots & \o_{n-1}^{j-1}\\
\o_1^{j+1} & \cdots & \o_{k-1}^{j+1} & \o_{k+1}^{j+1} & \cdots & \o_{n-1}^{j+1}\\
\vdots & \ddots & \vdots & \vdots & \ddots & \vdots \\
\o_1^{n-1} & \cdots & \o_{k-1}^{n-1} & \o_{k+1}^{n-1} & \cdots & \o_{n-1}^{n-1}\\\end{array}}}\right|\\
= \displaystyle\sum_{k=1}^{n-1}{(-1)^{j-1+k}\o_{k}^{\l}(-1)^{j+k}\cof(\o_k^{j},M)} = -\displaystyle\sum_{k=1}^{n-1}{\cof(\o_k^{j},M)\o_{k}^{\l}}\end{array}}$$
\end{itemize}

Thus, replacing the rows $L_j'$, for $j=1, \ldots, n-1$, Matrix (\ref{matrizdeVseMis}) becomes
\begin{equation}\label{matrizobtida}\left[\begin{array}{ccccccc}
\M  & \cdots & 0 & \vdots & -\cof(\o_n^1,\M_n)& \cdots & -\cof(\o_n^1,\M_m)\\
\vdots & \ddots & \vdots & \vdots & \vdots & \ddots & \vdots\\ 
0 & \cdots & \M & \vdots & -\cof(\o_n^{n-1},\M_n)& \cdots & -\cof(\o_n^{n-1},\M_m)\\
\cdots \ & \cdots & \cdots & \vdots &  \cdots \ \ \ \cdots \ \ \ \cdots & \cdots & \cdots \ \ \ \cdots \ \ \ \cdots\\ 
\ffrac{\partial \M_n}{\partial x_1} & \cdots & \ffrac{\partial \M_n}{\partial x_{n-1}} & \vdots & \ffrac{\partial \M_n}{\partial x_p}& \cdots & \ffrac{\partial \M_n}{\partial x_m}\\
\vdots & \ddots & \vdots & \vdots & \vdots & \ddots & \vdots\\ 
\ffrac{\partial \M_m}{\partial x_1} & \cdots & \ffrac{\partial \M_m}{\partial x_{n-1}} & \vdots & \ffrac{\partial \M_m}{\partial x_p}& \cdots & \ffrac{\partial \M_m}{\partial x_m}\\
\end{array}\right].
\end{equation} that still has maximal rank. {Let us denote the first $n-1$ row vectors of Matrix (\ref{matrizobtida}) by $L_j''$, for $j=1, \ldots, n-1$,} and let us consider the following expression for $j=n,\ldots, m$,
%if we denote the rows of Matrix (\ref{matrizobtida}) by  $L_j''$, for $j=1,\ldots,m$, we verify that $\dim\langle L_1'',\ldots,L_{n-1}''\rangle =n-1$ because the submatrix \begin{equation*}\left[\begin{array}{ccc}
%\M  & \cdots & 0 \\
%\vdots & \ddots & \vdots \\ 
%0 & \cdots & \M \\
%\end{array}\right]=\M \,Id_{(n-1)}
%\end{equation*} is non-singular. Furthermore, $\langle L_1'',\ldots,L_{n-1}''\rangle  \subset \langle L_1',\ldots,L_{n-1}'\rangle $, since each row $L_j''$ is a linear combination of the rows $L_1',\ldots,L_{n-1}'$, by Equation(\ref{novaslinhas}). Finally, since $\dim\langle L_1',\ldots,L_{n-1}'\rangle =n-1,$ we conclude that $$\langle L_1'',\ldots,L_{n-1}''\rangle =\langle L_1',\ldots,L_{n-1}'\rangle .$$
\footnotesize \begin{equation*}\renewcommand{\arraystretch}{3}{\begin{array}{l} \M L_j'-\displaystyle\sum_{k=1}^{n-1}{\ffrac{\partial\M_j}{\partial x_k}L_k''}\\
=\M\left(\ffrac{\partial \M_j}{\partial x_1}, \ldots, \ffrac{\partial \M_j}{\partial x_{n-1}}, \ffrac{\partial \M_j}{\partial x_n}, \ldots, \ffrac{\partial \M_j}{\partial x_m}\right) \\
+\left(-\M\ffrac{\partial \M_j}{\partial x_1}, \ldots, -\M\ffrac{\partial \M_j}{\partial x_{n-1}}, \displaystyle\sum_{k=1}^{n-1}\ffrac{\partial \M_j}{\partial x_k}\cof(\o_n^k,M_n), \ldots, \displaystyle\sum_{k=1}^{n-1}\ffrac{\partial \M_j}{\partial x_k}\cof(\o_n^k,M_m) \right)\\
=\left(0, \ldots, 0, \displaystyle\sum_{k=1}^{n-1}\ffrac{\partial \M_j}{\partial x_k}\cof(\o_n^k,M_n)+\M\ffrac{\partial \M_j}{\partial x_n}, \ldots, \displaystyle\sum_{k=1}^{n-1}\ffrac{\partial \M_j}{\partial x_k}\cof(\o_n^k,M_m)+\M\ffrac{\partial \M_j}{\partial x_m} \right).\\
%\\
%=\left(0, \ldots, 0, \displaystyle\sum_{k\in I_n}\ffrac{\partial \M_j}{\partial x_k}\cof(V_n^k,M_n), \ldots, \displaystyle\sum_{k\in I_m}\ffrac{\partial \M_j}{\partial x_k}\cof(V_n^k,M_m)\right)\\
\end{array}}
\end{equation*} \normalsize Note that $\M=\cof(\o_n^i,\M_i)$, for $i=n, \ldots,m$, so that the expression $$\M L_j'-\displaystyle\sum_{k=1}^{n-1}{\ffrac{\partial\M_j}{\partial x_k}L_k''}$$ is equal to $$\left(0, \ldots, 0, \displaystyle\sum_{k\in I_n}\ffrac{\partial \M_j}{\partial x_k}\cof(\o_n^k,M_n), \ldots, \displaystyle\sum_{k\in I_m}\ffrac{\partial \M_j}{\partial x_k}\cof(\o_n^k,M_m)\right).$$ \normalsize Thus, by Equation (\ref{eqdosmij}), we obtain $$\M L_j'-\displaystyle\sum_{k=1}^{n-1}{\ffrac{\partial\M_j}{\partial x_k}L_k''}=(0, \ldots, 0, m_{nj}, \ldots, m_{mj}).$$

For $j=n,\ldots, m$, we replace the row $L_j'$ in Matrix (\ref{matrizobtida}) by $$\M L_j'-\displaystyle\sum_{k=1}^{n-1}{\ffrac{\partial\M_j}{\partial x_k}L_k''},$$ such that the matrix obtained: 
\begin{equation}\label{matrizobtidadois}\left[\setlength{\arraycolsep}{0.1cm}{\begin{array}{ccccccc}
\M  & \cdots & 0 & \vdots & -\cof(\o_n^1,\M_n)& \cdots & -\cof(\o_n^1,\M_m)\\
\vdots & \ddots & \vdots & \vdots & \vdots & \ddots & \vdots\\ 
0 & \cdots & \M & \vdots & -\cof(\o_n^{n-1},\M_n)& \cdots & -\cof(\o_n^{n-1},\M_m)\\
\cdots \ & \cdots & \cdots & \vdots &  \cdots \ \ \ \  \cdots \ \ \ \ \cdots & \cdots & \cdots \ \ \ \ \cdots \ \ \ \ \cdots\\ 
 &  &  & \vdots & &  & \\
 & O_{(n-1)} &  & \vdots &  & (-\M')^t & \\ 
& &  & \vdots & &  &\\
\end{array}}\right]
\end{equation} also is non-singular. So, since $\M\neq0$, we have that $\det\M'\neq0.$

Thus, we can conclude that Matrix (\ref{matrizinicialxan}) is non-singular if and only if Matrix (\ref{matrizcomMlinha}) is non-singular, which occurs if and only if $$\det\left[\begin{array}{c}
d_x\x_1(p) \\
		\vdots \\
  d_x\x_m(p) \\
\end{array}\right]\neq0.$$ Therefore, $p$ will be a non-degenerate zero of $\x_{|_{\S^1(\o)}}$ if and only if $p$ is a non-degenerate zero of $\x$. \end{proof}

\begin{lema}\label{naodegkekmaisum} Let $p\in A_{k+1}(\o)$ be a zero of $\x_{|_{\S^{k+1}(\o)}}$. Then, for almost every $a\in\R^n\setminus\{\vec{0}\}$, $p$ is a non-degenerate zero of $\x_{|_{\S^{k+1}(\o)}}$ if and only if $p$ is a non-degenerate zero of $\x_{|_{\S^{k}(\o)}}$.
\end{lema}
\begin{proof} Let $p\in A_{k+1}(\o)$ be a zero of $\x_{|_{\S^{k+1}(\o)}}$ and let $\mathcal{U}\subset M$ be an open neighborhood of $p$ at which $\M(x)\neq0$ and the singular sets $\S^k(\o)$ ($k=2,\ldots,n$) are defined by $\mathcal{U}\cap\S^k(\o)=\{x\in\mathcal{U}: \M_n(x)=\ldots=\M_m(x)=\d_2(x)=\ldots=\d_k(x)=0\}$. 
%with $\rank(d\M_n, \ldots,d\M_m,d\d_2,\ldots,d\d_k)=m-n+k$, $\forall x\in\S^k(\o)\cap\mathcal{U}.$ %E $$\S^{k+1}(V)=\{x\in\mathcal{U}: \M_n(x)=\ldots=\M_m(x)=\d_2(x)=\ldots=\d_{k+1}(x)=0\},$$ com $$\rank(d\M_n, \ldots,d\M_m,d\d_2,\ldots,\d_{k+1})=m-n+k+1, \ \forall x\in\S^{k+1}(V)\cap\mathcal{U}.$$
%{\small\begin{equation*}\setlength{\arraycolsep}{0.06cm}{\renewcommand{\arraystretch}{2}{
%\begin{array}{lll}
%\S^k(\o)&=&\{x\in\mathcal{U}: \M_n(x)=\ldots=\M_m(x)=\d_2(x)=\ldots=\d_k(x)=0\},\\
%\S^{k+1}(\o)&=&\{x\in\mathcal{U}: \M_n(x)=\ldots=\M_m(x)=\d_2(x)=\ldots=\d_{k+1}(x)=0\},
%\end{array}}}
%\end{equation*}} com
%{\small\begin{equation*}\setlength{\arraycolsep}{0.06cm}{\renewcommand{\arraystretch}{2}{
%\begin{array}{l}
%\rank(d\M_n, \ldots,d\M_m,d\d_2,\ldots,d\d_k)=m-n+k, \text{ em } x\in\S^k(\o)\cap\mathcal{U},\\
%\rank(d\M_n, \ldots,d\M_m,d\d_2,\ldots,d\d_{k+1})=m-n+k+1, \ x\in\S^{k+1}(\o)\cap\mathcal{U}.
%\end{array}}}
%\end{equation*}}
By Szafraniec's characterization (\cite[p. 196]{sza2}), $p$ is a zero of the restriction $\x_{|_{\S^{k+1}(\o)}}$ if and only if there exists an unique $(\lambda_n,\ldots,\lambda_m,\beta_2,\ldots,\beta_{k+1})\in\R^{m-n+k+1}$ such that 
\begin{equation*} \x(p)+\displaystyle\sum_{j=n}^{m}{\lambda_id\M_i(p)}+\displaystyle\sum_{j=2}^{k+1}{\beta_{j}d\d_{j}(p)}=0.
\end{equation*} Since $p$ is a zero of $\x_{|_{\S^{k}(\o)}}$, we have $\b_{k+1}=0$. Moreover, also by Szafraniec's characterization, $p$ is a non-degenerate zero of $\x_{|_{\S^{k+1}(\o)}}$ if and only if the determinant of the following matrix does not vanish at $p$: 
\footnotesize
\begin{equation}\label{matrizinicialcasok}\left[\setlength{\arraycolsep}{0.05cm}{\begin{array}{ccccccccc}
\multirow{3}{*}{$\Jac_{x}\left(\x+\displaystyle\sum_{i=n}^{m}{\lambda_id\M_i}+\displaystyle\sum_{j=2}^{k}{\beta_{j}d\d_{j}}\right)$} & \vdots & \ffrac{\partial \M_n}{\partial x_1} & \cdots & \ffrac{\partial \M_m}{\partial x_1} & \ffrac{\partial \d_2}{\partial x_1} & \cdots & \ffrac{\partial \d_k}{\partial x_1}& \ffrac{\partial \d_{k+1}}{\partial x_1} \\
		 & \vdots & \vdots & \ddots & \vdots & \vdots & \ddots & \vdots & \vdots\\
  & \vdots & \ffrac{\partial \M_n}{\partial x_m} & \cdots & \ffrac{\partial \M_m}{\partial x_m} & \ffrac{\partial \d_2}{\partial x_m} & \cdots & \ffrac{\partial \d_k}{\partial x_m}& \ffrac{\partial \d_{k+1}}{\partial x_m}  \\
 %\cdots \ \ \cdots \ \ \cdots \ \ \cdots \ \ \cdots \ \ \cdots \ \ \cdots \ \ \cdots  & \vdots &  \cdots \ \cdots  & \ \cdots \ \cdots \  \cdots  &  \cdots \ \ \cdots  &  \cdots \  \cdots  &  \cdots  \ \cdots  &  \cdots  \ \cdots  & \cdots  \ \cdots \\		
			  \cdots \ \ \cdots \ \ \cdots \ \ \cdots \ \ \cdots \ \ \cdots \ \ \cdots \ \ \cdots & \vdots & \cdots & \cdots & \cdots \ \ \cdots & \cdots & \cdots & \cdots & \cdots \ \ \cdots\\
			  d_x\M_n & \vdots & & & &&&&\\
			  \vdots          & \vdots & & && && \\
 			  d_x\M_m & \vdots & & & &&& \\       
  		   d_x\d_2 & \vdots & &  \multicolumn{5}{c}{O_{(m-n+k+1)}} &\\
  		   \vdots       & \vdots & &  &&&& \\
 			   d_x\d_{k+1} & \vdots & & & &&&\\     
 			   % d_x\d_{k+1} & \vdots & & & &&&\\             
\end{array}}\right].
\end{equation}
\normalsize

Analogously, $p$ is a non-degenerate zero of $\x_{|_{\S^k(\o)}}$ if and only if the determinant of the following matrix does not vanish at $p$: 
\footnotesize
\begin{equation}\label{matrizinicialsemdeltakmais1}\left[\setlength{\arraycolsep}{0.05cm}{\begin{array}{cccccccc}
 & \vdots & \ffrac{\partial \M_n}{\partial x_1} & \cdots & \ffrac{\partial \M_m}{\partial x_1} & \ffrac{\partial \d_2}{\partial x_1} & \cdots & \ffrac{\partial \d_k}{\partial x_1} \\
		\Jac_x\left(\x+\displaystyle\sum_{i=n}^{m}{\lambda_id\M_i}+\displaystyle\sum_{j=2}^{k}{\beta_{j}d\d_{j}}\right) & \vdots & \vdots & \ddots & \vdots & \vdots & \ddots & \vdots \\
  & \vdots & \ffrac{\partial \M_n}{\partial x_m} & \cdots & \ffrac{\partial \M_m}{\partial x_m} & \ffrac{\partial \d_2}{\partial x_m} & \cdots & \ffrac{\partial \d_k}{\partial x_m}  \\
  \cdots \ \ \cdots \ \ \cdots \ \ \cdots \ \ \cdots \ \ \cdots \ \ \cdots \ \ \cdots & \vdots & \cdots & \cdots & \cdots \ \ \cdots & \cdots & \cdots & \cdots \\
			  d_x\M_n & \vdots & & & &&&\\
			  \vdots          & \vdots & & && & \\
 			  d_x\M_m & \vdots & & & && \\       
  		   d_x\d_2 & \vdots & \multicolumn{5}{c}{O_{(m-n+k)}} &   \\
  		   \vdots       & \vdots & &  &&& \\
 			   d_x\d_k & \vdots & & & &&\\                
\end{array}}\right].
\end{equation} \normalsize %isto é, se o determinante da matriz obtida eliminando-se a derivada $d\d_{k+1}(p)$ continuar não nulo. 
Thus, to prove the lemma it is enough to show that the Matrix (\ref{matrizinicialcasok}) is non-singular at $p$ if and only if the Matrix (\ref{matrizinicialsemdeltakmais1}) is non-singular at $p$.

Note that the Jacobian matrix with respect to $x$ \begin{equation}\label{Jacmatrix}\Jac_x\left(\x+\displaystyle\sum_{i=n}^{m}{\lambda_id\M_i}+\displaystyle\sum_{j=2}^{k}{\beta_{j}d\d_{j}}\right)\end{equation} is a submatrix of both the Matrices (\ref{matrizinicialcasok}) and (\ref{matrizinicialsemdeltakmais1}). And remember that, {for $x$ in an open neighborhood of $p$}, $\d_{k+1}=\det(d\M_n, \ldots, d\M_m,d\d_2,\ldots, d\d_k, \z_1,\ldots,\z_{n-k}),$ where $\{\z_1(x),\ldots,\z_{n-k}(x)\}$ is a basis of a vector subspace supplementary to $\langle\bar{\o}(x)\rangle \cap N_x^{\ast}\S^{k-1}(\o)$ in $\langle\bar{\o}(x)\rangle $. Thus, $$\langle\bar{\o}(x)\rangle =\langle\z_1(x),\ldots,\z_{n-k}(x)\rangle \oplus(\langle\bar{\o}(x)\rangle \cap N_x^{\ast}\S^{k-1}(\o)).$$ 

Since, for almost every $a$, $\x_{|_{\S^{k-1}(\o)}}(p)\neq0$ then $\x(p)\in\langle\bar{\o}(p)\rangle \setminus N_p^{\ast}\S^{k-1}(\o)$. {That is, there exists $(\mu_1, \ldots, \mu_{n-k})\in\R^{n-k}\setminus\{\vec{0}\}$ so that $$\x(p)=\displaystyle\sum_{i=1}^{n-k}{\mu_i\z_i(p)}+\f(p),$$ for some $\f(p)\in N_p^{\ast}\S^{k-1}(\o)$, where $\f(p)=\displaystyle\sum_{i=n}^{m}{\tilde{\lambda_i}d\M_i(p)}+\displaystyle\sum_{j=2}^{k-1}{\tilde{\b_j}d\d_j(p)}$. Thus, 
\begin{equation}\label{eqmi}
\x(p)=\displaystyle\sum_{i=1}^{n-k}{\mu_i\z_i(p)}+\displaystyle\sum_{i=n}^{m}{\tilde{\lambda_i}d\M_i(p)}+\displaystyle\sum_{j=2}^{k-1}{\tilde{\b_j}d\d_j(p)},
\end{equation} and the expression $$\x(p)+\displaystyle\sum_{j=n}^{m}{\lambda_id\M_i(p)}+\displaystyle\sum_{j=2}^{k}{\beta_{j}d\d_{j}(p)}$$ can be written as: \begin{equation}\label{expressaodozmais}\displaystyle\sum_{i=1}^{n-k}{\mu_i\z_i(p)}+\displaystyle\sum_{i=n}^{m}{(\lambda_i+\tilde{\lambda_i})d\M_i(p)}+\displaystyle\sum_{j=2}^{k-1}{(\b_j+\tilde{\b_j})d\d_j(p)}+\b_k d\d_k(p).\end{equation}} 

{Let us consider the mapping $$H(x)=\displaystyle\sum_{i=1}^{n-k}{\mu_i\z_i(x)}+\displaystyle\sum_{i=n}^{m}{(\lambda_i+\tilde{\lambda_i})d\M_i(x)}+\displaystyle\sum_{j=2}^{k-1}{(\b_j+\tilde{\b_j})d\d_j(x)}+\b_k d\d_k(x),$$ defined on an open neighborhood $\mathcal{U}\subset M$ of $p$, which is equal to $$ 
\resizebox{!}{1.15cm}{$
\displaystyle\sum_{i=1}^{n-k}{\mu_i\left[
\begin{array}{c}
\z_{i}^1\\
\vdots\\
\z_{i}^m
\end{array}\right]}+\displaystyle\sum_{i=n}^{m}{(\lambda_i+\tilde{\lambda_i})\left[\begin{array}{c}
\ffrac{\partial\M_i}{\partial x_1}\\
\vdots\\
\ffrac{\partial\M_i}{\partial x_m}
\end{array}\right]}+\displaystyle\sum_{j=2}^{k-1}{(\b_j+\tilde{\b_j})\left[\begin{array}{c}
\ffrac{\partial\d_j}{\partial x_1}\\
\vdots\\
\ffrac{\partial\d_j}{\partial x_m}
\end{array}\right]}+\b_k\left[\begin{array}{c}
\ffrac{\partial\d_k}{\partial x_1}\\
\vdots\\
\ffrac{\partial\d_k}{\partial x_m}
\end{array}\right]$}
$$ and can be written as $$\left[\begin{array}{c}\displaystyle\sum_{i=1}^{n-k}{\mu_i\z_{i}^1}+\displaystyle\sum_{i=n}^{m}{(\lambda_i+\tilde{\lambda_i})\ffrac{\partial\M_i}{\partial x_1}}+\displaystyle\sum_{j=2}^{k-1}{(\b_j+\tilde{\b_j})\ffrac{\partial\d_j}{\partial x_1}}+\b_k\ffrac{\partial\d_k}{\partial x_1}\\
\vdots\\
\displaystyle\sum_{i=1}^{n-k}{\mu_i\z_i^m}+\displaystyle\sum_{i=n}^{m}{(\lambda_i+\tilde{\lambda_i})\ffrac{\partial\M_i}{\partial x_m}}+\displaystyle\sum_{j=2}^{k-1}{(\b_j+\tilde{\b_j})\ffrac{\partial\d_j}{\partial x_m}}+\b_k\ffrac{\partial\d_k}{\partial x_m}
\end{array}\right].$$ Then, the Jacobian matrix of $H(x)$ is given by:}
\footnotesize
\begin{equation}\label{matjacdaexpressao}\left[\begin{array}{c}\displaystyle\sum_{i=1}^{n-k}{\mu_id_x\z_{i}^1}+\displaystyle\sum_{i=n}^{m}{(\lambda_i+\tilde{\lambda_i})d_x\ffrac{\partial\M_i}{\partial x_1}}+\displaystyle\sum_{j=2}^{k-1}{(\b_j+\tilde{\b_j})d_x\ffrac{\partial\d_j}{\partial x_1}}+\b_kd_x\ffrac{\partial\d_k}{\partial x_1}\\
\vdots\\
\displaystyle\sum_{i=1}^{n-k}{\mu_id_x\z_{i}^m}+\displaystyle\sum_{i=n}^{m}{(\lambda_i+\tilde{\lambda_i})d_x\ffrac{\partial\M_i}{\partial x_m}}+\displaystyle\sum_{j=2}^{k-1}{(\b_j+\tilde{\b_j})d_x\ffrac{\partial\d_j}{\partial x_m}}+\b_kd_x\ffrac{\partial\d_k}{\partial x_m}
\end{array}\right].\end{equation}

\normalsize

To apply the Lemma \ref{lematecnico}, fix the notation: $A_i(x)=\left(a_{1i}(x), \ldots, a_{mi}(x)\right)$, such that
\begin{equation*}\begin{array}{l}
{A_i(x):=\left\{\begin{array}{ll}
\z_i(x), & i=1, \ldots, n-k;\\
d\M_i(x), & i=n, \ldots, m;\\
\end{array}\right.}\\
\\
{A_{n-k+j-1}(x):=d\d_j(x), \ \ j=2, \ldots, k;} \\
\\
{\a_i:=\left\{\begin{array}{ll}
\mu_i, & i=1, \ldots, n-k; \ (\text{we can suppose } \a_1\neq0, \text{ since } \x(p)\neq \f(p))\\
(\lambda_i+\tilde{\lambda_i}), & i=n, \ldots, m;\\
\end{array}\right.}\\
\\
\a_{n-k+j-1}:=(\b_j+\tilde{\b_j}), \ \ j=2, \ldots, k; \ \ (\tilde{\b_k}=0).\\
\end{array}
\end{equation*}
%\begin{equation*}\begin{array}{l}
%\a_{n-k+j-1}:=(\b_j+\tilde{\b_j}), \ \ j=2, \ldots, k; \ \ (\tilde{\b_k}=0).\\
%\end{array}
%\end{equation*}
Since $\x(p)+\displaystyle\sum_{j=n}^{m}{\lambda_id\M_i(p)}+\displaystyle\sum_{j=2}^{k+1}{\beta_{j}d\d_{j}(p)}=0$, by Expression (\ref{expressaodozmais}) we have $$\displaystyle\sum_{i=1}^{m}{\a_iA_i}(p)=0.$$ That is, $$\displaystyle\sum_{i=1}^{m}{\a_ia_{ji}(p)}=0, \, \forall j=1, \ldots, m.$$ Then, applying Lemma \ref{lematecnico} we know that \begin{equation}\a_1\cof(a_{ik}(p))-\a_k\cof(a_{i1}(p))=0, \, \forall i,k=1, \ldots m.\end{equation} 

We also have that \begin{equation*}\begin{array}{ll}
\d_{k+1}&=\det\left(d\M_n, \ldots,d\M_m, d\d_2, \ldots, d\d_k, \z_1, \ldots, \z_{n-k}\right)\\
        &=\det\left(A_n, \ldots, A_m, A_{n-k+1}, \ldots, A_{n-1}, A_1, \ldots, A_{n-k}\right)\\
        &=(-1)^{\varepsilon}\det\left(A_1, \ldots, A_m\right)
\end{array}\end{equation*} where $\varepsilon$ is either equal to zero  or equal to $1$, depending on the number of required permutations between the column vectors of the matrix $$\left(A_n, \ldots, A_m, A_{n-k+1}, \ldots, A_{n-1}, A_1, \ldots, A_{n-k}\right)$$ in order to obtain the matrix $\left(A_1, \ldots, A_m\right)$. Thus,\begin{equation*}\renewcommand{\arraystretch}{2.3}{\setlength{\arraycolsep}{0.13cm}{\begin{array}{lcl}
(-1)^{\varepsilon}d\d_{k+1}&=&\displaystyle\sum_{i,j=1}^{m}{\cof(a_{ij})d a_{ij}}\\
        &=&\displaystyle\sum_{i=1}^{m}{\left(\cof(a_{i1})d a_{i1} + \displaystyle\sum_{j=2}^{m}{\cof(a_{ij})d a_{ij}}\right)}\\
        &\overset{\a_1\neq0}{=}&\displaystyle\sum_{i=1}^{m}{\left(\cof(a_{i1})d a_{i1} + \displaystyle\sum_{j=2}^{m}{\ffrac{\a_j}{\a_1}\cof(a_{i1})d a_{ij}}\right)}\\
\end{array}}}\end{equation*} which implies that, for each $x\in\mathcal{U}$ \begin{equation}\label{graddeltakmaisum}\renewcommand{\arraystretch}{2.3}{\begin{array}{ll}
(-1)^{\varepsilon}\a_1d\d_{k+1}&=\displaystyle\sum_{i=1}^{m}{\left(\a_1\cof(a_{i1})d a_{i1} + \displaystyle\sum_{j=2}^{m}{\a_j\cof(a_{i1})d a_{ij}}\right)}\\
        &=\displaystyle\sum_{i=1}^{m}{\cof(a_{i1})\left[\displaystyle\sum_{j=1}^{m}{\a_jd a_{ij}}\right]}\\
        &=\displaystyle\sum_{i=1}^{m}{ \cof(a_{i1})\L_i }\\
\end{array}}\end{equation} where $\L_i, i=1, \ldots, m$, denote the rows of the Jacobian matrix (\ref{matjacdaexpressao}). {If we denote by $\tilde{L}_i, i=1, \ldots,m$, the row vectors of Jacobian matrix (\ref{Jacmatrix}) at $p$, then we can verify that \begin{equation}\label{equalJacobian}\displaystyle\sum_{i=1}^{m}{ \cof(a_{i1})\L_i }=\displaystyle\sum_{i=1}^{m}{ \cof(a_{i1})\tilde{L}_i}\end{equation} }

{Let us denote the first $m$ row vectors of Matrix (\ref{matrizinicialcasok}) by $L_i, i=1, \ldots,m$, and its last row vector by $L_{\d_{k+1}}$. Based on Expressions (\ref{graddeltakmaisum}) at $p$ and (\ref{equalJacobian}), we replace the row vector $L_{\d_{k+1}}$ by \begin{equation}\label{operacaocomlinhas}(-1)^{\varepsilon}\a_1L_{\d_{k+1}}-\displaystyle\sum_{i=1}^{m}{\cof(a_{i1})L_i},\end{equation} in order to obtain a new last row vector given by $$L_{\d_{k+1}}:=(\underbrace{0, \ldots, 0}_{m}, \g_n, \ldots, \g_m, \tilde{\g_2}, \ldots, \tilde{\g_k}, \tilde{\g_{k+1}}),$$} where \begin{equation*}\renewcommand{\arraystretch}{1}{\begin{array}{lll}
\g_j&=-\displaystyle\sum_{i=1}^{m}{\cof(a_{i1})\ffrac{\partial\M_j}{\partial x_i}}, & \forall j=n, \ldots, m;\\
\\
\tilde{\g_j}&=-\displaystyle\sum_{i=1}^{m}{\cof(a_{i1})\ffrac{\partial\d_j}{\partial x_i}}, & \forall j=2, \ldots, k+1.\\
\end{array}}\end{equation*} Note that, for $j=n, \ldots, m,$  $$\renewcommand{\arraystretch}{1.7}{\begin{array}{l}
\g_j=-\displaystyle\sum_{i=1}^{m}{\cof(a_{i1})\ffrac{\partial\M_j}{\partial x_i}}=-\left|\begin{array}{cccc}
\ffrac{\partial\M_j}{\partial x_1} & a_{12} & \cdots & a_{1m}\\
\vdots & \vdots & \ddots & \vdots\\
\ffrac{\partial\M_j}{\partial x_m} & a_{m2} & \cdots & a_{mm}
\end{array}\right|\\
\\
\Rightarrow\g_j=-\det\left(A_j, A_2, \ldots, A_m\right)=0.\end{array}}$$ For $j=2, \ldots, k$,  $$\renewcommand{\arraystretch}{1.7}{\begin{array}{l}
\tilde{\g_j}=-\displaystyle\sum_{i=1}^{m}{\cof(a_{i1})\ffrac{\partial\d_j}{\partial x_i}}=-\left|\begin{array}{cccc}
\ffrac{\partial\d_j}{\partial x_1} & a_{12} & \cdots & a_{1m}\\
\vdots & \vdots & \ddots & \vdots\\
\ffrac{\partial\d_j}{\partial x_m} & a_{m2} & \cdots & a_{mm}
\end{array}\right|\\
\\
\Rightarrow\tilde{\g_j}=-\det\left(A_{n-k+j-1}, A_2, \ldots, A_m\right)=0.
\end{array}}$$ 

Therefore, after replacing the vector row $L_{\d_{k+1}}$ in the Matrix (\ref{matrizinicialcasok}), we obtain 
\footnotesize
\begin{equation}\label{matrizgammatil}\left[ \renewcommand{\arraystretch}{1}{\setlength{\arraycolsep}{0.04cm}{\begin{array}{cccccccccc}
 & \vdots & \ffrac{\partial \M_n}{\partial x_1} & \cdots & \ffrac{\partial \M_m}{\partial x_1} & \ffrac{\partial \d_2}{\partial x_1} & \cdots & \ffrac{\partial \d_k}{\partial x_1}& \vdots &\ffrac{\partial \d_{k+1}}{\partial x_1} \\
		\Jac\left(\x+\displaystyle\sum_{i=n}^{m}{\lambda_id\M_i}+\displaystyle\sum_{j=2}^{k}{\beta_{j}d\d_{j}}\right) & \vdots & \vdots & \ddots & \vdots & \vdots & \ddots & \vdots & \vdots &\vdots\\
  & \vdots & \ffrac{\partial \M_n}{\partial x_m} & \cdots & \ffrac{\partial \M_m}{\partial x_m} & \ffrac{\partial \d_2}{\partial x_m} & \cdots & \ffrac{\partial \d_k}{\partial x_m}& \vdots &\ffrac{\partial \d_{k+1}}{\partial x_m}  \\
			  \cdots \ \ \cdots \ \ \cdots \ \ \cdots \ \ \cdots \ \ \cdots \ \ \cdots \ \ \cdots & \vdots & \cdots & \cdots & \cdots \ \ \cdots & \cdots & \cdots & \cdots & \vdots & \cdots \ \cdots\\
			  d_x\M_n & \vdots & & & &&&&\vdots &0\\
			  \vdots       & \vdots & & & & &&&\vdots &\vdots \\
 			  d_x\M_m & \vdots & & & & &&&\vdots &0\\       
  		   d_x\d_2 & \vdots &\multicolumn{6}{c}{O_{(m-n+k-1)}}    &\vdots & 0\\
  		   \vdots       & \vdots & &  &&&&&\vdots & \vdots\\
 			   d_x\d_k & \vdots & & & &&&&\vdots &0\\     
 			   \cdots \ \ \cdots \ \ \cdots \ \ \cdots \ \ \cdots \ \ \cdots \ \ \cdots \ \ \cdots & \vdots & \cdots & \cdots & \cdots \ \ \cdots & \cdots & \cdots & \cdots & \vdots & \cdots \ \cdots\\
 			   \vec{0} & \vdots &  \multicolumn{6}{c}{\vec{0}} & \vdots &\tilde{\g_{k+1}}\\             
\end{array}}}\right].
\end{equation} \normalsize

Let us show that $\tilde{\g_{k+1}}(p)\neq0$. We have

$$\renewcommand{\arraystretch}{2}{\begin{array}{cl}
\tilde{\g_{k+1}}&=-\displaystyle\sum_{i=1}^{m}{\cof(a_{i1})\ffrac{\partial\d_{k+1}}{\partial x_i}}\\
                &=-\det(d\d_{k+1}, A_2, \ldots, A_{m})\\
                &=-\det(d\d_{k+1}, \z_2, \ldots, \z_{n-k},d\d_2, \ldots, d\d_k, d\M_n, \ldots, d\M_m).
\end{array}}$$ Suppose that $\tilde{\g_{k+1}}=0$. Since each one of the sets $\{\z_2(p), \ldots, \z_{n-k}(p)\}$ and $\{d\d_{k+1}(p),d\d_2(p), \ldots, d\d_k(p), d\M_n(p), \ldots, d\M_m(p)\}$ consist of linearly independent vectors, there exists $j\in\{2, \ldots, n-k\}$ such that $\z_j(p)\in N_p^{\ast}\S^{k+1}(\o).$ We can suppose that without loss of generality that $j=n-k$, that is, $$\z_{n-k}(p)\in N_p^{\ast}\S^{k+1}(\o)=\langle d\M_n, \ldots, d\M_m,d\d_2, \ldots, d\d_k,d\d_{k+1}\rangle .$$ Since $\x_{|_{\S^{k+1}}}(p)=0$, we have $\x(p)\in N_p^{\ast}\S^{k+1}(\o)$. Thus, by Equation (\ref{eqmi}), we obtain
$$\renewcommand{\arraystretch}{2.5}{\begin{array}{ll}
& \displaystyle\sum_{i=1}^{n-k}{\mu_i\z_i}+ \underbrace{\displaystyle\sum_{i=n}^{m}{\tilde{\lambda_i}d\M_i}+ \displaystyle\sum_{j=2}^{k-1}{\tilde{\b_j}d\d_j}}_{\small \in N_p^{\ast}\S^{k+1}(\o)}\normalsize\in N_p^{\ast}\S^{k+1}(\o)\\
& \Rightarrow \displaystyle\sum_{i=1}^{n-k-1}{\mu_i\z_i}=\displaystyle\sum_{i=1}^{n-k}{\mu_i\z_i}-\mu_{n-k}\z_{n-k}\in N_p^{\ast}\S^{k+1}(\o).
\end{array}}$$ Therefore, $\displaystyle\sum_{i=1}^{n-k-1}{\mu_i\z_i}$ and $\mu_{n-k}\z_{n-k}$ are linearly independent vectors in the vector subspace $\langle\z_1, \ldots, \z_{n-k}\rangle \cap N_p^{\ast}\S^{k+1}(\o)$. That is, $$\dim\left(\langle\z_1(p), \ldots, \z_{n-k}(p)\rangle \cap N_p^{\ast}\S^{k+1}(\o)\right)\geq2.$$ Since $\langle\bar{\o}\rangle =\langle\z_1, \ldots, \z_{n-k}\rangle \oplus\left(\langle\bar{\o}\rangle \cap N_p^{\ast}\S^{k-1}(\o)\right),$ we have %$\langle V_1, \ldots, V_n\rangle \cap N_p\S^{k+1}(V)$ pode ser escrito como 
$$\setlength{\arraycolsep}{0.1cm}{\begin{array}{lcl}
\langle\bar{\o}\rangle \cap N_p^{\ast}\S^{k+1}(\o)
%&=&\langle\z_1, \ldots, \z_{n-k}\rangle \cap N_p^{\ast}\S^{k+1}(\o)\\
%&\oplus&\langle\bar{\o}\rangle \cap N_p^{\ast}\S^{k-1}(\o)\cap N_p^{\ast}\S^{k+1}(\o)\\
&=&\langle\z_1, \ldots, \z_{n-k}\rangle \cap N_p^{\ast}\S^{k+1}(\o)\\
&\oplus&\langle\bar{\o}\rangle \cap N_p^{\ast}\S^{k-1}(\o).\end{array}}$$ That is, $$\dim\left(\langle\bar{\o}(p)\rangle \cap N_p^{\ast}\S^{k+1}(\o)\right)\geq 2+(k-1)=k+1,$$ which means that $p\in\S^{k+2}(\o)$. But this is a contradiction, since $p\in A_{k+1}(\o)$ by hypothesis and $\S^{k+2}(\o)=\S^{k+1}(\o)\setminus A_{k+1}(\o)$. Therefore $\tilde{\g_{k+1}}(p)\neq0$.\\

Thus, we conclude that the Matrix (\ref{matrizinicialcasok}) is non-singular at $p$ if and only if the Matrix (\ref{matrizgammatil}) is non-singular at $p$, which occurs if and only if the Matrix (\ref{matrizinicialsemdeltakmais1}) is non-singular at the point $p$. 
\end{proof}

\begin{lema}\label{naodegan} For almost every $a\in\R^n\setminus\{\vec{0}\}$, if $p\in A_n(\o)$ then $p$ is a non-degenerate zero of $\x_{|_{\S^{n-1}(\o)}}$.
\end{lema}

\begin{proof} We know that if $p\in A_n(\o)$ then $\x_{|_{\S^{n-1}(\o)}}(p)=0$. By Szafraniec's characterization \cite[p.149-151]{sza1}, $p$ is a non-degenerate zero of $\x_{|_{\S^{n-1}(\o)}}$ if and only if the following conditions hold:
\begin{enumerate}[$(i)$]
\item $\d(p)=\det(d\M_n, \ldots,d\M_m,d\d_2,\ldots,d\d_{n-1},\x)(p)=0$;\\
\item $\det(d\M_n, \ldots,d\M_m,d\d_2,\ldots,d\d_{n-1},d\d)(p)\neq0$.
\end{enumerate}
Condition $(i)$ is clearly satisfied, since $\x_{|_{\S^{n-1}(\o)}}(p)=0$. Let us verify that condition $(ii)$ also holds.

For each $x\in\S^{n-1}(\o)$ in an open neighborhood $\mathcal{U}\subset M$ of $p$, let $\{\z'(x)\}$ be a smooth basis for a vector subspace supplementary to $\langle\bar{\o}(x)\rangle \cap N_x^{\ast}\S^{n-2}(\o)$ in the vector space $\langle\bar{\o}(x)\rangle $. Since $\x(x)\in\langle\bar{\o}(x)\rangle $, we have $$\x(x)=\lambda(x)\z'(x)+\f(x),$$ where $\lambda(x)\in\R$ and $\f(x)\in \langle\bar{\o}(x)\rangle \cap N_x^{\ast}\S^{n-2}(\o)$, $\forall x\in\mathcal{U}\cap\S^{n-1}(\o)$. 

In particular, if $x\in A_n(\o)$, we know that, for almost every $a\in\R^n\setminus\{\vec{0}\}$, {$\x_{|_{\S^{n-2}(\o)}}(x)\neq0$} and, consequently, $\x(x)\notin N_x^{\ast}\S^{n-2}(\o)$. Thus $\lambda(p)\neq0$. For all $x\in\mathcal{U}\cap\S^{n-1}(\o)$, we obtain $$\setlength{\arraycolsep}{0.07cm}{\begin{array}{ccl}
\d(x)&=&\det(d\M_n, \ldots,d\M_m,d\d_2,\ldots,d\d_{n-1},\lambda \z'+\f)(x)\\
&=&\lambda(x)\det(d\M_n, \ldots,d\M_m,d\d_2,\ldots,d\d_{n-1},\z')(x)\\
&=&\lambda(x)\d_n(x),
\end{array}}$$ with $\d_n(p)=0$ and $\lambda(p)\neq0$. Then, by Lemma \ref{lematecnicodois} we have $$\begin{array}{l}\langle d\M_n(p), \ldots, d\M_m(p), d\d_2(p), \ldots, d\d_{n-1}(p), d\d(p)\rangle \\
=\langle d\M_n(p), \ldots, d\M_m(p), d\d_2(p), \ldots, d\d_{n-1}(p), d(\lambda\d_n)(p)\rangle .
\end{array}$$ However, $d(\lambda\d_n)(x)=d\lambda(x)\d_n(x)+\lambda(x)d\d_n(x)$, $\d_n(p)=0$ and $\lambda(p)\neq0$. Thus, $$\begin{array}{l}\langle d\M_n(p), \ldots, d\M_m(p), d\d_2(p), \ldots, d\d_{n-1}(p), d\d(p)\rangle \\
=\langle d\M_n(p), \ldots, d\M_m(p), d\d_2(p), \ldots, d\d_{n-1}(p), d\d_n(p)\rangle .
\end{array}$$ Therefore, $\det(d\M_n(p), \ldots, d\M_m(p), d\d_2(p), \ldots, d\d_{n-1}(p), d\d(p))\neq0$. 
\end{proof}

\begin{lema}\label{nondegeneratexisk} For almost every $a\in\R^n\setminus\{\vec{0}\}$, the 1-form $\x_{|_{\S^k(\o)}}$ admits only non-degenerate zeros, $k\geq1$.
\end{lema}

\begin{proof} Suppose that $\x_{|_{\S^k(\o)}}(p)=0$. Then, for almost every $a\in\R^n\setminus\{\vec{0}\}$, $p\in A_{k}(\o)\cup A_{k+1}(\o)$ since $Z(\x_{|_{\S^k(\o)}})\cap\S^{k+2}(\o)=\emptyset$ by Lemma \ref{ptscrticrestasigmak} and $\S^{k}(\o)=A_k(\o)\cup A_{k+1}(\o)\cup\S^{k+2}(\o)$.

If $p\in A_k(\o)$ then $\x_{|_{A_k(\o)}}(p)=0$. Since $\x_{|_{A_k(\o)}}$ admits only non-degenerate zeros and $A_k(\o)\subset\S^k(\o)$ is a open subset, we conclude that $p$ is a non-degenerate zero of $\x_{|_{\S^k(\o)}}$.

If $p\in A_{k+1}(\o)$ and $k< n-1$ then $\x_{|_{\S^{k+1}(\o)}}(p)=0$. In particular, since $A_{k+1}(\o)\subset\S^{k+1}(\o)$ is an open subset then $\x_{|_{A_{k+1}(\o)}}(p)=0$. By Lemmas \ref{zerosnaodegdeaum} and \ref{zerosnaodegdeak}, $\x_{|_{A_{k+1}(\o)}}$ admits only non-degenerate zeros, and since $A_{k+1}(\o)$ is an open set of $\S^{k+1}(\o)$, we conclude that $p$ is a non-degenerate zero of $\x_{|_{\S^{k+1}(\o)}}$. Therefore, by Lemma \ref{naodegkekmaisum}, $p$ is non-degenerate zero of $\x_{|_{\S^k(\o)}}$. Finally, if $p\in A_{n}(\o)$, by Lemma \ref{naodegan}, $p$ is a non-degenerate zero of $\x_{|_{\S^{n-1}(\o)}}$.

\end{proof}

%Let $L\in\R P^{n-1}$ be a straight line in $\R^n$ and let $\pi_L:\R^n\rightarrow L$ be the orthogonal projection to $L$. In \cite{Fukuda}, T. Fukuda apply Morse theory and well known properties of singular sets $A_k(f)$ of a Morin map $f:M\rightarrow\R^n$ to study the critical points of mappings $\pi_L\circ f:M\rightarrow L$ and their restrictions to the singular sets $\pi_L\circ f|_{A_k(f)}$ and $\pi_L\circ f|_{\overline{A_k(f)}}$. In this work, we could verify that a generic 1-form $$\x(x)=\displaystyle\sum_{i=1}^{n}{a_i\x_i(x)}$$ associated to a Morin $n$-coframe $\o=(\o_1, \ldots, \o_n)$ has properties that are similar to the properties of the generic orthogonal projections $\pi_L\circ f(x)$ associated to a Morin map $f=(f_1, \ldots, f_n)$ and its restrictions. As a consequence of these results, we can obtain a generalization of Fukuda's Theorem \cite[Theorem 1]{Fukuda} for the case of Morin $n$-coframes. We end the paper with the proof of this generalized theorem, whose proof uses the classical Poincaré-Hopf Theorem for 1-forms.  

\begin{teo}\label{fukudaparacampos} {Let $\o=(\o_1, \ldots, \o_n)$ be a Morin $n$-coframe defined on an $m$-dimensional compact manifold $M$. Then, $$\chi(M)\equiv\displaystyle\sum_{k=1}^{n}{\chi(\overline{A_k(\o)})} \mod 2.$$}
\end{teo}

\begin{proof} Let us denote by $Z(\f)$ the set of zeros of a 1-form $\f$ and let us denote by $\#Z(\f)$ the number of elements of this set, whenever $Z(\f)$ is finite. {Let $$\x(x)=\displaystyle\sum_{i=1}^{n}{a_i\o_i(x)}$$ be a 1-form with $a=(a_1, \ldots, a_n)\in\R^n\setminus\{\vec{0}\}$ satisfying the generic conditions of the previous lemmas of Sections 3 and 4.}

Since $M$ is compact and the submanifolds $\S^k(\o)$ are closed in $M$, by the Poincaré-Hopf Theorem for 1-forms we obtain \begin{itemize}
\item $\chi(M)\equiv\#Z(\x) \mod2$;
\item $\chi(\overline{A_k(\o)})=\chi(\S^k(\o))\equiv\#Z(\x_{|_{\S^k(\o)}}) \mod2$, for $k=1, \ldots, n-1$;
\item $\chi(\overline{A_n(\o)})=\chi(\S^n(\o))\equiv\#Z(\x_{|_{\S^n(\o)}}) \mod2$.
\end{itemize} %where $\chi(A)$ denotes the Euler characteristic of $A$.

By Lemma \ref{zeroszsobresigma1}, if $p\in Z(\x)$ then $p\in\S^1(\o)$ and $\x_{|_{\S^1(\o)}}(p)=0$. Moreover, by Lemma \ref{lemainterzeroscomsigma2}, $Z(\x)\cap\S^2(\o)=\emptyset$. Thus $p\in A_1(\o)$. On the other hand, Lemma \ref{lemazerosrestricoes} shows that if $p\in Z(\x_{|_{\S^1(\o)}})\cap A_1(\o)$ then $p$ is also a zero of the 1-form $\x$. Thus, $$\#Z(\x)\equiv\#Z(\x_{|_{\S^1(\o)}}\cap A_1(\o)) \mod2.$$

By Lemma \ref{ptscrticrestasigmak}, if $p\in Z(\x_{|_{\S^k(\o)}})$ then $p\notin\S^{k+2}(\o)$. Thus, $p\in A_k(\o)\cup A_{k+1}(\o)$ and, for $k=1, \ldots, n-1$, we have $$\#Z(\x_{|_{\S^k(\o)}})\equiv\#Z(\x_{|_{\S^k(\o)}}\cap A_k(\o))+ \#Z(\x_{|_{\S^k(\o)}}\cap A_{k+1}(\o))\mod2.$$

By Lemma \ref{lemazerosrestricoes}, we also have $$\#Z(\x_{|_{\S^k(\o)}}\cap A_{k+1}(\o))=\#Z(\x_{|_{\S^{k+1}(\o)}}\cap A_{k+1}(\o))$$ and by Lemma \ref{ptosansaozeros}, $$\#A_n(\o)=\#Z(\x_{|_{\S^{n-1}(\o)}}\cap A_{n}(\o)).$$

Then, \begin{itemize}
\item $\chi(M)\equiv\#Z(\x_{|_{\S^1(\o)}}\cap A_1(\o)) \mod2$;
\item For $k=1, \ldots, n-1$, $$\chi(\overline{A_k(\o)})\equiv\#Z(\x_{|_{\S^k(\o)}}\cap A_k(\o))+ \#Z(\x_{|_{\S^{k+1}(\o)}}\cap A_{k+1}(\o))\mod2;$$
\item $\chi(\overline{A_n(\o)})=\#Z(\x_{|_{\S^{n-1}(\o)}}\cap A_{n}(\o))$.
\end{itemize}
Therefore, $$\renewcommand{\arraystretch}{1.7}{\begin{array}{lll}
\chi(M)+\displaystyle\sum_{k=1}^{n}{\chi(\overline{A_k(\o)})}&\equiv&2\#Z(\x_{|_{\S^1(\o)}}\cap A_1(\o))\\
&+&2 \#Z(\x_{|_{\S^2(\o)}}\cap A_2(\o))+\ldots\\
&+&2 \#Z(\x_{|_{\S^{n-1}(\o)}}\cap A_{n-1}(\o))\\
&+&2 \#Z(\x_{|_{\S^{n-1}(\o)}}\cap A_{n}(\o))\mod2\\
&\equiv&  0\mod2.
\end{array}}$$\end{proof}

{As for the definition of Morin $n$-coframes, the results presented in Sections \ref{s2} and \ref{s3} of this paper also can be naturally adapted to the context of $n$-frames.} In particular, the main theorems that have been used, as the Poincaré-Hopf Theorem and the Szafraniec's characterization, have their respective versions for vector fields.

Finally, we end the paper with a very simple example. Let us verify that Theorem \ref{fukudaparacampos} indeed holds for the Morin $2$-frame $V=(V_1,V_2)$ presented in the Example \ref{ex3}. To do this, it is enough to see that the torus $\emph{T}$ is a compact manifold with $\chi(\emph{T})=0$. Moreover, $\overline{A_1(V)}=\S^1(V)$ is given by two circles in $\R^3$ and $\overline{A_2(V)}$ consists of four points, so that $\chi(\overline{A_1(V)})=0$ and $\chi(\overline{A_2(V)})=4$. Therefore, $$\chi(\emph{T})\equiv \chi(\overline{A_1(V)})+\chi(\overline{A_2(V)}) \mod 2.$$

%\bibliographystyle{plain}
%\bibliography{refbiblio}

\begin{thebibliography}{10}

\bibitem{Ando}
Yoshifumi Ando.
\newblock On local structures of the singularities {$A_k\;D_k$} and {$E_k$} of
  smooth maps.
\newblock {\em Trans. Amer. Math. Soc.}, 331(2):639--651, 1992.

\bibitem{Dutertrefukui}
Nicolas Dutertre and Toshizumi Fukui.
\newblock On the topology of stable maps.
\newblock {\em J. Math. Soc. Japan}, 66(1):161--203, 2014.

\bibitem{Eliasberg}
Ja~M \`Elia\v{s}berg.
\newblock On singularities of folding type.
\newblock {\em Mathematics of the USSR-Izvestiya}, 4(5):1119, 1970.

\bibitem{Fukuda}
Takuo Fukuda.
\newblock Topology of folds, cusps and {M}orin singularities.
\newblock In {\em A f\^ete of topology}, pages 331--353. Academic Press,
  Boston, MA, 1988.

\bibitem{InaIshiKawaThang}
Kazumasa {Inaba}, Masaharu {Ishikawa}, Masayuki {Kawashima}, and Tat {Thang
  Nguyen}.
\newblock {On linear deformations of Brieskorn singularities of two variables
  into generic maps}.
\newblock {\em ArXiv e-prints}, November 2014.

\bibitem{KalmarTerpai}
Boldizs{\'a}r Kalm{\'a}r and Tam{\'a}s Terpai.
\newblock Characteristic classes and existence of singular maps.
\newblock {\em Trans. Amer. Math. Soc.}, 364(7):3751--3779, 2012.

\bibitem{Morin2}
Bernard Morin.
\newblock Formes canoniques des singularit\'{e}s d'une application
  diff\'{e}rentiable.
\newblock {\em Comptes Rendus Hebdomadaires des s\'{e}ances de l'Acad\'{e}mie
  des Sciences}, 260(25):6503--6506, 1965.

\bibitem{Quine}
J.~R. Quine.
\newblock A global theorem for singularities of maps between oriented
  {$2$}-manifolds.
\newblock {\em Trans. Amer. Math. Soc.}, 236:307--314, 1978.

\bibitem{Ruiz1}
Camila~M. {Ruiz}.
\newblock {A new proof of a theorem of Dutertre and Fukui on Morin
  singularities}.
\newblock {\em ArXiv e-prints}, February 2016.

\bibitem{Saeki}
Osamu Saeki.
\newblock Studying the topology of {M}orin singularities from a global
  viewpoint.
\newblock {\em Math. Proc. Cambridge Philos. Soc.}, 117(2):223--235, 1995.

\bibitem{SaekiSakuma}
Osamu Saeki and Kazuhiro Sakuma.
\newblock Maps with only {M}orin singularities and the {H}opf invariant one
  problem.
\newblock {\em Math. Proc. Cambridge Philos. Soc.}, 124(3):501--511, 1998.

\bibitem{Saji1}
Kentaro {Saji}.
\newblock {Criteria for Morin singularities into higher dimensions}.
\newblock {\em ArXiv e-prints}, December 2014.

\bibitem{Saji3}
Kentaro {Saji}.
\newblock {Criteria for Morin singularities for maps into lower dimensions, and
  applications}.
\newblock {\em ArXiv e-prints}, October 2015.

\bibitem{Saji2}
Kentaro {Saji}.
\newblock {Isotopy of Morin singularities}.
\newblock {\em ArXiv e-prints}, October 2015.

\bibitem{SzaboSzucsTerpai}
Endre Szab{\'o}, Andr{\'a}s Sz{\H{u}}cs, and Tam{\'a}s Terpai.
\newblock On bordism and cobordism groups of {M}orin maps.
\newblock {\em J. Singul.}, 1:134--145, 2010.

\bibitem{sza2}
Zbigniew Szafraniec.
\newblock The {E}uler characteristic of algebraic complete intersections.
\newblock {\em J. Reine Angew. Math.}, 397:194--201, 1989.

\bibitem{sza1}
Zbigniew Szafraniec.
\newblock A formula for the {E}uler characteristic of a real algebraic
  manifold.
\newblock {\em Manuscripta Math.}, 85(3-4):345--360, 1994.

\bibitem{Szucs}
Andr{\'a}s Sz{\"u}cs.
\newblock On the cobordism groups of cooriented, codimension one {M}orin maps.
\newblock {\em J. Singul.}, 4:196--205, 2012.

\end{thebibliography}

\end{document}